\documentclass[final,twoside,reqno]{amsbook}

\usepackage{tikz}
\usepackage[centertags]{amsmath}
\usepackage{amssymb,upref}
\usepackage{xy} \xyoption{all}

\usepackage{lmpidx} 
\usepackage{lmpidxdummy} 

\usepackage{lmp-macro}
\usepackage{LesMosPflfigures}

\usepackage{euler}
 
\usepackage[backref=page,colorlinks=true,linkcolor=blue,citecolor=blue]{hyperref}

\begin{document}
\frontmatter
\title[Connes--Chern character and eta cochains]{Connes--Chern character
for manifolds with boundary and eta cochains}
\author{Matthias Lesch}
\thanks{M.~L.~was partially supported by the 
        Hausdorff Center for Mathematics and by the 
        Sonder\-forschungs\-be\-reich/\-Transregio 12 
        ``\uppercase{S}ymmetries and \uppercase{U}niversality in 
        \uppercase{M}esoscopic \uppercase{S}ystems''
        (\uppercase{B}ochum--\uppercase{D}uisburg/\uppercase{E}ssen--\uppercase{K}\"oln--\uppercase{W}arszawa)}

\address{Mathematisches Institut,
Universit\"at Bonn,
Endenicher Allee 60,
53115 Bonn,
Germany}

\email{ml@matthiaslesch.de, lesch@math.uni-bonn.de}
\urladdr{www.matthiaslesch.de, www.math.uni-bonn.de/people/lesch}
\author{Henri Moscovici}
\address{Department of Mathematics,
The Ohio State University,
Columbus, OH 43210,
USA}
\email{henri@math.ohio-state.edu}

\thanks{H. M. was partially
 supported by the US National Science Foundation
    awards no. DMS-0245481 and no. DMS-0652167} 
\author{Markus J. Pflaum}
\address{Department of Mathematics,
University of Colorado UCB 395,
Boulder, CO 80309,
USA} 
\email{markus.pflaum@colorado.edu}
\urladdr{http://euclid.colorado.edu/$\sim$pflaum}

\thanks{M. J. P. was partially supported by the DFG}

\thanks{M. L. would like to thank the Department of Mathematics of the
University of Colorado for hospitality and support.}
\thanks{Both H. M. and M. J. P.  acknowledge with thanks the 
hospitality and support of the Hausdorff Center for Mathematics. 
 }

\date{}

\subjclass[2000]{Primary 58Jxx, 46L80; Secondary 58B34, 46L87}

\keywords{Manifolds with boundary, \textup{b}-calculus, noncommutative geometry, Connes--Chern character, relative cyclic cohomology,
$\eta$-invariant}


\begin{abstract}
We express the {\CoChch} of the Dirac operator
associated to a \textup{b}-metric on a manifold
with boundary in terms of a retracted cocycle in relative cyclic cohomology,
whose expression depends on a scaling/cut-off
parameter. Blowing-up the metric one recovers the pair of characteristic currents
that represent the
corresponding de Rham relative homology class, while the blow-down
yields a relative cocycle whose
expression involves higher eta cochains and 
their \textup{b}-analogues.
The corresponding pairing formul\ae\,  with relative
K-theory classes capture information about the boundary and allow to
derive geometric consequences. As a by-product, we show that the
generalized Atiyah-Patodi-Singer pairing introduced by Getzler and Wu 
is necessarily restricted to almost flat bundles.

\end{abstract}

\maketitle
\setcounter{page}{4}

\tableofcontents
\listoffigures

%
\newcommand{\calE}{\mathcal E}
\newcommand{\calJ}{\mathcal J}
\newcommand{\sfD}{\mathsf D}
\newcommand{\tb}[1]{^\textup{b}\!#1}

\newcommand{\diracbdy}{\mathsf{D}_\pl}
\newcommand{\diracbdyvar}[1]{\mathsf{D}_{\pl,#1}}
\newcommand{\diracbdyt}{\mathsf{D}_{\pl,t}}

\newcommand{\btch}{\,^\textup{b}\hspace*{-0.2em}\operatorname{\widetilde{ch}}}

\mainmatter
\chapter*{Introduction}
Let $M$ be a compact smooth $m$-dimensional manifold with boundary 
$\partial M \neq \emptyset$. 
Assuming that $M$ possesses a $Spin^c$ structure, the fundamental class
in the relative $K$-homology group $K_m (M, \pl M)$ can be realized analytically
in terms of the Dirac operator $\dirac$ (graded if $m$ is even) associated
to a riemannian metric on $M$. More precisely,
according to~\cite[\S 3]{BauDouTay:CRC}, if
$\dirac_e$ is a closed extension of $\dirac$ satisfying the condition 
that either $\dirac^*_e \dirac_e$ or $\dirac_e \dirac^*_e$
 has compact resolvent ({\em e.g.} both $\dirac_{\rm min}$
and $\dirac_{\rm max}$ are such), then the bounded operator 
$F = \dirac_e (\dirac^*_e \dirac_e +1)^{-1/2}$ defines a relative Fredholm module over the
pair of $C^*$-algebras $\bigl(\cC(M), \cC(\pl M) \bigr)$, hence an element 
$[\dirac] \in K_m (M, \pl M)$. Moreover, by~\cite[\S 4]{BauDouTay:CRC},  
the connecting homomorphism 
maps $[\dirac]$ to the fundamental class $[\diracbdy] \in K_{m -1} (\partial M)$ 
corresponding to the Dirac operator $\diracbdy$ 
associated to the boundary restriction of the metric and of $Spin^c$ structure.

\nind{Index@$\Index_{[{\mathsf D}]}$}
The map $\Index_{[\dirac]} :K^m (M, \pl M) \rightarrow \Z$, 
defined by the pairing of $K$-theory with the $K$-homology class of $[\dirac]$, can be expressed
in cohomological terms by means of Connes' 
Chern character with values in
cyclic cohomology~\cite{Con:NDG}. Indeed, the relative $K$-homology group
$ \, K_m (M , \pl M)$, viewed as the Kasparov group
$\, KK_m \bigl(\cC_0 (M\setminus \partial M) ; \C\bigr)$,
can be realized as homotopy classes of Fredholm modules over the Fr\'echet algebra
$ \calJ^\infty (\partial M,M) $ 
of smooth functions on $M$ vanishing to any order on $\partial M$;
$ \calJ^\infty (\partial M,M) $ is a
 local $C^*$-algebra, $H$-unital and dense in 
$\cC_0 (M\setminus \partial M) = \{ f \in \cC (M) \mid f_{|\partial M}=0\}$.
One can therefore define the {\CoChch} of $[\dirac]$ 
by restricting the operator $F= \dirac_e (\dirac^*_e \dirac_e +1)^{-1/2}$, 
or directly $\dirac$, to $\cJ^\infty  (\partial M, M)$ and regarding 
it as a finitely summable Fredholm module.
The resulting periodic cyclic cocycle corresponds,
via the canonical isomorphism between the periodic cyclic cohomology
$\, HP^{\textrm{ ev/odd}} \bigl(\cJ^\infty (\partial M, M)\bigr)$ and
the \deRham\ homology $\, H_{\rm ev/odd}^{\rm dR}
(M  \setminus \partial M ; \C)$ ({\it cf.}~\cite{BraPfl:HAWFSS}),
to the \deRham\ class of the current (with arbitrary support)
associated to the $\hat{A}$-form of the riemannian metric. In fact, one can even recover the 
$\hat{A}$-form itself out of local cocycle representatives for the  {\CoChch},
as in~\cite[Remark 4, p. 119]{ConMos:TCC} or 
\cite[Remark II.1, p. 231]{ConMos:LIF}.
However, the boundary $\pl M$ remains conspicuously absent in such
representations.
\sind{de Rham!homology}

It is the purpose of this paper to provide cocycle representatives for 
the {\CoChch} of the fundamental $K$-homology
class $[\dirac] \in K_m (M, \partial M)$ 
that capture and reflect geometric information about the boundary. 
Our point of departure is Getzler's construction~\cite{Get:CHA} of 
the {\CoChch} of $[\dirac]$. Cast
in the propitious setting of 
\textnm{Melrose}'s \textup{b}-calculus~\cite{Mel:APSIT}, 
Getzler's cocycle has however the disadvantage,
from the viewpoint of its geometric functionality, of
being realized not in the relative 
cyclic cohomology complex proper but in its entire extension. Entire
cyclic cohomology~\cite{Con:ECC}
was devised primarily for handling infinite dimensional geometries and
is less effective as a tool than ordinary cyclic cohomology when dealing with
finite dimensional
$K$-homology cycles. To remedy this drawback we undertook the task of producing
cocycle realizations for the {\CoChch} directly in the 
relative cyclic cohomology complex associated to the pair of algebras
$\bigl( \cC^\infty (M), \cC^\infty (\partial M)\bigr)$. This is achieved by adapting and
implementing in the context of relative cyclic cohomology the
retraction procedure of~\cite{ConMos:TCC}, which converts
the entire {\CoChch} into the periodic one.  The resulting
cocycles automatically carry information about the boundary and this allows
to derive geometric consequences.  
\medskip

It should be mentioned that the relative point of view in the framework
of cyclic cohomology was first exploited in \cite{LesMosPfl:RPC}
to obtain cohomological expressions for $K$-theory invariants associated
to parametric pseudodifferential operators.
It was subsequently employed by Moriyoshi and Piazza to establish a
 Godbillon-Vey index pairing for longitudinal
Dirac operators on foliated bundles \cite{MorPia:RPA}.

\medskip

Here is a quick synopsis of the main results of the present paper.  
Throughout the paper, we
fix an exact \textup{b}-metric $g$ on $M$, and denote by $\dirac$ the corresponding 
\textup{b}-Dirac operator.
We define for each $t > 0$ and any $n \ge m = \dim M$, $n \equiv m$ (mod $2$), 
pairs of cochains 
 \begin{equation}\label{eq:relCochain}
 \bigl(\bch^n_t(\dirac),   \ch^{n+1}_t (\diracbdy) \bigr)\quad
 \text{resp.}\quad \bigl(\btch^n_t(\dirac), \ch^{n-1}_t(\diracbdy) \bigr)
 \end{equation} 
 over the pair of algebras $ \bigl( \cC^\infty (M), \cC^\infty (\partial M)\bigr)$, given 
 by the following expressions:
\begin{equation} \label{cht}
\begin{split}
  \bch^n_t(\dirac)\, &:=\,
   \sum_{j\geq 0} \bCh^{n-2j} (t \dirac) +
  B \bTslch^{n+1}_t (\dirac),\\ 
  \ch^{n+1}_t (\diracbdy)\, &:= \,
  \sum_{j\geq 0} \Ch^{n-2j+1} (t \diracbdy) +
  B \Tslch^{n+2}_t (\diracbdy),\\
  \btch^n_t(\dirac) \, &:=\,
    \bch^n_t(\dirac) +\Tslch^n_t(\diracbdy)\circ i^* .
  \end{split}
\end{equation}
In these formul\ae, $\,\Ch^{\bullet} ( \diracbdy)$ stand for the components
of the \textnm{Jaffe-Lesniewski-Osterwalder} cocycle~\cite{JLO:QKT}
representing the {\CoChch} in entire cyclic cohomology, 
$\bCh^{\bullet} ( \dirac)$ denote the corresponding \textup{b}-analogue
({\it cf.} \eqref{Eq:DefChern}), while the cochains  
$\Tslch^{\bullet}_t (\diracbdy)$, resp.  $\bTslch^{\bullet}_t ( \dirac) $
(see \eqref{eq:ML200909263},
are manufactured out of the canonical transgression formula
as in \cite{ConMos:TCC};
$\, i : \partial M \rightarrow M$ denotes the inclusion.
One checks that
\begin{equation} \label{relcohochern}
	\begin{split}
  (b + B) \bigl(\bch^n_t  (\dirac) \bigr) \, &=\,
  \ch^{n+1}_t (\diracbdy) \circ i^\ast\\
   (b + B) \bigl(\btch^n_t  (\dirac) \bigr) \, &=\,
  \ch^{n-1}_t (\diracbdy) \circ i^\ast,
 \end{split}
\end{equation} 
which shows that the cochains \eqref{eq:relCochain} are cocycles in the
\emph{relative} 
 total $(b, B)$-complex of
 $ (\cC^\infty (M), \cC^\infty (\partial M) ) $. Moreover, 
 the class of this cocycle in periodic cyclic cohomology is independent
 of $t > 0$ and of $n = m + 2k, \, k \in \Z_+$. Furthermore,
 its limit as $t \rightarrow 0$ gives the pair of 
$\hat A$ currents corresponding to the \textup{b}-manifold M, that is 
\begin{equation} \label{eq: 0lim cocycle}
  \big(\lim_{t\searrow 0} \bch^n_t (\dirac), \, \lim_{t\searrow 0} 
  \ch^{n+1}_t (\diracbdy) \big)
  = \left(\int_{\tb M} \hat{A}(\bnabla^2_g) \wedge \bullet , \,
  \int_{\pl M}\hat{A}(\nabla^2_{g_\partial}) \wedge \bullet\right) \, .
\end{equation}
The notation in the right hand side requires an explanation. With
\begin{equation}  \label{eq: A-currents}
  \hat{A}(\bnabla^2_g) = \det \left(\frac{\bnabla_g^2/4 \pi i}{\sinh \bnabla_g^2/4 \pi i}\right)^{\frac{1}{2}} ,
\quad \hat{A}(\nabla^2_{g_\partial}) = 
 \det \left(\frac{\nabla_{g_\partial}^2/4 \pi i}{\sinh \nabla_{g_\partial}^2/4 \pi i}\right)^{\frac{1}{2}} ,
\end{equation}
and $\displaystyle \int_{\tb M} : \bOmega^m (M) \rightarrow \C$ denoting
the \textup{b}-integral of \textup{b}-differential $m$-forms on $M$ associated to the 
trivialization of the normal bundle to $\pl M$ underlying the \textup{b}-structure,
both terms in the right hand side of \eqref{eq: 0lim cocycle} are viewed as
$(b, B)$-cochains associated to currents. More precisely, incorporating
the $2 \pi i$ factors which account for the conversion of the Chern character
in cyclic homology into the Chern character in \deRham\ cohomology,  
for $M$ even dimensional this identification takes the form
\begin{multline}\label{eq:bA-hat cochain}
 \int_{\tb M} \hat{A}(\bnabla^2_g) \wedge \big(f_0, \ldots , f_{2q}\big)\\ 
 = \frac{1}{(2 \pi i)^q (2q)!}
 \int_{\tb M} \hat{A}(\bnabla^2_g) \wedge f_0\, df_1\wedge \ldots \wedge f_{2q}  \, ,  
\end{multline}
respectively
\begin{multline}\label{eq:Abdry-hat cochain}
  \int_{\partial M} \hat{A}(\nabla^2_{g_\partial}) \wedge \big(h_0, \ldots , h_{2q-1}\big) \\
  =   \frac{1}{(2 \pi i)^{q} (2q-1)!}
 \int_{\partial M} \hat{A}(\nabla^2_{g_\partial}) \wedge h_0\, dh_1\wedge \ldots \wedge h_{2q-1} .
\end{multline}

Finally  (\textit{cf.}~Theorem \ref{t: CC-character}  \textit{infra}),
the limit formula \eqref{eq: 0lim cocycle} implies that both
$\big(\bch^n_t (\dirac),\ch^{n+1}_t (\diracbdy) \big)$ and $\big(\btch^n_t (\dirac),\ch^{n-1}_t (\diracbdy) \big)$
represent the Chern character of the fundamental relative $K$-homology class  $[\dirac] \in K_m(M,\partial M)$.
%
%

Under the assumption that $\Ker \diracbdy = 0$, we next prove (Theorem
\ref{t:ML20081215} \textit{infra})
that the pair of retracted 
cochains $\big(\btch^n_t (\dirac),\ch^{n-1}_t (\diracbdy) \big)$ has a limit 
as $t\to\infty$. For $n$ even, or equivalently $M$ even dimensional,
this limit has the expression
\begin{equation}\label{eq:ML20081215-1-intro}
\begin{split}
   \btch_\infty^n(\dirac)&=\sum_{j=0}^{n/2} \kappa^{2j} (\dirac)+
   B\bTslch_\infty^{n+1}(\dirac)+\Tslch_\infty^n(\diracbdy)\circ i^* , \\
   \ch_\infty^{n-1}(\diracbdy)&=          B\Tslch_\infty^{n}(\diracbdy) ,
\end{split}
\end{equation}
with the cochains $\kappa^{\bullet}(\dirac)$, occurring only when $ \Ker\dirac\neq\{0\}$, 
given by
\begin{equation} \label{eq: Ker Str}
    \kappa^{2j} (\dirac)(a_0, \ldots ,a_{2j})=\,
    \Str \bigl(\varrho_H(a_0)\, \go_H(a_1,a_2) \cdots \go_H(a_{2j-1},a_{2j})\bigr) ;
 \end{equation} 
here $H$ denotes the orthogonal projection onto $ \Ker \dirac$, and 
\begin{align*}
            \varrho_H(a):=HaH\, , \quad
             \go_H(a,b):= \varrho_H(ab)-\varrho_H(a)\varrho_H(b) ,
             \text{ for all $a, b \in \cC^\infty (M)$} .
\end{align*}
When $M$ is odd dimensional, the limit cocycle takes the 
similar form
\begin{equation}\label{eq:ML20081215-2-intro}
\begin{split}
    \btch_\infty^n(\dirac)&=B\bTslch_\infty^{n+1}(\dirac)+ \Tslch_\infty^n(\diracbdy)\circ i^*\, , \\
         \ch_\infty^{n-1}(\diracbdy)&= B\Tslch_\infty^n(\diracbdy).
         \end{split}
\end{equation}
The absence of cochains of the form $ \kappa^{\bullet} (\diracbdy)$ 
in the boundary component is due to the assumption that $\Ker \diracbdy = 0$.

\medskip

The geometric implications become apparent when one inspects
 the ensuing pairing with
$K$-theory classes.  For $M$ even dimensional,
 a class in $K^m (M, \pl M)$ can be represented
as a triple $[E,F, h]$, where $E$, $F$ are vector bundles over $M$, which we will 
identify with projections $p_E, p_F \in \Mat_N (\cC^\infty (M))$, 
and $h :[0, 1] \rightarrow \Mat_N (\cC^\infty (\partial M))$ 
is a smooth path of projections connecting their restrictions to the boundary
$E_\pl $ and $F_\pl $. For $M$ odd dimensional, a representative of a class in 
$K^m (M, \pl M)$ is a triple $(U,V,h)$, where $U,V:M\to U(N)$ are unitaries and $h$ 
is a homotopy between their restrictions to the boundary $U_{\pl}$ and $V_{\pl}$. 
In both cases, the Chern character of $[X, Y , h] \in K^m (M, \partial M)$ is 
represented by the relative cyclic homology cycle over the algebras
$\bigl( \cC^\infty (M), \cC^\infty (\partial M)\bigr)$
%
%
\begin{equation} 
  \ch_\bullet  \big( [X, Y, h]  \big) \, = \,
  \Big(
    \ch_\bullet(Y) - \ch_\bullet (X) \, , \, - \Tslch_\bullet (h)   \Big) \, ,  
\end{equation}
where $ \ch_\bullet $, resp. $\Tslch_\bullet $ denote the components of
the standard Chern character in cyclic homology resp.~of its canonical
transgression (see Section \ref{s:TheChernCharacter}).  

The pairing $\big\langle [\dirac], \, [X, Y,h] \big\rangle\in\Z$
between the classes $[\dirac] \in K_m(M,\partial M)$ and
$[X, Y , h] \in K^m (M, \partial M)$ acquires the cohomological expression
\begin{equation}\label{eq:relpair even/odd}
\begin{split}
\big\langle [\dirac], \, [X,& Y,h] \big\rangle \, =\,
\big\langle \big( \btch_t^n (\dirac), \ch_t^{n-1} (\diracbdy) \big), \,
  \ch_\bullet [X, Y,h]  \big\rangle =  \\
   = \big\langle& \sum_{j\ge 0} \bCh^{n-2j} (t\dirac) + B\, \bTslch_t^{n+1} (\dirac), 
  \,  \ch_\bullet (Y) -\ch_\bullet (X)  \big\rangle  \\
  &+\big\langle \Tslch_t^n(\diracbdy),\ch_\bullet(Y_\pl)-\ch_\bullet(X_\pl)\big\rangle\\
  &- \big\langle  \sum_{j\ge 0}  \Ch^{n-2j-1} (t\diracbdy) +
  B \Tslch_t^{n} (\diracbdy), \,
  \Tslch_\bullet (h)   \big\rangle ,
\end{split}
\end{equation}
which holds for any $t > 0$.
Letting $t \rightarrow 0$ yields the local form of the pairing formula
\begin{equation} \label{eq: pairing lim0}
\begin{split}
\big\langle& [\dirac], \, [X,Y, h] \big\rangle \, =\,
 \\
   &= \int_{\tb M} \hat{A}(\bnabla^2_g) \wedge \big( \ch_\bullet (Y) - \ch_\bullet (X) \big)  
- \int_{\partial M} \hat{A}(\nabla^2_{g_\partial})  \wedge  \Tslch_\bullet (h) .
   \end{split}
\end{equation}
It should be pointed out that \eqref{eq: pairing lim0} holds in complete generality, without 
requiring the invertibility of $\diracbdy$.

 When $M$ is even dimensional
and $\diracbdy$ is invertible, the equality between the above limit and 
the limit as $t \rightarrow \infty$ yields, for any  $n =2\ell \geq m $, the 
identity 
\begin{equation}  \label{eq:APS-ind}
\begin{split}
 \sum_{0\leq k \leq \ell} &
 \big\langle \kappa^{2k}(\dirac)  , \ch_{2k}(p_F) - \ch_{2k}(p_E)\big\rangle +
 \big\langle B \bTslch_\infty^{n+1} (\dirac) , \ch_n(p_F) - \ch_n(p_E) \big\rangle \, +\\
  &\quad +
  \big\langle \Tslch_\infty^n(\diracbdy), \ch_n (p_{F_\pl})-
  \ch_n (p_{E_\pl})\big\rangle \,=\\
  &= \int_{\tb M} \hat{A}(\bnabla^2_g) \wedge 
  \big( \ch_\bullet (p_F) - \ch_\bullet (p_E) \big)  
  - \int_{\partial M} \hat{A}(\nabla^2_{g_\partial})  \wedge  \Tslch_\bullet (h) \\
  &\qquad\quad +\big\langle B \Tslch_\infty^{n} (\diracbdy),  
  \Tslch_{n-1}  (h) \big\rangle ,
\end{split}
\end{equation} 
where 
\begin{displaymath}
  \begin{split}
  \ch_{2k} (p) = & 
  \begin{cases}
    \tr_0 (p), & \text{for $k=0$}, \\
    (-1)^{k} \frac{(2k)!}{k!} \, \tr_{2k} 
    \big( (p-\frac 12) \otimes p^{\otimes 2k}\big), & \text{for $k>0$} .
  \end{cases}  
  \end{split}
\end{displaymath}
Like the Atiyah-Patodi-Singer index formula~\cite{APS:SARI}, the 
equation \eqref{eq:APS-ind}
involves index and eta cochains, only of higher order.
Moreover, the same type of 
identity continues to hold in the odd dimensional case.
Explicitly, it takes the form
\begin{equation}  
\begin{split}
 (-1)^{\frac{n-1}{2}} & \, \textstyle{\big( \frac{n-1}{2} \big)!}\, 
 \Big(\big\langle B \bTslch_\infty^{n+1} (\dirac) , 
 (V^{-1} \otimes V)^{\otimes \frac{n+1}{2}} - 
 (U^{-1} \otimes U)^{\otimes \frac{n+1}{2}}  \big\rangle \, +\\
 &\qquad \big\langle \Tslch_\infty^n(\diracbdy), 
 (V^{-1}_\pl \otimes V_\pl)^{\otimes \frac{n+1}{2}} - 
 (U^{-1}_\pl \otimes U_\pl)^{\otimes \frac{n+1}{2}} \big\rangle \Big) \,=\\
 &= \int_{\tb M}\hat{A}(\bnabla^2_g)\wedge\big( \ch_\bullet (V) - \ch_\bullet (U)\big)  
- \int_{\partial M} \hat{A}(\nabla^2_{g_\partial})  \wedge  \Tslch_\bullet (h) \\
 &\qquad +\big\langle B \Tslch_\infty^{n} (\diracbdy),  \Tslch_{n-1}  (h) \big\rangle.
\end{split}
\end{equation} 
\medskip

The relationship between the relative pairing and the Atiyah-Patodi-Singer index theorem
can actually be made explicit, and leads to interesting geometric consequences. Indeed,
under the necessary assumption that $M$ is even dimensional, we show
({\it cf.} Theorem \ref{t:RelativePairingKTheory}) that
the above pairing can be expressed as follows:
\begin{equation} 
\label{Eq:Pairing1}
\begin{split}
\langle [\dirac], [E, F, h] \rangle = \indAPS \dirac^F - \indAPS \dirac^E +
  \SF (h , \dirac_\partial) ;
\end{split}
\end{equation}
here $\indAPS$ stands for
the {\APS}-index, and $ \SF (h , \dirac_\partial) $  denotes the spectral
flow along the path of operators  $\bigl(\diracbdy^+\bigr)^{h(s)}$;
$\diracbdy^+$ is the restriction of $\sfc(dx)^{-1}\diracbdy$ to the positive
half spinor bundle and $\sfc(dx)$ denotes Clifford multiplication by the inward normal vector.
On applying the {\APS}\ index formula~\cite[Eq. (4.3)]{APS:SARI}, the 
pairing takes the explicit form
\begin{equation} 
\label{Eq:Pairing2}
\begin{split}
\langle [\dirac], [E, F, h] \rangle \,= &
 \int_{\tb M} \hat{A}(\bnabla^2_g) \wedge \big( \ch_\bullet (F)- \ch_\bullet (E) \big)  \\
& - \Big( \xi (\diracbdy^{+,F_\partial}) - \xi (\diracbdy^{+,E_\partial})  \Big)
\, + \,  \SF (h, \dirac_\partial) ,
 \end{split}
\end{equation}
where
\begin{equation} \label{xi}
\xi (\diracbdy^{+,E_\partial}) \, = \, 
\frac{1}{2} \Big( \eta (\diracbdy^{+,E_\partial}) \, + \,
  \dim  \Ker \diracbdy^{+,E_\partial}   \Big) .
\end{equation}
Comparing this expression with the local form of the pairing \eqref{eq: pairing lim0}
leads to a generalization 
of the {\APS}\ 
 odd-index formula~\cite[Prop. 6.2, Eq. (6.3)]{APS:SARIII}, from trivialized flat bundles
to pairs of equivalent vector bundles in $K$-theory. Precisely ({\it cf.}
Corollary \ref{t: gen APS flat}), if
$E',F'$ are two such bundles on a closed odd dimensional spin manifold $N$,
and $h$ is the homotopy implementing the equivalence of $E'$ with $F'$, then
\begin{equation} 
\begin{split}
\xi (\dirac_{g'}^{F'}) - \xi (\dirac_{g'}^{E'}) 
\, = \,  \int_{N} \hat{A} (\nabla_{g'}^2) \wedge
   \Tslch_\bullet (h) \, + \,  \SF (h , \dirac_{g'}) \, ,
\end{split}
\end{equation} 
where  $\dirac_{g'}$ denotes the Dirac
operator associated to a riemannian metric $g'$ on $N$; equivalently, 
\begin{equation} 
  \int_0^1  \frac{1}{2} \frac{d}{dt} \big( \eta (p_{h(t)} \, \dirac_{g'} \, p_{h(t)}) \big) dt 
     \, = \, \int_{N} \hat{A} (\nabla_{g'}^2) \wedge \Tslch_\bullet (h) \, ,  
\end{equation}
where $p_{h(t)}$ is the path of projections joining $E'$ and $F'$, and 
the left hand side is the natural extension of the
 real-valued index in~\cite[Eq. (6.1)]{APS:SARIII}.

\medskip

Let us briefly comment on the main analytical challenges encountered in
the course of proving the results outlined above.  In order to compute
the limit as $t\searrow 0$ of the Chern character, one needs 
to understand the asymptotic behavior of expressions of the form
\begin{equation}
\begin{split}
\blangle& A_0,A_1,\ldots,A_k\rangle_{\sqrt{t}\dirac}\\
  &:=\int_{\Delta_k} 
\bTr\bigl( A_0 e^{-\sigma_0 t\dirac^2} A_1 e^{-\sigma_1 t\dirac^2} \ldots A_k e^{-\sigma_k t\dirac^2} \bigr) d\sigma,
\end{split}
\end{equation}
where $\Delta_k$ denotes the standard simplex $\{\sigma_0+\ldots+\sigma_k=1, \sigma_j\ge 0\}$ and $A_0,\ldots,A_k$
are \textup{b}-differential operators of order $d_j, j=0,\ldots,k$;
$d:=\sum_{j=0}^k d_j$ denotes the sum of their orders.
The difficulty here is twofold. Firstly, the \textup{b}-trace is a 
\emph{regularized} extension of the trace to \textup{b}-pseudodifferential operators on the
\emph{non-compact} manifold $M\setminus \partial M$ (recall that the \textup{b}-metric
 degenerates at $\partial M$). 
Secondly, the expression inside the \textup{b}-trace involves a product of 
operators. The Schwartz kernel
of the product
$A_0 e^{-\sigma_0 t\dirac^2}A_1 e^{-\sigma_1 t\dirac^2}\cdot \ldots\cdot A_k e^{-\sigma_k t \dirac^2}$
does admit a \emph{pointwise} asymptotic expansion  
(see  \cite{Wid:STC},  \cite{ConMos:CCN}, \cite{BloFox:APO}),
namely
\begin{equation}
\begin{split}
    \Bigl(A_0 &e^{-\sigma_0 t\dirac^2} A_1 e^{-\sigma_1 t\dirac^2} \cdot\ldots\cdot A_k e^{-\sigma_k t \dirac^2}\Bigr)(p,p) \\
        &=: \sum_{j=0}^{n} a_j(A_0,\ldots,A_k,\dirac)(p)\; t^{\frac{j-\dim M
        -d}{2}}+O_p(t^{(n+1-d-\dim M)/2}).
\end{split}
\end{equation} 
However, this asymptotic expansion is 
only \emph{locally} uniform in $p$; it is not \emph{globally} uniform on the non-compact manifold $M\setminus \partial M$.
A further complication
 arises from the fact that the function $a_j(A_0,\ldots,A_k,\dirac)$ is not necessarily integrable. 
 Nevertheless, a \emph{partie finie}-type regularized integral, 
which we denote by $\int_{\tb M} a_j(A_0,\ldots,A_k,\dirac)d\vol $, does exist
and we prove (\emph{cf.} Theorem \ref{t:JLObCommutatorAsymptotic})
that the corresponding \textup{b}-trace admits an asymptotic expansion of the
form
\begin{equation}
\begin{split}
    \bTr\Bigl(A_0 &e^{-\sigma_0 t\dirac^2} A_1 e^{-\sigma_1 t\dirac^2} \ldots A_k e^{-\sigma_k t \dirac^2}\Bigr) \\
        &= \sum_{j=0}^{n} \int_{\tb M} a_j(A_0,\ldots,A_k,\dirac)d\vol \;
        t^{\frac{j-\dim M -d}{2}}+\\
              &\qquad +O\Bigl(\bigl(\prod_{j=1}^k
              \sigma_j^{-d_j/2}\bigr)t^{(n+1-d-\dim M)/2}\Bigr).
\end{split}
\end{equation} 
When $\,\diracbdy$ is invertible and hence $\dirac$ is a Fredholm operator, we can also prove
the following estimate (\emph{cf.} \eqref{eq:integrated-multiple-btrace-estimate-large})
\begin{equation}
    \begin{split} 
    |\blangle A_0&(I-H),...,A_k(I-H)\rangle_{\sqrt{t}\dirac}|\\
                &\le \tilde C_{\delta,\eps} 
      t^{-d/2-(\dim M)/2-\eps} e^{-t\delta},\quad \text{for \emph{all }} 0<t<\infty  ,
    \end{split}
\end{equation}
 for any $\eps>0$ and any $0<\delta<\inf\specess \dirac^2$. Here, $\specess$
 denotes the essential spectrum and $H$ is the orthogonal projection onto $ \Ker\dirac$. This estimate allows us
to compute the limit as $t \nearrow \infty$ and thus 
derive the formul\ae\, 
\eqref{eq:ML20081215-1-intro} and \eqref{eq:ML20081215-2-intro}.
\medskip

A few words about the organization of the paper are now in order. 
We start by recalling, in Chapter \ref{s:preliminaries},
some basic material on relative cyclic cohomology~\cite{LesMosPfl:RPC},
\textup{b}-calculus~\cite{Mel:APSIT} and Dirac operators. 
In Section \ref{s:btrace} we discuss in detail the \textup{b}-trace
in the context of a manifold with cylindrical ends. 

As a quick illustration of the usefulness of the relative cyclic cohomological
approach in the present context,
we digress in Section \ref{s:McK-S} to establish an
analogue of the well-known McKean--Singer formula for
manifolds with boundary; we
then employ it to recast in these terms
Melrose's proof of the Atiyah-Patodi-Singer index theorem
(\emph{cf.}~\cite[Introduction]{Mel:APSIT}).

 In Section
\ref{s: b-trace formula}, refining an observation due to Loya~\cite{Loy:DOB},
we give an effective formula for the \textup{b}-trace, 
which will turn out to be a convenient technical device.

After setting up the notation related to \textup{b}-Clifford modules and \textup{b}-Dirac
operators in Section \ref{s:b-Clifford-Dirac}, we revisit in the remainder of the chapter
Getzler's version of the relative entire
{\CoChch} in the setting of \textup{b}-calculus~\cite{Get:CHA}.

In Chapter \ref{s:cylinder-estimates} we prove some crucial estimates for the
heat kernel of a \textup{b}-Dirac operator, which are then applied
in Section \ref{estim Chern} to analyze
the short and long time behavior of the components of the \textup{b}-analogue of
the entire Chern character. As a preparation more standard resolvent and heat kernel estimates are
discussed in Section \ref{sec:basic-estimates}.

The final Chapter \ref{chap:Main} contains our main results: 
Section \ref{s: heat expansion} is devoted to asymptotic expansions for the
\textup{b}-analogues of the \textnm{Jaffe-Lesniewski-Osterwalder} components. 
The retracted relative cocycle representing
the {\CoChch} in relative cyclic cohomology is
constructed in Section \ref{s:retracted-relative-cocycle}, where we also compute 
the expressions of its small and large scale limits.

Finally, Section \ref{s: geom pairing} derives the ensuing pairing formul\ae \,
with the $K$-theory, establishes the connection with the 
Atiyah-Patodi-Singer index theorem, and discusses the
geometric consequences. 
The paper concludes with an explanatory note  (Section \ref{s:conclude})
 elucidating the relationship between the results presented here
and the prior work in this direction by 
Getzler~\cite{Get:CHA} and Wu~\cite{Wu:CCC}, and clarifying why their
generalized {\APS}  pairing is necessarily restricted to almost flat bundles.

\aufm{Matthias Lesch, Henri Moscovici, Markus J. Pflaum}
\newcommand{\comment}[1]{\relax}

%
%
\chapter{Preliminaries}
\label{s:preliminaries}
We start by recalling some basic material concerning relative cyclic cohomology,
the Chern character and Dirac operators. Furthermore, 
for the convenience of the reader we provide in Sections 
\ref{App:bdefbmet}--\ref{s:IndFam} a
quick synopsis of the fundamentals of the \textup{b}-calculus for manifolds with 
boundaries due to \textnm{Melrose}. 
For further details we refer the reader to the monograph 
\cite{Mel:APSIT} and the article \cite{Loy:DOB}. 

\section{The general setup}
\label{SubSec:Setup}
Associated to a compact smooth manifold $M$ with boundary $\pl M$, there is 
a commutative diagram of Fr\'echet algebras
with exact rows
%
\nind{J@$\cJ$}
\nind{J@$\cJ^\infty$}
\nind{E@$\cE^\infty(\partial M, M)$}
\begin{equation}
  \label{Dia:BasicShExSeq}
  \xymatrix{ 
    0 \ar[r] & \cJ^\infty (\partial M, M) \ar[r]\ar[d] &\cC^\infty (M)
    \ar[r]^{\!\!\!\!\! \varrho} \ar[d]^\id &\cE^\infty (\partial M , M) \ar[r]\ar[d]^{\mbox{ }_{\|\partial M}} & 
    0\\ 0 \ar[r] & \cJ (\partial M, M) \ar[r] & \cC^\infty (M) \ar[r] &
    \cC^\infty ( \partial M) \ar[r] & 0 .
  }
\end{equation}
$\cJ (\partial M, M) \subset \cC^\infty (M)$ is 
the closed ideal of smooth functions on $M$ vanishing on $\partial M$,
$\cJ^\infty (\partial M, M) \subset \cJ (\partial M, M)$ denotes the closed 
ideal of smooth functions on $M$ vanishing up to infinite order on 
$\partial M$, and $\cE^\infty (\partial M , M)$ is the algebra of 
Whitney functions\sind{Whitney functions} over the subset $\partial M \subset M$.
More generally, for every closed subset $X\subset M$ the ideal 
$\cJ^\infty (X , M) \subset \cC^\infty (M)$ is defined as being
\begin{displaymath}
    \cJ^\infty (X , M) :=\big\{ f\in\cC^\infty (M)\mid D f_{|X}=0
    \text{ for every differential operator $D$ on $M$} \big\}.
\end{displaymath}
By Whitney's extension theorem (\emph{cf}.~\cite{Mal:IDF,Tou:IFD}), the algebra 
$\cE^\infty (X,M)$ of Whitney functions over $X \subset M$ is naturally 
isomorphic to the quotient of $\cC^\infty (M)$ by the closed ideal 
$\cJ^\infty (X,M)$; we take this as a definition of  $\cE^\infty (X,M)$. 
The right vertical arrow in diagram 
\eqref{Dia:BasicShExSeq} is given by the map
\[
  \cE^\infty (X,M) \rightarrow \cC^\infty (X), \quad
  F \mapsto F_{\|X} := F + \cJ (X,M) ,
\]
which is a surjection. 

Let us check that the Fr\'echet algebra  $\cJ^\infty := \cJ^\infty (\partial M ,M)$
is a local $C^*$-algebra. First, by the multivariate 
Fa\`a di Bruno formula \cite{ConSav:MFBFA} the unitalization 
$\cJ^{\infty,+}$ of $\cJ^\infty$ is seen to be
closed under holomorphic calculus in the unitalization $J^+$ of 
the algebra $J := \cC_0 (M\setminus \partial M)$. Since $\cJ^{\infty,+}$ is also dense
in $J^+$, it follows that $\cJ^\infty := \cJ^\infty (\partial M ,M)$ is indeed
a local $C^*$-algebra whose $C^*$-closure is the $C^*$-algebra $J$.
Using this together with
\excision\ in $K$-homology (\emph{cf}.~for example \cite{HigRoe:AKH}),
one can easily check that the space 
of equivalence classes of Fredholm modules over $\cJ^\infty$ coincides naturally with the  
$K$-homology of the pair of $C^*$-algebras $\big(\cC (M), \cC(\partial M)\big)$. 
Moreover, by
\cite[p.~298]{Con:NG} one has the following commutative diagram
\begin{equation}
  \label{dia:FredModHP}
  \xymatrix{
  \hspace{-25mm}\Bigl\{\parbox{5cm}{finitely summable Fredholm\\ modules over $\cJ^\infty$}\Bigr\}\ar[r]^{\ch_\bullet} \ar[d]& 
  HP^\bullet (\cJ^\infty ) \ar[d] \\
   K^\bullet (J) = KK_\bullet (J,\C) \ar[r] & \Hom (K_\bullet (J) ,\C) ,
  }
\end{equation}
where the right vertical arrow is given by natural pairing between periodic cyclic cohomology
and $K$-theory, and the lower horizontal arrow by the pairing of $K$-theory 
with $K$-homology via the 
Fredholm index.

 A Dirac, resp. a \textup{b}-{Dirac} operator on $M$ determines 
 a Fredholm module over $\cJ^\infty$ and therefore a class
 in the \Khomology\ of the pair $\big(\cC (M), \cC(\partial M)\big)$.
 In this article, we are concerned with geometric representations of
the  {\CoChch} of such a class and of the ensuing
pairing with the $K$-theory of the pair $\big(\cC (M), \cC(\partial M)\big)$.
\section{Relative cyclic cohomology}
\label{SubSec:RelCycCoh}
As in \cite{LesMosPfl:RPC}, we associate to a short exact sequence 
of Fr\'echet algebras 
\begin{equation}
\label{Eq:ShExSeq}
 0 \longrightarrow \cJ \longrightarrow \cA \overset{\sigma}{\longrightarrow} \cB \longrightarrow 0,
\end{equation}
with $\cA$ and $\cB$ unital, a short exact sequence of mixed complexes 
\begin{equation}
  0 \longrightarrow \big( C^\bullet (\cB), b, B\big) \longrightarrow \big( C^\bullet (\cA), b , B\big) 
  \longrightarrow \big( Q^\bullet , b , B\big)  \longrightarrow 0 ,
\end{equation}
where $C^\bullet (\cA)$ denotes the Hochschild cochain complex of a Fr\'echet algebra $\cA$, 
b the Hochschild coboundary, and $B$ is the Connes
coboundary (\emph{cf.}~\cite{Con:NDG,Con:NG}).
Recall that the \emph{Hochschild cohomology}\sind{Hochschild (co)homology} 
of $\cA$ is computed by the complex
$\big( C^\bullet (\cA), b\big)$,
the \emph{cyclic cohomology}\sind{cyclic (co)homology} of $\cA$ is the 
cohomology of the total complex \\
$\big( \tot^\bullet \cB C^{\bullet,\bullet} (\cA), b+B \big)$, where
\nind{T@$\tot^\bullet \cB C^{\bullet,\bullet} (\cA)$}
\nind{C@$ \cB C^{p,q} (\cA)$}
\nind{C@$ \cB C_\text{\rm\tiny per}^{p,q} (\cA)$}
\nind{T@$\tot^\bullet \cB C_\text{\rm\tiny per}^{\bullet,\bullet} (\cA)$}
\begin{displaymath}
  \cB C^{p,q} (\cA) =
  \begin{cases}
    C^{q-p} (\cA) := \big( \cA^{\hatotimes q-p+1}\big)^*, & 
    \text{for $q\geq p\geq 0$}, \\
    0, & \text{otherwise},
  \end{cases}
\end{displaymath}
while the \emph{periodic cyclic cohomology}\sind{periodic cyclic (co)homology} 
of $\cA$ is the cohomology of the total complex \\
$\big( \tot^\bullet \cB C_\text{\rm\tiny per}^{\bullet,\bullet} (\cA), b+B \big)$, 
where
\begin{displaymath}
  \cB C_\text{\rm\tiny per}^{p,q} (\cA) =
  \begin{cases}
    C^{q-p} (\cA) := \big( \cA^{\hatotimes q-p+1}\big)^*, & 
    \text{for $q\geq p$}, \\
    0, & \text{else}.
  \end{cases}
\end{displaymath}
In \cite{LesMosPfl:RPC} we noted that the relative cohomology 
theories, or in other words the cohomologies of the quotient mixed complex 
$\big( Q^\bullet ,b,B \big)$, can be calculated from a particular 
mixed complex quasi-isomorphic to $Q^\bullet$,  namely from the 
direct sum mixed complex 
\[ 
  \big( C^\bullet (\cA) \oplus C^{\bullet +1} (\cB) , \widetilde{b},\widetilde{B}\big), 
\]
where
\sind{Hochschild (co)homology!relative}\sind{relative!Hochschild (co)homology}
\nind{HH@$HH^\bullet (\cA,\cB)$}
\nind{HC@$HC^\bullet (\cA,\cB)$}
\nind{HP@$HP^\bullet (\cA,\cB)$}
\begin{equation}
\label{Eq:DefCoBdrRelMixDer}
  \widetilde b =  \left(
  \begin{array}{cc}
     b & -\sigma^* \\
     0 & -b  
  \end{array}
  \right), \quad \text{and} \quad 
  \widetilde B = \left(
  \begin{array}{cc}
     B & 0 \\
     0 & -B 
  \end{array}
  \right).
\end{equation}
In particular, the \emph{relative Hochschild cohomology}\sind{relative!Hochschild (co)homology}\sind{Hochschild (co)homology!relative} $HH^\bullet (\cA,\cB)$ is 
computed by the complex 
$\big( C^\bullet (\cA) \oplus C^{\bullet +1} (\cB) , \widetilde{b}\big)$,
the \emph{relative cyclic cohomology}\sind{relative!cyclic (co)homology}\sind{cyclic (co)homology!relative} $HC^\bullet (\cA,\cB)$ by 
the complex $ \big( \tot^\bullet \cB C^{\bullet,\bullet} (\cA)\oplus \tot^{\bullet +1} 
 \cB C^{\bullet,\bullet} (\cB), \widetilde{b} + \widetilde{B} \big) $,
and the \emph{relative periodic cyclic cohomology}\sind{relative!periodic cyclic (co)homology}\sind{periodic cyclic (co)homology!relative} $HP^\bullet (\cA,\cB)$ by 
$ \big( \tot^\bullet \cB C_\text{\tiny\rm per}^{\bullet,\bullet} (\cA)\oplus  
  \tot^{\bullet +1} \cB C_\text{\tiny\rm per}^{\bullet,\bullet} (\cB), 
  \widetilde{b} + \widetilde{B} \big) $.

Note that of course
\begin{equation}
\begin{split}
       \big( \tot^\bullet &\cB C^{\bullet,\bullet} (\cA)\oplus \tot^{\bullet +1} 
 \cB C^{\bullet,\bullet} (\cB), \widetilde{b} + \widetilde{B} \big) \\
 &\simeq \bigl(\tot^\bullet \cB C^{\bullet, \bullet}(\cA,\cB),\widetilde{b}+\widetilde{B}\big),
\end{split}
\end{equation}
where $\cB C^{p,q}(\cA,\cB):=\cB C^{p,q}(\cA)\oplus \cB C^{p,q+1}(\cB)$.

\sind{pairing}
Dually to relative cyclic cohomology, one can define relative cyclic homology 
theories. We will use these  throughout this article as well, and in particular 
their pairing with relative cyclic cohomology. For the convenience of the
reader we recall their definition,
referring to \cite{LesMosPfl:RPC} for more details. 
The short exact sequence 
\eqref{Eq:ShExSeq} gives rise to the following short exact sequence of 
homology mixed complexes
\begin{equation}
\label{Eq:HomMixComSeq}
  0 \rightarrow \big( K_\bullet , b, B\big) \rightarrow \big( C_\bullet (\cA), b , B\big) 
  \rightarrow \big( C_\bullet (\cB), b , B\big)  \rightarrow 0 ,
\end{equation}
where here $b$ denotes the Hochschild boundary, and $B$ the Connes boundary. 
\sind{Hochschild (co)homology}
The kernel mixed complex $K_\bullet$ is quasi-isomorphic to the direct sum mixed complex 
\[ 
  \big( C_\bullet (\cA) \oplus C_{\bullet +1} (\cB) , \widetilde{b},\widetilde{B}\big), 
\]
where
\begin{equation}
\label{Eq:DefBdrRelMixDer}
  \widetilde b =  \left(
  \begin{array}{cc}
     b & 0 \\
     -\sigma_* & -b  
  \end{array}
  \right), \quad \text{and} \quad 
  \widetilde B = \left(
  \begin{array}{cc}
     B & 0 \\
     0 & -B 
  \end{array}
  \right) .
\end{equation}
This implies that the \emph{relative cyclic homology}\sind{relative!cyclic (co)homology}\sind{cyclic (co)homology!relative} $HC_\bullet (\cA,\cB)$ is the 
homology of $\big( \htot_\bullet \mathcal B C_{\bullet,\bullet}
 (\cA,\cB) , \widetilde b + \widetilde B\big)$, 
where $\mathcal B C_{p,q} (\cA,\cB) =  \mathcal BC_{p,q} (\cA) \oplus 
 \mathcal BC_{p,q+1} (\cB)$. Likewise, the \emph{relative periodic cyclic homology}\sind{relative!periodic cyclic (co)homology}\sind{periodic cyclic (co)homology!relative} 
$HP_\bullet (\cA,\cB)$ is the homology of 
$\big( \hTot_\bullet \mathcal B C^\text{\tiny per}_{\bullet,\bullet}
(\cA,\cB) , \widetilde b + \widetilde B\big)$, where 
$\mathcal B C^\text{\tiny per}_{p,q} (\cA,\cB) = 
 \mathcal BC^\text{\tiny per}_{p,q} (\cA) \oplus 
 \mathcal BC^\text{\tiny per}_{p,q+1} (\cB)$. 

By \cite[Prop.~1.1]{LesMosPfl:RPC}, the relative cyclic 
(co)homology groups inherit a natural pairing
\begin{equation}
\label{Eq:RelCycPair}
  \langle - , - \rangle_\bullet : \,
  HC^\bullet (\cA,\cB) \times HC_\bullet (\cA,\cB)
  \rightarrow \C ,
\sind{pairing}
\nind{$< - , - >$}
\end{equation}
which will be called the {\it relative cyclic pairing}, and which on 
chains and cochains is defined by
\begin{equation}
\label{Eq:DualPair}
\begin{split}
  \langle - , -\rangle : \; & 
  \Big( \mathcal B C^{p,q} (\cA) \oplus \mathcal B C^{p,q+1} (\cB) \Big)
  \times \Big( \mathcal B C_{p,q} (\cA) \oplus \mathcal B C_{p,q+1} (\cB) \Big)
  \rightarrow 
  \C, \\
  & \big( (\varphi,\psi), (\alpha,\beta) \big) \mapsto 
    \langle \varphi , \alpha \rangle + \langle \psi , \beta \rangle .
\end{split}
\end{equation}
This formula also describes the pairing between the relative periodic cyclic 
(co)homology groups.

Returning to diagram \eqref{Dia:BasicShExSeq}, we can now express the 
(periodic) cyclic cohomology of the pair 
$\big( \cC^\infty (M), \cE^\infty (\partial M,M)\big) $ resp.~of the pair
$\big( \cC^\infty (M) , \cC^\infty (\partial M) \big)$ in terms of
the cyclic cohomology complexes of $\cC^\infty (M)$ and 
$\cE^\infty (\partial M,M)$ resp.~$\cC^\infty (\partial M)$. We note that the ideal
$\cJ^\infty (\partial M, M)$ is H-unital, since 
$\big(\cJ^\infty (\partial M, M)\big)^2 = \cJ^\infty (\partial M, M)$
(\emph{cf.}~\cite{BraPfl:HAWFSS}). 
Hence \excision\ holds true for the ideal $\cJ^\infty (\partial M, M)$, and 
any of the above cohomology theories of $\cJ^\infty (\partial M, M)$ coincides 
with the corresponding relative cohomology of the pair
$\big( \cC^\infty (M), \cE^\infty (\partial M,M)\big) $. In particular, we 
have the following chain of quasi-isomorphisms 
\begin{equation}
\label{Eq:DefRelCycCoh}
\begin{split}
  \tot^\bullet & \, \cB C^{\bullet,\bullet} \big( \cJ^\infty (\partial M ,M)\big) 
  \sim_\text{\tiny\rm qism}\tot^\bullet \cB C^{\bullet,\bullet}
  \big( \cC^\infty (M), \cE^\infty (\partial M ,M) \big)  
  \sim_\text{\tiny\rm qism} \\
  & \, \sim_\text{\tiny\rm qism} 
  \tot^\bullet  \, \cB C^{\bullet,\bullet} (\cC^\infty (M)) \oplus 
  \tot^{\bullet+1} \cB C^{\bullet,\bullet}  (\cE^\infty (\partial M , M)) .
\end{split}
\end{equation}
Next recall from \cite{BraPfl:HAWFSS} that the map 
\[
\begin{split}
  \tot^k &\, \cB C_\text{\tiny\rm per}^{\bullet,\bullet}  (\cC^\infty (M)) 
  \rightarrow \tot^k \cB C_\text{\tiny\rm per}^{\bullet,\bullet}  
  (\cE^\infty (\partial M , M)), \\
  \psi & \mapsto \Big( \! \big( \cE^\infty (\partial M, M) \big)^{\hatotimes k+1}
  \ni F_0  \otimes \ldots \otimes F_k  \mapsto 
  \psi \big( {F_0}_{\|X}  \otimes \ldots \otimes {F_k}_{\|X} \big) \! \Big)
\end{split}
\]
between the periodic cyclic cochain complexes is a quasi-isomorphism. 
As a consequence of the Five Lemma one  obtains quasi-isomorphisms
\begin{equation}
\label{Eq:qismRelCycCoh}
\begin{split}
  \tot^\bullet &\,\cB C_\text{\tiny\rm per}^{\bullet,\bullet} 
  \big( \cJ^\infty (\partial M, M) \big) \sim_\text{\tiny\rm qism}
  \tot^\bullet \cB C_\text{\tiny\rm per}^{\bullet,\bullet}
  \big( \cC^\infty (M), \cE^\infty (\partial M,M)\big)\sim_\text{\tiny\rm qism} \\
  & \, \sim_\text{\tiny\rm qism} 
  \tot^\bullet  \, \cB C_\text{\tiny\rm per}^{\bullet,\bullet} (\cC^\infty (M)) 
  \oplus 
  \tot^{\bullet+1} \cB C_\text{\tiny\rm per}^{\bullet,\bullet}  
  (\cC^\infty (\partial M)) .
\end{split}
\end{equation}
In this paper we will mainly work with the relative complexes
over the pair of algebras $\big( \cinf{M},\cinf{\pl M} \big)$, because its cycles 
carry geometric information about the boundary, which is lost when considering 
only cycles over the ideal $\cJ^\infty (\partial M,M)$. In this respect
we note that periodic cyclic cohomology satisfies \excision\ 
by \cite{CunQui:EPCC,CunQui:EPCCII},
hence in the notation of \eqref{Eq:ShExSeq},
$HP^\bullet(\cJ)$ is canonically isomorphic to $HP^\bullet(\cA,\cB)$.

\section{The Chern character} 
\label{s:TheChernCharacter}\sind{Chern character|(}
For future reference, we recall the Chern character and its transgression in
cyclic homology, both in the even and in the odd case.

\subsection{Even case}
\nind{M@$\Mat_N(\cA),\Mat_\infty(\cA)$}
\nind{C@$\ch_\bullet(e)$}
The Chern character of an idempotent 
$e\in\Mat_\infty(\cA) := \lim\limits_{N\to\infty} \Mat_N(\cA)$ is 
the class in $ HP_0 (\cA)$ of the cycle given by the formula
\begin{equation}
  \ch_\bullet(e) := \tr_0 (e) + \sum_{k=1}^\infty (-1)^k \frac{(2k)!}{k!}  
  \tr_{2k} \Big( \big(e - \frac{1}{2} \big)\otimes e^{\otimes (2k)}\Big) ,
\end{equation}
where for every $j\in \N$ the symbol $e^{\otimes j}$ is an abbreviation for the 
$j$-fold tensor product $e\otimes\dots\otimes e$,
and $\tr_j$ denotes the generalized trace map 
$\Mat_N(\cA)^{\otimes j}\longrightarrow \cA^{\otimes j}$.
\nind{e@$e^{\otimes j}$}
\nind{t@$\tr_j$}
\sind{trace!generalized trace map}
\sind{generalized trace map}

If $(e_s)_{0\leq s \leq 1}$ is a smooth path of idempotents, then the 
transgression formula reads
\begin{equation}   
\label{evenh}
  \frac{d}{d s}\ch_\bullet(e_s)
  = (b+B)\slch_\bullet(e_s,(2e_s-1)\dot e_s) ;
\end{equation}
here the secondary Chern character $\slch_\bullet$ is given by
\sind{Chern character!secondary}
%
\begin{equation}
  \slch_\bullet(e,h):=\iota (h)\ch_\bullet(e),
\end{equation}
where the map $\iota (h)$ is defined by
\begin{equation}\begin{split}
 \iota (h)&(a_0\otimes a_1\otimes \ldots\otimes a_l)\\
        &=\sum_{i=0}^l (-1)^i 
        (a_0\otimes \ldots\otimes a_i\otimes h\otimes a_{i+1}
        \otimes \ldots\otimes a_l).
\end{split}
\end{equation}

\sind{Ktheory@$K$-theory!relative}
\sind{relative!Ktheory@$K$-theory}
\nind{K@$K_j(\cA,\cB)$}
\nind{c@$\ch_\bullet(p,q,h)$}
%
A relative $K$-theory class in $K_0(\cA,\cB)$ can be represented 
by a triple $(p,q,h)$ with projections $p,q\in\Mat_N(\cA)$ and $h:[0,1]\to \Mat_N(\cB)$
a smooth path of projections with $h(0)=\sigma(p), h(1)=\sigma(q)$
(\emph{cf.} \cite[Def. 4.3.3]{HigRoe:AKH}, see also \cite[Sec. 1.6]{LesMosPfl:RPC}).
The Chern character of $(p,q,h)$ is represented by the relative cyclic cycle
\begin{equation}\label{eq:ChernCharEven}
\ch_\bullet\bigl( p,q,h\bigr)\, = \, 
\Bigl( \ch_\bullet(q)-\ch_\bullet(p) \, ,\, -\Tslch_\bullet(h)\Bigr),
\end{equation}
where
\begin{equation}\label{eq:ChernTransEven}
\Tslch_\bullet(h)=\int_0^1 \slch_\bullet\bigl(h(s),(2h(s)-1)\dot h(s)\bigr) ds.
\end{equation}
That the r.h.s. of Eq.~\eqref{eq:ChernCharEven} is a relative cyclic cycle follows
from the transgression formula Eq.~\eqref{evenh}. 
From a secondary transgression formula \cite[(1.43)]{LesMosPfl:RPC} 
one deduces that \eqref{eq:ChernCharEven} indeed corresponds
to the standard Chern character on $K_0(\cJ)$ under \excision.
\sind{transgression formula}
\sind{transgression formula!secondary}

\subsection{Odd case} 
\nind{G@$\GL_\infty(\cA), \GL_N(\cA)$}
The odd case parallels the even case in many aspects.
Given an element $g\in \GL_\infty (\cA):=\lim\limits_{N\to\infty} \GL_N(\cA),$
the odd Chern character is the following normalized periodic cyclic cycle: 
\begin{equation}
  \ch_\bullet (g) \,=\, \sum_{k=0}^\infty\, (-1)^k  k! 
  \tr_{2k+1} \big( (g^{-1} \otimes g)^{\otimes (k+1)} \big). 
\end{equation}  

If $(g_s)_{0\leq s \leq 1}$ is a smooth path in $\GL_\infty(\cA)$, the transgression
formula (\emph{cf}.~\cite[Prop.~3.3]{Get:OCC}) reads
\begin{equation}
\label{Eq:ChTransgress}
  \frac{d}{ds} \ch_\bullet (g_s) \, =\, (b+B) \, 
  \slch_\bullet (g_s, \dot g_s),
\end{equation}
where the secondary Chern character $\slch_\bullet$ is defined by  
\begin{align}
  \slch_\bullet & (g , h ) = \tr_0 (g^{-1}h) + \\
    + & \sum_{k=0}^\infty (-1)^{k+1} k! 
    \sum_{j=0}^k \tr_{2k+2} \big( (g^{-1} \otimes g)^{\otimes (j+1)}\otimes
   g^{-1} h \otimes (g^{-1} \otimes g)^{ \otimes (k-j)} \big).\nonumber
\end{align}
%

A relative $K$-theory class in $K_1(\cA,\cB)$ can be represented by
a triple $(U,V,h)$, where $U,V\in\Mat_N(\cA)$ are unitaries and
$h:[0,1]\to \Mat_N(\cB)$ is a path of unitaries joining $\sigma(U)$
and $\sigma(V)$. Putting
\begin{equation}\label{eq:ChernTransOdd}
\Tslch_\bullet(h)=\int_0^1 \slch_\bullet\bigl(h_s,\dot h_s\bigr) ds,
\end{equation}
the Chern character of $(U,V,h)$ is represented by the relative cyclic cycle
\nind{c@$\ch_\bullet( U,V,h)$}
\begin{equation}\label{eq:ChernCharOdd}
\ch_\bullet\bigl( U,V,h\bigr)\, = \, 
\Bigl( \ch_\bullet(V)-\ch_\bullet(U) \, ,\, -\Tslch_\bullet(h)\Bigr).
\end{equation}
Again the cycle property follows from the transgression formula
Eq.~\eqref{Eq:ChTransgress} and with the aid of a secondary transgression
formula \cite[(1.15)]{LesMosPfl:RPC} one shows that
\eqref{eq:ChernCharOdd} corresponds to the standard Chern character
on $K_1(\cJ)$ under \excision\ \cite[Thm.~1.7]{LesMosPfl:RPC}.
\sind{Chern character|)}

\section{Dirac operators and $q$-graded Clifford modules}
\label{s:qDirac}
\nind{Clq@$\Cl_q$|(}
To treat both the even and the odd cases simultaneously we make use of the Clifford
supertrace (\emph{cf.~e.g.}~\cite[Appendix]{Get:CHA}).
Denote by $\Cl_q$ the complex Clifford algebra\sind{Clifford algebra!complex} 
on $q$ generators, that is $\Cl_q$ is the universal $C^*$-algebra on unitary generators 
$e_1,...,e_q$ subject to the relations
\begin{equation}
\label{eq:Clifford-relations}
  e_j e_k+e_k e_j= -2 \delta_{jk}.
\end{equation}
Let $\sH=\sH^+\oplus \sH^-$ be a $\Z_2$-graded Hilbert space with grading operator
$\ga$. We assume additionally that $\sH$ is a $\Z_2$-graded right $\Cl_q$-module.
Denote by $\sfcr: \sH \otimes \Cl_q \rightarrow \sH$ the right $\Cl_q$-action
and define operators $E_j:\sH \rightarrow \sH$ for $j=1,\cdots,q$ by 
$E_j := \sfcr\big( - \otimes e_j\big)$. Then the $E_j$ are unitary operators 
on $\sH$ which anti-commute with $\ga$.

\nind{L@$\sL(\sH)$}
The $C^*$-algebra $\sL(\sH)$ of bounded linear operators on $\sH$ is naturally 
$\Z_2$-graded, too. For operators $A,B\in\sL(\sH)$ of pure degree $|A|,|B|$ the 
\textit{supercommutator}\sind{supercommutator} is defined by
\begin{equation}\label{eq:def-super-commutator}
      [A,B]_\super:= AB-(-1)^{|A||B|}BA.
\end{equation}
Furthermore denote by $\sL_{\Cl_q}(\sH)$ the \textit{supercommutant} of 
$\Cl_q$ in $\sH$, that is $\sL_{\Cl_q}(\sH)$ consists of those $A\in \sL(\sH)$ for 
which  $[A,E_j]_\super=0$, $j=1,...,q$.
For $K\in \sL_{\Cl_q}^1(\sH)=\bigsetdef{A\in \sL_{\Cl_q}(\sH)}{A \text{ trace class }}$ one 
defines the \emph{degree $q$ Clifford supertrace}\sind{Clifford supertrace!degree $q$}
\begin{equation}\label{eq:Clifford-super-trace}
     \Str_q(K):=(4\pi)^{-q/2}\Tr(\ga E_1\cdot ... \cdot E_q K).
\end{equation}
The following properties of $\Str_q$ are straightforward to verify.

\begin{lemma}\label{l:clifford-trace}
For $K, K_1, K_2\in \sL_{\Cl_q}^1(\sH)$, 
one has 
\begin{enumerate}
\item $\Str_q K=0$, if $|K|+q$ is odd.
\item $\Str_q$ vanishes on super-commutators: $\Str_q([K_1,K_2]_\super)=0$.
\end{enumerate}
\end{lemma}

Let $(M,g)$ be a smooth riemannian manifold.
Associated to it is the bundle $\Cl (M) := \Cl (T^*M)$ of Clifford algebras. Its fiber over 
$p\in M$ is given by the Clifford algebra generated by elements of
$T^*_pM$ subject to the relations
\begin{equation}
  \label{eq:cliffordrel}
  \xi \cdot \zeta + \zeta \cdot \xi = - 2 g (\xi,\zeta) \quad 
  \text{for all $\xi,\zeta \in T^*_pM$}.
\end{equation}
\begin{definition}[\emph{cf.}~{\cite[Sec.~5]{Get:CHA}}] 
  Let $q$ be a natural number. By a 
  \textit{degree $q$ Clifford module}\sind{Clifford module! degree $q$} 
  over $M$ one then understands a $\Z_2$-graded complex vector bundle  $W \rightarrow M$ 
  together with a hermitian metric $\langle - ,-\rangle$, a Clifford action
  $\sfc = \sfcl: \bT^* M \otimes W \rightarrow W$, and an action 
  $\sfcr: W \otimes \Cl_q \rightarrow W$
  such that both actions are graded and unitary  and supercommute with each other. 
  A \textit{Clifford superconnection} on a degree $q$ Clifford module $W$ over 
  $M$ is a superconnection 
  \[
    \A : \Omega^\bullet (M, W) := 
    \Gamma^\infty \big( M;\Lambda^\bullet(T^*M) \otimes W \big) \rightarrow  
    \Omega^\bullet (M, W) 
  \]
  which supercommutes with the action of $\Cl_q$, satisfies
  \[
     \big[\A ,  \sfc (\omega) \big]_\super = \sfc \big(\nabla \omega \big) \quad 
     \text{for all $\omega \in \Omega^1 (M)$},
  \]   
  and is metric in the sense that
  \[
     \left\langle \A \xi , \zeta \right\rangle  +
     \left\langle \xi , \A \zeta \right\rangle =
     d \left\langle \xi , \zeta \right\rangle \quad 
     \text{for all $\xi,\zeta \in \Omega^\bullet (M, W)$}.
  \]
  Here, and in what follows, $\nabla$ denotes the Levi-Civita connection 
  belonging to $g$. 

  The \textit{Dirac operator} associated to a degree $q$ Clifford module 
  $W$ and a Clifford superconnection $\A$ is defined as the differential 
  operator 
  \[
    \dirac := \sfcl \circ \A : \Gamma^\infty (M;W) \rightarrow 
    \Gamma^\infty \big( M;\Lambda^\bullet(T^*M) \otimes W \big) \rightarrow 
    \Gamma^\infty (M;W).
  \]
\end{definition}
\sind{Dirac operator}
\sind{Clifford (super)connection}
In this paper the term ``Dirac operator"  will always refer to
the Dirac operator associated to a Clifford (super)connection in the above sense. 
Such Dirac operators are automatically formally self-adjoint.
By a \semph{Dirac type operator} we understand a first order differential operator 
such that the principal symbol of its square is scalar (\emph{cf.}~\cite{Tay:PDEII}).

\nind{Clq@$\Cl_q$|)}

\subsection{The \textup{JLO} cochain associated to a Dirac operator}\label{ss:JLODO}
Let $M$ be a compact riemannian manifold without boundary
and let $\dirac$ be a Dirac type operator as described above. 
Since $M$ is compact and since $\dirac$ is elliptic the 
heat operator $e^{-r\dirac^2}, r>0,$ is smoothing and hence
for pseudodifferential operators $A_0,\cdots, A_k \in \pdo^\infty (M,W)$
we put, following \cite[Sec. 2]{Get:CHA}, 
\begin{equation}
  \begin{split}
   \langle A_0, \cdots,A_k \rangle_{\dirac_t}  \, := &\int_{\Delta_k} \, \Str_q
    \big(A_0 \, e^{- \sigma_0 \dirac_t^2} \cdots A_k \, e^{- \sigma_k \dirac_t^2} \big)d\sigma \\
               = &\Str_q \big( (A_0,\ldots,A_k )_{\dirac_t} \big), 
\end{split}
\end{equation}
where 
\begin{equation}
  \label{eq:simplex}
  \Delta_k:=
 \bigsetdef{\sigma=(\sigma_0,...,\sigma_k)\in \R^{k+1}}{\sigma_j\ge 0, \,\sigma_0+\ldots+\sigma_k=1}
\end{equation}
denotes the standard $k$-simplex
and 
\begin{equation}\label{eq:ML20090128-3}
   (A_0,\ldots,A_k )_{\dirac}  
      := \int_{\Delta_k} \, A_0 \, e^{- \sigma_0 \dirac^2} \cdots A_k \, e^{- \sigma_k \dirac^2} d\sigma.
\end{equation}
Furthermore, for smooth functions $a_0,\ldots,a_k\in\cC^\infty(M)$, one puts
\begin{align} 
\label{Eq:DefChernClosed}
  \Ch^k (\dirac) (a_0,\cdots,a_k) & := \langle a_0, [\dirac,a_1],
  \cdots, [\dirac, a_k]\rangle , \\
\label{Eq:DefSlChernClosed}
\begin{split} 
  \slch^k(\dirac , \mathsf{V} )(a_0,\cdots,a_k) &:= \\ 
    \sum_{0\leq j \leq k} (-1)^{j \, \deg{\mathsf{V}}} \,  
    \langle a_0, [\dirac,a_1],&
    \cdots , [\dirac, a_j], \mathsf{V} ,  [\dirac, a_{j+1}], \cdots , 
    [\dirac, a_k] \rangle.
\end{split}
\end{align}
$\Ch^\bullet(\dirac)$ is, up to a normalization factor depending on $q$,
the \JLO\ cocycle associated to $\dirac$.
For a comparison with the standard non-Clifford covariant \JLO\ cocycle
see also Section \plref{s:CocTransgressNoClifford} below.

Now consider a family of Dirac operators, $\dirac_t$, depending smoothly
on a parameter $t$. The operation $\slch$ will mostly be used with 
$V=\dot \dirac_t$ as a second argument. Here $\dirac_t$ is considered of 
odd degree regardless of the value of $q$.

$\Ch^\bullet(\dirac_t)$ then satisfies 
\begin{equation}\label{eq:cocycle}
 b \Ch^{k-1}(\dirac_t) + B \Ch^{k+1}(\dirac_t) =0
\end{equation}
and
\begin{equation}\label{eq:transgression}
 \frac{d}{dt} \Ch^k(\dirac_t) +
 b\slch^{k-1}(\dirac_t , \dot{\dirac}_t ) +B\slch^{k+1}(\dirac_t,\dot\dirac_t) =0.
\end{equation}

\section[Relative Connes--Chern character]{The relative Connes--Chern character of a 
Dirac operator over a manifold with boundary}
\label{Sec:RelConCheDirac}
\sind{ConnesCherncharacter@Connes--Chern character!relative}\sind{relative!ConnesCherncharacter@Connes--Chern character}
In this section, $M$ is a compact manifold with boundary, $g_0$ is a riemannian metric 
which is smooth up to the boundary, and $W \rightarrow M$ is a degree $q$ Clifford module.
We choose a hermitian metric $h$ on $W$ together with a
Clifford connection which is unitary with respect to $h$. 
Let $\dirac=\dirac(\nabla,g_0)$ be the associated Dirac operator; we suppress the dependence on 
$h$ from the notation. Then $\dirac$ is a densely defined operator on the Hilbert space $\sH$
of square-integrable sections of $W$. 

\sind{relative!Fredholm module}
According to \cite[Prop.~3.1]{BauDouTay:CRC}, as 
outlined in the introduction, $\dirac$ defines a relative Fredholm 
module over the pair of $C^*$-algebras $\bigl(\cC(M), \cC(\pl M) \bigr)$. Recall that 
the relative Fredholm module is given by\nind{F@$F$} \nind{D@${\mathsf D}_e$} 
$F = \dirac_e (\dirac^*_e \dirac_e +1)^{-1/2}$, where $\dirac_e$ is a closed extension of 
$\dirac$ such that either $\dirac^*_e \dirac_e$ or $\dirac_e \dirac^*_e$  has compact 
resolvent ({\em e.g.} 
both the closure $\dirac_{\rm min}=\clos{\dirac}$ and the ``maximal extension" $\dirac_{\rm
max}=(\dirac^t)^*$, which is the adjoint of the formal adjoint, satisfy this condition), 
and that the  $K$-homology class $[F]$ does not depend on the particular choice 
of $\dirac_e$ (see \cite[Prop.~3.1]{BauDouTay:CRC}).
Furthermore, \cite[\S 2]{BauDouTay:CRC} shows that over the $C^*$-algebra
\nindp{$\cC_0 (M\setminus \partial M)$} of continuous functions vanishing at
infinity, whose \Khomology\ is by \excision\ isomorphic to the relative
$K$--homology group $K^\bullet\big(\cC(M),\cC(\partial M)\big)$, 
\nind{KCM@$K^\bullet(\cC(M),\cC(\partial M))$}
one has even more freedom to choose closed extensions
of $\dirac$, and in particular the self-adjoint extension $\dirac_{\APS}$ obtained by imposing
{\APS}\ boundary conditions yields the same $K$-homology class as $[F]$\nind{F@$[F]$}
in $K^\bullet\bigl(\cC_0(M\setminus\partial M)\bigr)$.

It is well-known that $\dirac_{\APS}$ has an $m^+$-summable resolvent (\emph{cf. e.g.}
\cite{GruSee:WPP}). Moreover, multiplication by $f\in \cJ^\infty(\partial M,M)$
preserves the {\domain} of $\dirac_{\APS}$ and $[\dirac_{\APS},f]=\sfc(df)$ is bounded.
Thus $\dirac_{\APS}$ defines naturally an $m^+$-summable Fredholm module over
the local $C^*$-algebra  
$\cJ^\infty (\partial M,M) \subset \cC_0 (M\setminus \partial M)$. Since by \excision\
in $K$-homology $K^\bullet \big(\cC_0 (M\setminus \partial M)\big)$ is naturally isomorphic to 
$K^\bullet \big(\cC(M), \cC(\pl M) \big)$, one concludes that the class $[F]$ of the 
relative Fredholm module  coincides under this 
isomorphism with the class $[\dirac]$ of the $m^+$-summable Fredholm module over 
$\cJ^\infty (\partial M,M)$. 

Let us now consider the {\CoChch} of $[\dirac]$. According to 
\cite{ConMos:TCC}, it can be represented by the truncated \JLO-cocycle 
of the operator $\dirac$ (with $n\geq m$ of the same parity as $m$): 
\begin{equation}
  \ch_t^n (\dirac)  = \sum_{k \geq 0} \Ch^{n-2k} (t\dirac) +
  B \Tslch_t^{n+1} (\dirac) .
\end{equation}
Recall from \cite{JLO:QKT} that the \JLO-cocycle is given by 
\begin{equation}
\label{Eq:JLOdirac}
  \begin{split}
  \Ch^k (\dirac ) (a_0 , \ldots , a_k )  = &
  \int_{\Delta_k} \Str_q\bigl( a_0 e^{-\sigma_0 \dirac^2} [\dirac,a_1]\ldots[\dirac,a_k] e^{-\sigma_k \dirac^2}\bigr)d\sigma, \\
  & \text{for } a_0, \ldots , a_k \in \cJ^\infty (\partial M, M) . 
 \end{split}
\end{equation}
Note that the cyclic cohomology class of $\ch_t^n (\dirac)$ is independent of 
$t$, and that $\ch_t^\bullet$ is the {\CoChch} as given in 
Diagram \eqref{dia:FredModHP}. To obtain the precise form of the {\CoChch} 
$ \ch_t^\bullet (\dirac) \in HP^\bullet \big( \cJ^\infty(\partial M , M) \big)$ 
one notes first that by \cite{BraPfl:HAWFSS} 
$HP^\bullet \big( \cJ^\infty(\partial M , M) \big)$
is isomorphic to the relative \deRham\ cohomology group 
$H^\textrm{dR}_\bullet (M,\partial M;\C)$  and then 
one has to calculate the limit 
$\lim_{t\searrow 0} \ch_t^n (\dirac) (a_0, \ldots , a_k)$. 
Since $\Ch^k$ is continuous with respect to the Fr\'echet topology on 
$\cJ^\infty (\partial M, M)$, and $\cC^\infty_c (M\setminus \partial M)$ 
is dense in $\cJ^\infty (\partial M, M)$, it suffices to consider the case where 
all $a_j$ in \eqref{Eq:JLOdirac} have compact support in $M\setminus \partial M$. 
But in that case one can use standard local heat kernel\sind{heat kernel} analysis 
or Getzler's asymptotic calculus as in \cite{ConMos:CCN} or \cite{BloFox:APO} to 
show that for $n \geq m$ and same parity as $m$
\begin{equation}
  \label{eq:ExpCheConChar}
  \lim_{t\searrow 0} \big[\ch_t^n\big]_k (\dirac) (a_0, \ldots , a_k) = 
  \int_M \go_{\dirac}\wedge a_0 \, da_1 \wedge \ldots da_k .
\end{equation}
Here, $\big[\ch_t^n\big]_k$ denotes the component of $\ch_t^n$ of degree $k$ and
$\go_{\dirac}$ is the local index form of $\dirac$.  
By Poincar\'e duality the class of the current \eqref{eq:ExpCheConChar}
in $H^\textrm{dR}_\bullet (M,\partial M)$ depends 
only on the absolute \deRham\ cohomology class of $\go_{\dirac}$ in 
$H_{\textrm{dR}}^\bullet(M)$. 
By the transgression formul\ae\ this cohomology class is independent
of $\nabla$ and $g_0$. 

Finally let $g$ be an arbitrary smooth metric on the \emph{interior} 
$M^\circ=M\setminus \partial M$ which does not necessarily extend to the boundary.
Then we can still conclude from the transgression formula 
that the absolute de Rham cohomology class in $H_\textrm{dR}^\bullet(M)\cong
H_\textrm{dR}(M\setminus\partial M)$ of the index form $\go_{\dirac}$
of $\dirac(\nabla,g)$ represents the {\CoChch} of $[(\dirac(\nabla,g_0)]$.

Summing up and using Diagram \eqref{dia:FredModHP}, we obtain
the following statement.

\begin{proposition}\label{p:ML200909292} 
Let $M$ be a compact manifold with boundary and 
riemannian metric $g_0$. Let $W$ be a 
degree $q$ Clifford module over $M$. For any choice of a hermitian metric $h$ and unitary
Clifford connection $\nabla$ on $W$  the Dirac operator $\dirac = \dirac (\nabla, g_0) $
defines naturally a class $[\dirac] \in K_m(M\setminus\pl M)$.
The {\CoChch} of $[\dirac]$ is independent of
the choice of $\nabla$ and $g_0$. In particular the index map 
$\Index_{[\dirac]}:K^\bullet(M,\partial M)\longrightarrow\Z,$
defined by the pairing with $K$-theory, is independent of $\nabla$ and $g_0$.

Furthermore, for \emph{any} smooth riemannian metric $g$ in the \emph{interior}
$M^\circ=M\setminus\partial M$ the de Rham cohomology class $\go_{\dirac(\nabla,g)}$ represents
the {\CoChch} of $\dirac(\nabla,g_0)$.
\end{proposition}
\nind{Index@$\Index_{[{\mathsf D}]}$}

\section{Exact \textup{b}-metrics and \textup{b}-functions on cylinders}
\label{App:bdefbmet}
Let $M$ be a compact manifold with boundary of dimension $m$, let $\partial M$ 
be its boundary, and denote by $M^\circ$ its interior $M\setminus \partial M$. 
Then choose a collar  for $M$ which means a diffeomorphism of the form
$(r,\eta) : Y  \rightarrow [0,2)\times \partial M$, where 
$Y \subset M$ is an open neighborhood of $\partial M = r^{-1} (0)$. The
map $r: Y \rightarrow [0, 2)$ is called the 
\textit{boundary defining function}\sind{boundary defining function}
of the collar, the submersion $\eta : Y \rightarrow \partial M$
its \textit{boundary projection}\sind{boundary projection}. 
For $s \in \, (0, 2)$ denote by $Y^s$ 
the open subset $r^{-1} \big( [0,s) \big)$, 
put $Y^{\circ s}:= r^{-1}\big( (0,s)\big)$ and finally let 
$M^s := M \setminus Y^s$ and $\overline{M^s} := M \setminus Y^{\circ s}$;
likewise $Y^\circ:=Y\setminus \partial M$. 
Next, let $x : Y \rightarrow \R$ be the
smooth function $x := \ln \circ r$. Then 
$(x,\eta): Y^{\circ 3/2} \rightarrow \, (-\infty, \ln \frac 32) \times \partial M$ 
is a diffeomorphism of $Y^{3/2}$ onto a cylinder.

After having fixed these data for $M$, we choose the most essential 
ingredient for the \textup{b}-calculus, namely an 
\emph{exact \textup{b}-metric}\sind{bmetric@\textup{b}-metric}\sind{exactbmetric@exact \textup{b}-metric} for $M$. Following \cite{Mel:APSIT}, one 
understands  by this  a riemannian metric $\bmet$ on $M^\circ$ such that on 
$Y^\circ$, the metric can be written in the form 
\begin{equation}
  {\bmet}_{|Y^\circ}= \frac{1}{r_{|Y^\circ}^2} ( dr \otimes dr)_{|Y^\circ} + 
  \eta_{|Y^\circ}^*\pmet,
\end{equation}
where $\pmet$ is a riemannian metric on the boundary $\partial M$.
If $M$ is equipped with an exact \textup{b}-metric we will for brevity
call it a \emph{\textup{b}-manifold}. \sind{bmanifold@\textup{b}-manifold}

Clearly, one then has in the cylindrical coordinates $(x,\eta)$
\begin{equation}
  {\bmet}_{|Y^\circ}= (dx \otimes dx)_{|Y^\circ}  +  \eta_{|Y^\circ}^* \pmet  .
\end{equation}

\FigbManifold
\FigManifoldCylindricalEnds

This means that the interior $M^\circ$ together with $\bmet$ is a
\emph{complete manifold with cylindrical ends}.
\sind{cylindrical ends}\sind{manifold with cylindrical ends}
Thus although we usually tend to visualize a compact manifold with boundary like
in Figure \ref{fig:M}, a \textup{b}-manifold looks like the one
in Figure \ref{fig:MCylinder}. For calculations it will often be
more convenient to work in cylindrical coordinates and hence next
we are going to show how the smooth functions on $M$ can be described
in terms of their asymptotics on the cylinder.

Consider the cylinder $\R\times \partial M := \R\times \partial M$ 
together with the product metric 
\begin{equation}\label{eq:ML20090219-1}
  \cylmet = dx \otimes dx + \operatorname{pr}_2^* \pmet ,
\end{equation}
where here (with a slight abuse of language), $x$ denotes the first coordinate 
of the cylinder, and 
$\operatorname{pr}_2: \R\times \partial M \rightarrow \partial M$ the 
projection onto the second factor. 

Next we introduce various algebras of what we choose to call 
\emph{\textup{b}-functions}\sind{bfunction@\textup{b}-function} on 
$\R\times \partial M$. For $c\in \R$ define 
$\bcC \big( (-\infty , c) \times \partial M \big)$ 
resp.~$\bcC \big( (c,\infty) \times \partial M \big)$ as the algebra 
of smooth functions $f$ on $(-\infty , c) \times \partial M$ 
resp.~on $(c,\infty) \times \partial M $ for which there exist functions 
$f^-_0,f^-_1,f^-_2,\ldots \in \cC^\infty (\partial M) $ 
resp.~$f^+_0,f^+_1,f^+_2,\ldots \in \cC^\infty (\partial M) $ such that the
following asymptotic expansions hold true in $x\in \R$:
\begin{equation}\label{eq:ML20090219-3}
\begin{split}
  f (x, -)& \, \sim_{x \rightarrow -\infty} 
  f^-_0 + f^-_1 e^{x} + f^-_2 e^{2x} + \ldots  
  \quad \text{resp.}\\
  f (x,-)& \sim_{x \rightarrow \infty}  
  f^+_0 + f^+_1 e^{-x} + f^+_2 e^{-2x} + \ldots   \, .
\end{split}
\end{equation}
More precisely, this means that there exists for every $k,l\in \N$ and every 
differential operator $D$ on  $\partial M$ a constant $C>0$ such that
\begin{equation}\label{eq:ML20090219-4}
\begin{split}
 \Big| & \partial^l_x D f (x,p) - 0^l D f^-_0 (p) - 
 \ldots - k^l Df^-_k (p) e^{kx} \Big| 
 \leq C e^{(k+1) x} \\ & \hspace{50mm} \text{for all $x\leq c-1$ and 
 $p\in \partial M$ resp.}\\
 \Big| & \partial^l_x D f (x,p) - 0^l D f^+_0 (p) -   
 \ldots - (-k)^l Df^+_k (p) e^{-kx} \Big| 
 \leq C e^{-(k + 1) x} \\ & \hspace{50mm} \text{for all $x\geq c+1$ and 
 $p\in \partial M$}.
\end{split}   
\end{equation}
The asymptotic expansion guarantees that 
$f \in \bcC \big( (-\infty , c) \times \partial M  \big)$ if and only if the 
transformed function 
$[0,e^c[ \times \partial M \ni (r,p) \mapsto f \big( \ln r , p\big)$
is a smooth function on the collar $[0,e^c[ \times \partial M$. 

The concept of \textup{b}-functions has an obvious global meaning on
$M^\circ$. Because of its importance we single it out as
\begin{proposition}\label{p:bsmooth}
 A smooth function $f\in \cC^\infty(M^\circ)$ extends to a smooth
 function on $M$ if and only if it is a \textup{b}-function.\sind{bfunction@\textup{b}-function}
 In other words this means that the restriction map $\cC^\infty(M)\ni f\mapsto f_{|M^\circ}\in
 \bcC(M^\circ)$ is an isomorphism of algebras.
\end{proposition}
The claim is clear from the asymptotic expansions \eqref{eq:ML20090219-4}. 

The algebra $\bcC \big( \R\times \partial M \big)$ of \textup{b}-functions on the full
cylinder consists of all smooth functions $f $ on $\R \times \partial M$
such that 
\[ f_{| (-\infty , 0) \times \partial M} \in 
   \bcC \big( (-\infty , 0) \times \partial M \big)\quad \text{and} 
   \quad
   f_{| (0,\infty)\times \partial M  } \in 
   \bcC \big( (0,\infty)\times \partial M  \big).
\] 
Next, we define the algebras of \textup{b}-functions with compact support on the 
cylindrical ends by $\bcptC  \big( (-\infty , c) \times \partial M \big) := $
\[
  \big\{ f \in \bcC \big( (-\infty , c) \times \partial M \big) \mid
  f(x,p) = 0 \text{ for $x\geq c-\varepsilon$, $p\in \partial M$ and
  some $\varepsilon >0$}\big\} 
\]
resp.~by  $\bcptC  \big( (c,\infty)\times \partial M  \big) := $
\[
  \big\{ f \in \bcC \big( (c,\infty)\times \partial M \big) \mid
  f(x,p) = 0 \text{ for $x\leq c+\varepsilon$, $p\in \partial M$ and
  some $\varepsilon >0$}\big\}. 
\]
For sections in a vector bundle the notation
$\Gamma^\infty_\textup{cpt}((-\infty,0)\times \partial M;E)$ has the
analogous meaning.

The essential property of the thus defined algebras of 
\textup{b}-functions is that the coordinate system 
$(x,\eta) : Y^{\circ 3/2} \rightarrow (-\infty , \ln 3/2 ) \times \partial M$
induces an isomorphism
\begin{equation}
 (x,\eta)^* : \bcC \big( (-\infty , \ln 3/2 ) \times \partial M \big)
 \rightarrow \cC^\infty \big( Y^{\frac 32}\big),  
\end{equation}
which is defined by putting
\[ 
 (x,\eta)^* f := \Biggl( 
 Y^{\frac 32} \ni p \mapsto
 \begin{cases}
   f \big( x(p), \eta(p)\big) , & \text{if $p \notin \partial M $,} \\
   f^-_0 (\eta(p)), & \text{if $p \in \partial M $.}
 \end{cases} 
 \Biggr), 
\]
for all $f \in \bcC \big( (-\infty , \ln 3/2 ) \times \partial M \big)$. Under this 
isomorphism, $\bcptC \big( (-\infty , 3/2 ) \times \partial M \big)$ is mapped
onto $\cC_\txtcpt^\infty \big( Y^{\frac 32}\big)$.
In this article, we will use the isomorphism $(x,\eta)^*$ to obtain 
essential information about solutions of boundary value problems on $M$ by 
transforming the problem to the cylinder over the boundary and then performing 
computations there with \textup{b}-functions on the cylinder.

The final class of \textup{b}-functions used in this work is the algebra 
$\bsS \big( \R\times \partial M \big)$ of 
\textit{exponentially fast decreasing functions} or 
\textit{\textup{b}-Schwartz test functions} on the cylinder defined as
the space of all smooth functions $f \in \cC^\infty \big( \R\times \partial M \big)$
such that for all $l,n \in \N$ and all differential operators $D \in \Diff (\partial M)$
there exists a  $C_{l,D,n} >0$ such that
\begin{equation}
  \label{Eq:DefbSchwartz}
  \left| \partial^l_x D f(x,p)\right| \leq C_{l,D,N} e^{-n |x|}
  \quad \text{for all $x\in \R$ and $p\in \partial M$}.
\end{equation}
Obviously, $(x,y)^*$ maps $\bsS \big( \R\times \partial M \big) \cap 
\bcptC \big( (-\infty , \ln 3/2) \times \partial M\big)$ onto
the function space $\cJ^\infty\big(\partial M ,Y^{\frac 32}\big) 
\cap \cC_\txtcpt^\infty \big( Y^{\frac 32}\big)$.
\section{Global symbol calculus for pseudodifferential operators}
\label{Sec:GloSymCalPsiOp}
In this section, we briefly recall the global symbol for pseudodifferential 
operators which was introduced by Widom in \cite{Wid:CSCPO}
(see also \cite{FulKen:RPG,Pfl:NSRM}).
We assume that $(M_0,g)$ is a riemannian manifold (without boundary),
and that $\pi_E: E \rightarrow M_0$ and $\pi_F: F \rightarrow M_0$ are smooth vector 
bundle carrying a hermitian metric $\mu_E$ resp.~$\mu_F$. 
In later applications, $M_0$ will be the interior of a given manifold with 
boundary $M$.

Recall that there exists an open neighborhood $\Omega_0$ of the diagonal in 
$M_0\times M_0$ such that each two points $p,q\in M_0$ can be joined by a 
unique geodesic. 
Let $\alpha_0$ be a \textit{cut-off function} for $\Omega_0$ which means a smooth map  
$M_0\times M_0 \rightarrow [0,1]$ which has support in $\Omega_0$ and is equal to $1$ 
on a neighborhood of the diagonal. These data give rise to the map
\begin{equation}
  \label{eq:DefConIndLin}
  \Phi : M \times M \rightarrow TM, \quad
  (p,q) \mapsto 
  \begin{cases}
    \alpha_0 (p,q) \exp_p^{-1} (q) & \text{if $(p,q) \in \Omega_0$},\\
    0 & \text{else},
  \end{cases}
\end{equation}
which is called a \textit{connection-induced linearization} 
(\emph{cf.}~\cite{FulKen:RPG}). Next denote for $(p,q)\in \Omega_0$ by 
$\tau^E_{p,q} : E_p \rightarrow E_q$ the parallel transport in $E$ along the geodesic
joining $p$ and $q$. This gives rise to the map
\begin{equation}
  \label{eq:DefConIndParTrans}
  \tau^E : E \times M \rightarrow E, \quad
  (e,q) \mapsto 
  \begin{cases}
    \alpha_0 \big( \pi_E(e),q \big) \tau^E_{\pi^E(e),q} (e), & 
    \text{if $(\pi_E(e),q) \in \Omega_0$},\\
    0 & \text{else},
  \end{cases}
\end{equation}
which is called a \textit{connection-induced local transport} on $E$.
(\emph{cf.}~\cite{FulKen:RPG}).

Next let us define the symbol spaces $\cS^m \big( T^*M_0 ; \pi_{T^*M}^* E \big) $.
For fixed $m\in \R$ this space consists of all smooth sections
$a : T^*M_0 \rightarrow \pi_{T^*M_0}^* E$
such that in local coordinates $x: U \rightarrow \R^{\dim M_0}$ of 
over $U\subset M_0$ open and vector bundle coordinates
$(x,\eta): E_{|U} \rightarrow \R^{\dim M_0 + \dim_\R E}$ the following estimate
holds true for each compact $K\subset U$ and appropriate $C_K >0$ depending on $K$:
\begin{equation}
   \big\| \partial_x^\alpha \partial_\xi^\beta ( \eta \circ a ) (\xi) \big\|
   \leq C_K \, (1 + \| T^*x (\xi) \| )^{m - |\beta|} \quad
   \text{for all $\xi \in T^*_{|K}M_0$}.
\end{equation} 
Given a symbol $a \in \cS^m \big( T^*M_0 ; \pi_{T^*M}^* \Hom (E,F) \big) $
one defines now a pseudodifferential operator 
$\Op (a) \in \Psi^m \big( M_0; E,F)$ by 
\begin{equation}
  \begin{split}
  \label{Eq:DefOp}
   \big( \Op (a) & \, u \big) (p) := \\
   := \, & \frac{1}{(2\pi)^{\dim M}}\hspace{-2mm}
   \int\limits_{T_p M_0 \times T^*_p M_0} \hspace{-2mm}
   \alpha_0 (p, \exp v) \, 
   e^{-i \langle v , \xi \rangle} \, a ( p,\xi ) 
   \tau^E (u(\exp v) , p ) \, dv \, d\xi ,
  \end{split}
\end{equation}
where $u\in \Gamma^\infty_\text{\rm\tiny cpt} (E)$ and $p\in M_0$.
Moreover, there is a quasi-inverse, the symbol map 
$\sigma: \Psi^m \big( M_0 ; E,F \big) \rightarrow 
   \cS^m \big( T^*M_0 ; \pi_{T^*M}^* \Hom (E,F) \big)$
which is defined  by
\begin{equation}
\label{Eq:DefSym}
  \sigma (A) (\xi) (e) :=   A \Big( \alpha_0 (p ,{-}) \, \tau^E(e,{-}) 
  \, e^{i \langle \xi , \Phi (p,{-})\rangle} \Big) \big( \pi(\xi) \big), 
\end{equation}
where $p \in M_0, \; \xi \in T^*_pM_0,\; e \in E_p$.
It is a well-known result from global symbol calculus 
(\emph{cf.}~\cite{Wid:CSCPO,FulKen:RPG,Pfl:NSRM}) that the map $\Op$ 
maps $\cS^{-\infty} \big( T^*M_0 ; \pi_{T^*M}^* \Hom (E,F) \big) $ onto
$\Psi^{-\infty} \big( M_0 ; E,F  \big)$ and that up to these spaces,
$\Op$ and $\sigma$ are inverse to each other. 

\section{Classical \textup{b}-pseudodifferential operators}
\label{Sec:ClassbPseuOp}
Let us explain in the following the basics of the (small) calculus of  
\textup{b}-pseudodifferential operators on a manifold with boundary $M$. In our 
presentation, we lean on the approach \cite{Loy:DOB}. For more details 
on the original approach confer \cite{Mel:APSIT}.

In this section, we assume that $M$ carries a 
\textup{b}-metric\sind{bmetric@\textup{b}-metric} denoted by $\bmet$. 
Furthermore, let $\pi_E: E\rightarrow M$ and $\pi_F: F\rightarrow M$ 
be two smooth hermitian vector bundles over $M$, and fix metric connections 
$\nabla^E$ and $\nabla^F$. 
Then observe that by the Schwartz Kernel Theorem there is an isomorphism 
between bounded linear maps
\[
  A : \cJ^\infty \big( \partial M, M ; E \big) 
  \rightarrow \cJ^\infty \big( \partial M, M ; F' \big)'
\] 
and the strong dual 
$\cJ^\infty \big( \partial (M \times M), M\times M  ; E \boxtimes F' \big)'$,
where $\cJ^\infty \big( \partial M, M ; E \big) := \cJ^\infty \big( \partial M, M \big)
\cdot \cC^\infty (M;E)$. This isomorphism is given by
\begin{equation}
  A \mapsto K_A := \Big( \cJ^\infty \big( \partial M, M ; E \big) \hatotimes 
  \cJ^\infty \big( \partial M, M ; F' \big) \ni (u \otimes v) \mapsto
  \left\langle Au , v\right \rangle\Big), 
\end{equation}
where we have used that 
\[
  \cJ^\infty \big( \partial (M \times M), M\times M  ; E \boxtimes F' \big) =
  \cJ^\infty \big( \partial M, M ; E \big) \hatotimes 
  \cJ^\infty \big( \partial M, M ; F' \big)
\] 
with $\hatotimes$ denoting the completed bornological tensor product. 
The \textup{b}-volume form $\bvol$ associated to $\bmet$ gives rise to an embedding
\sind{pairing}
\nind{$< - , - >$}
\begin{equation}
  \label{Eq:EmbMultAlg}
  \begin{split}
    & \cM^\infty \big( \partial M, M ; E' \otimes F \big)  \hookrightarrow 
    \cJ^\infty \big( \partial (M \times M), M\times M  ; E \boxtimes F' \big)' , \\
    & k \mapsto \left( u \otimes v \mapsto \int_{M \times M} 
    \langle k (p,q) , u(p) \otimes v(q)\rangle \, d\big(\bvol \otimes \bvol \big)(p,q) 
    \, \right) ,
  \end{split}
\end{equation}
which we use implicitly throughout this work. 
In the formula for the embedding, $\langle -,-\rangle$ denotes the natural pairing 
of an element of a vector bundle with an element of the dual bundle over the same 
base point, $u,v$ are elements of $\cJ^\infty \big( \partial M, M ; E \big)$ and 
$\cJ^\infty \big( \partial M, M ; F' \big)$ respectively, and
$\cM^\infty \big( X, M ; E)$ denotes for $X\subset M$ closed the space of all
sections $u \in \Gamma^\infty (M\setminus X; E)$ such that in local coordinates 
$(y,\eta): \pi_E^{-1}(U) \rightarrow \R^{\dim M +\dim E}$ with $U\subset M$ open
one has for every compact $K\subset U$, $p \in K\setminus X$, and 
$\alpha \in \N^{\dim M}$ an estimate of the form
\begin{displaymath}
  \left\| \partial^\alpha_y \big(\eta \circ u \big) (p) \right\| \leq 
  C \frac{1}{\Big(d\big(y(p),y(X\cap U)\big)\Big)^\lambda},
\end{displaymath}
where $C>0$ and $\lambda >0$ depend only on the local coordinate system, 
$K$, and $\alpha$. The fundamental property of $\cM^\infty \big( X, M ; E)$
is that 
\begin{displaymath}
  \cJ^\infty (X,M) \cdot \cM^\infty \big( X, M ; E) \subset  \cJ^\infty (X,M; E).
\end{displaymath}

Note that the vector bundle $E$ (and likewise the vector bundle $F$) gives rise 
to a pull-back vector bundle 
$\operatorname{pr}_{\partial M}^* E_{|\partial M}$ on the cylinder, where 
$\operatorname{pr}_{\partial M} :\R\times \partial M
\rightarrow \partial M$ is the canonical projection. This pull-back vector bundle
will be denoted by $E$ (resp.~$F$), too. 
As further preparation we introduce two auxiliary functions 
$\psi :M\rightarrow [0,1]$ and $\varphi: M \rightarrow [0,1]$ on $M$ which are smooth 
and satisfy  
$\supp \psi \subset \subset M^1$, $\psi (p) = 1$ for $p\in M^{3/2}$, 
$\supp \varphi \subset \subset Y^1 $, and finally $\varphi (p) = 1$ for 
$p\in Y^{1/2}$. 
Such a pair of functions will be called a {\it pair of auxiliary cut-off functions}.

By a \emph{\textup{b}-pseudodifferential operator}
\sind{bpseudodifferentialoperator@\textup{b}-pseudodifferential operator} of order 
$m \in \R$ we now understand a continuous operator 
$A: \cJ^\infty (\partial M , M ; E) \rightarrow \cJ^\infty (\partial M, M;F')'$ 
such that for one (and hence for all) pair(s) of auxiliary cut-off functions the 
following is satisfied: 
\begin{enumerate}
\item[($\bPsi$1)] 
  The operator $(1-\varphi)A(1-\varphi)$ is a compactly supported  
  pseudodifferential operator of order $m$ in the interior $M^\circ$.
\item[($\bPsi$2)] 
  The operator $\varphi A\psi$ is smoothing. Its integral
  kernel $K_{\varphi A\psi}$ has support in $\supp \varphi \times M$  
  and lies in 
  $\cJ^\infty \big( \partial (M \times  M) , M \times M; E' \boxtimes F\big)$.
\item[($\bPsi$3)] 
  The operator $\psi A\varphi$ is smoothing. Its integral
  kernel $K_{\psi A\varphi}$ has support in $M \times \supp \varphi$ and lies in 
  $\cJ^\infty \big( \partial (M \times  M) , M \times M;  E' \boxtimes F\big)$.
\item[($\bPsi$4)] 
  Consider the induced operator on the cylinder 
  \[
  \begin{split}
    \widetilde A : \, & \bsS \big( \R\times \partial M; E \big)
    \rightarrow \bsS \big( \R\times \partial M ; F \big)', \\ 
    & u \mapsto \Big( (t,p) \mapsto \big[ (1-\psi)A(1-\psi) \, 
      \big( (x,\eta)^* u \big) \big] \big( (x,\eta)^{-1} (t,p) \big)   \Big) ,
   \end{split} 
  \]
  where $\bsS \big( \R\times \partial M; E \big) :=
  \bsS ( \R ) \hatotimes \cC^\infty \big( \partial M;E\big)$. 
  Denote by  
  $\widetilde a := \sigma (\widetilde A) 
  \in\cS^m\big( T^*(\R\times \partial M) ;\pi^*\Hom (E, F)\big)$  
  the complete symbol of $\widetilde A$ defined by Eq.~\eqref{Eq:DefSym}
  with respect to the product metric on $\R\times \partial M$. 
  Then the following conditions hold true:
  \begin{enumerate}
  \item[(i)]  
     Let $y$ denote local coordinates of $\partial M$, $(y,\xi)$ the 
     corresponding local coordinates of $T^* \partial M$, and $\tau$ the 
     cotangent variable of the cylinder variable $t \in \R$. Then the symbol
     $\widetilde a (t,\tau, y, \xi)$
     can be (uniquely) extended to an entire function in $\tau \in \C$ such that  
     uniformly in $t$, uniformly in a strip $|\im \tau | \leq R$ with $R>0$ 
     and locally uniformly in $y$
     \begin{displaymath}\hspace{-8mm}
       \left\| \partial_t^k\partial_\tau^l\partial_y^\alpha \partial_\xi^\beta 
       \, \widetilde a \right\| \leq C_{k,l,\alpha,\beta} 
       \left( 1+ |\tau| + \|\xi\|\right)^{m-l-|\beta|} 
     \end{displaymath}
     for $l\in \N$, $\beta \in \N^{\dim M -1}$.
  \item[(ii)] 
     There exist symbols $\widetilde a_k (\tau,y,\xi) \in 
     \cS^m \big(\C\times T^*\partial M ;\pi^*\Hom (E, F)\big)$, $k\in \N,$  
     and 
     $r_n (t,\tau,y,\xi)\in \cS^m\big( \R\times \C \times T^*\partial M) ;
     \pi^*\Hom (E, F)\big)$,
     $n\in \N$, which all are entire in $\tau$ and fulfill growth conditions
     as in (i) such that for every 
     $n \in \N$ the following asymptotic expansion holds:
     \begin{displaymath}\hspace{4mm}
       \widetilde a (t,\tau,y,\xi) = 
       \sum_{k=0}^n e^{kt} \widetilde a_k (\tau,y,\xi) + 
       e^{(n+1)t} \, r_n(t,\tau,y,\xi). 
     \end{displaymath}
  \item[(iii)]   
     The Schwartz kernel $K_{\widetilde B}$ of the operator
     $\widetilde B:= \widetilde A - \Op (\widetilde a)$ with 
     $\Op (\widetilde a)$ defined by Eq.~\eqref{Eq:DefOp}
     can be represented in the form
     \begin{displaymath}\hspace{-16mm}
       K_{\widetilde B} (t,p , t',p')= \int_\R \, e^{i(t-t')\tau} \,
       \widetilde b (t,\tau, p,p') \, d\tau 
     \end{displaymath}
     with a symbol 
     \begin{displaymath}\hspace{5mm}
     \widetilde b (t,\tau, p,p') \in \cS^{-\infty} 
     \big( T^*\R \times \partial M\times \partial M; \pi^* \Hom (E,F) \big)
     \end{displaymath}
     which is entire in $\tau$ and which for every $\tilde m \in \N$,
     $k,l\in \N$ and every pair of differential
     operators $D_p$ and $D'_{p'}$ on $\partial M$ 
     (acting on the variable $p$ resp.~$p'$) satisfies the following estimate
     uniformly in $t$, $p$, $p'$ and uniformly in a strip 
     $|\im \tau | \leq R$ with $R>0$ 
     \begin{displaymath}\hspace{-18mm}
       \left\| \partial_t^k\partial_\tau^l D_p D'_{p'}  
       \, \widetilde b \right\| \leq C_{\tilde m, k,l,D,D'} 
       \left( 1+ |\tau| \right)^{\tilde m}.
     \end{displaymath}
   \item[(iv)] 
     There exist symbols 
     \[
       \widetilde b_k (\tau,p,p') \in 
       \cS^{-\infty} \big( \C \times \partial M\times \partial M; 
       \pi^*\Hom (E, F)\big), 
     \]  
     for $k\in \N$ and symbols 
     \[\hspace{6mm}
      r_n (t,\tau,p,p')\in \cS^m\big( \R \times \C \times\partial M\times \partial M; 
      \pi^*\Hom (E, F)\big),
     \]
     for $n\in \N$
     which all are entire in $\tau$ and fulfill growth conditions
     as in (iii) such that for every 
     $n \in \N$ the following asymptotic expansion holds:
     \begin{displaymath}\hspace{4mm}
       \widetilde b (t,\tau,p,p') = 
       \sum_{k=0}^n e^{kt} \, \widetilde b_k (\tau,p,p') + 
       e^{(n+1)t} \, r_n(t,\tau,p,p'). 
     \end{displaymath}
  \end{enumerate}
\end{enumerate}
If in addition to the above conditions the operators $(1-\varphi)A(1-\varphi)$ 
and $\widetilde A$ are both classical pseudodifferential operators, 
then $A$ is a classical \textup{b}-pseudodifferential operator of order $m$. 
We denote the space of classical \textup{b}-pseudodifferential operators on $(M,\bmet)$
of order $m$ between $E$ and $F$ by $\bPsi^m (M;E,F)$, and put as usual
$\bPsi^\infty (M;E,F) := \bigcup_{m\in \Z} \bPsi^m (M;E,F)$.
It is straightforward (though somewhat tedious) to check that 
$\bPsi^\infty (M;E) := \bPsi^\infty (M;E,E)$ even forms an algebra. 
Obviously, $\bPsi^\infty (M;E,F)$ contains as a natural subspace the 
space $\bdiff (M;E,F)$ of all \textup{b}-differential operators on $M$ from $E$ to $F$
which means of all local classical \textup{b}-pseudodifferential operators. 
The following is immediate to check.
\begin{proposition}
\label{prop:DefEqbDiff}
  Let $A \in \bPsi^\infty \big(M ;E,F\big)$. Using the notation from above 
  the following propositions are then equivalent:
  \begin{enumerate}
  \item 
     $A \in \bdiff \big( M; E,F \big)$.
  \item
     The operators $(1-\varphi)A (1 -\varphi)$ and $\widetilde A$ are 
     differential operators, and both the operators $\varphi A \psi$ and
     $\psi A \varphi$ vanish. 
  \item
     The operator $A$ acts as a differential operator over the interior, i.e.
     as a local operator on $\Gamma^\infty \big (E_{|M^\circ} \big)$. In addition, 
     over the cylinder $(-\infty,0)\times \partial M$ 
     the operator $\widetilde A$ can be written locally in the form
     \begin{equation}
     \label{eq:DefEqbDiff}
     \widetilde A = 
     \sum_{j + |\alpha| \leq \operatorname{ord} A} a_{j,\alpha} \partial_y^\alpha \partial_x^j ,
     \end{equation}
     where $a_{j,\alpha} \in \bcC\!\!\big( (-\infty,0) \times U\big)$, $U\subset \partial M$ 
     open and $y:U \rightarrow \R^{\dim M -1}$ are local coordinates of $\partial M$.
  \end{enumerate}
\end{proposition}

Over a cylinder $(-\infty,c)\times N$ with $c\in \R$ and $N$ a compact manifold,
we sometimes use the notation $\bcptPsi^\infty \big((-\infty,c)\times N;E,F\big)$ 
to denote the space of all pseudodifferential operators in
$\bPsi^\infty \big((-\infty,c)\times N;E,F\big)$ having support in some 
cylinder $(-\infty,c-\varepsilon ]\times N$ with $\varepsilon >0$.
We also put 
\begin{equation}
  \label{eq:defbcptdiff}
  \begin{split}
    \bcptdiff & \big((-\infty,c)\times N;E,F\big) := \\
    & \bdiff \big((-\infty,c)\times N;E,F\big) \cap 
    \bcptPsi^\infty \big((-\infty,c)\times N;E,F\big).
  \end{split}
\end{equation}
Note that in condition $(\bPsi 4)$ above, the operator $\widetilde A$
is an element of the space $\bcptPsi^\infty \big((-\infty,3/2)\times \partial M;E,F\big)$.

Throughout this work, we also need the \textup{b}-versions of Sobolev-spaces. 
The \textup{b}-Sobolev space $\bSob^m (M;E)$ is defined for $m \in \N$ by
\begin{equation}
  \label{Eq:DefbSob}
  \bSob^m (M,E) := \big\{ u \in L^2 (M,E) \mid D u \in L^2 (M,E)
  \text{ for all $D\in \bdiff^m (M,E)$}  \big\}.
\end{equation}
For the definition of $\bSob^m (M,E)$ for arbitrary $m\in \R$ 
we refer the reader to \cite{Mel:APSIT}. The following result 
is straightforward.
\begin{proposition}
\label{Prop:PropbSob}
  Let $A \in \bPsi^l (M;E,F)$ be a \textup{b}-pseudodifferential operator.
  Then the following holds true: 
  \begin{enumerate}
  \item 
   $A$ has a natural extension 
   \begin{equation}
   \label{eq:bPsibSob}
     A : \bSob^m (M;E) \rightarrow \bSob^{m-l} (M;F), 
   \end{equation}
   which we denote by the same symbol like the original operator.
 \item
   The \textup{b}-Sobolev-space $\bSob^1 (M,E)$ is the natural {\domain} of any elliptic 
   first order \textup{b}-pseudodifferential operator acting on sections of $E$.
 \item 
   If $A$ has order $l=0$, then $A$ is bounded.   
 \end{enumerate}
\end{proposition}
\section{Indicial family}
\label{s:IndFam}
\sind{indicial family}
Assume $A\in \bPsi^m (M;E,F)$. Denote by $\widetilde A$ and $\widetilde a$ 
the induced operator and its complete symbol on the cylinder 
$\R \times \partial M$ as above in condition $(\bPsi 4)$. Consider the 
zeroth order term $\widetilde a_0$ in the asymptotic expansion
$(\bPsi 4) (ii)$ and put for $\tau \in \C$, $u \in \Gamma^\infty (\partial M ; E)$
and $p\in \partial M$
\begin{align}
\cI (A) & (\tau) u (p)  :=  \Op \big(\widetilde a_0 (\tau, - )\big) u (p) =\label{eq:DefIndFam} \\
   = \, & \frac{1}{(2\pi)^{\dim M -1}} \hspace{-6mm}
   \int\limits_{T_p\partial M \times T^*_p\partial M} 
   \hspace{-2mm} \alpha_0 (p, \exp v) \, 
   e^{-i \langle v , \xi \rangle} \, \widetilde a_0(\tau,p,\xi ) \, 
   \tau^E \big(  u (\exp v) , p \big) \, dv \, d\xi ,\nonumber
\end{align}
where, as explained in Section \ref{Sec:GloSymCalPsiOp}, 
$\alpha_0 : M \times M \rightarrow [0,1]$ is a 
cut-off function vanishing outside the injectivity radius and $\tau^E$ is 
a connection induced parallel transport on $E$. One thus obtains an entire 
family $\cI (A)$ of pseudodifferential operators on 
the boundary $\partial M$ which is called the \textit{indicial family}\sind{indicial family}
of $A$. The indicial family plays a crucial role in deriving the Atiyah--Patodi--Singer
index formula within the \textup{b}-calculus (\emph{cf.}~\cite{Mel:APSIT}).

\chapter{The \textup{b}-Analogue of the Entire Chern Character}
\label{s:b-Chern-character}

After discussing in Section \ref{s:btrace} the \textup{b}-trace in the 
context of a manifold with cylindrical ends, we digress in 
Section \ref{s:McK-S} to establish a cohomological analogue of the 
well-known McKean--Singer formula in the framework of relative cyclic 
cohomology for the pseudodifferential  \textup{b}-calculus, and then employ 
it to recast Melrose's approach to the proof of the Atiyah--Patodi--Singer 
index theorem. In Section \ref{s: b-trace formula}  we establish an effective 
formula for the \textup{b}-trace, which will be used later in the paper. 
The rest of this chapter is devoted to a reformulation of Getzler's version 
of the entire Connes--Chern character in the setting of \textup{b}-calculus, 
formulated in terms of relative cyclic cohomology.

\section{The \textup{b}-trace}
\label{s:btrace}

From now on we assume that $M$ is a compact manifold with boundary,
that $r: Y \rightarrow [0,2)$ is a boundary defining function, 
and that $\bmet $ is an exact \textup{b}-metric on $M$. These are the main
ingredients of the \textup{b}-calculus, which we will use in what follows
(see Sections \ref{App:bdefbmet} to \ref{s:IndFam} for basic definitions 
and the monograph \cite{Mel:APSIT} for further material on the 
\textup{b}-calculus).  

Before we can construct the \textup{b}-analogue of the entire Chern character 
we have to recall here however the definition of the \textup{b}-\emph{trace}
(\emph{cf.}~\cite{Mel:APSIT}), since this notion plays an essential role in our work. 
It will often be convenient to choose cylindrical coordinates
(see Figure \ref{fig:MCylinder} on page \pageref{fig:MCylinder})
$(x,\eta): Y^\circ \rightarrow \, (-\infty, \ln 2 ) \times \partial M$ 
over a collar $Y \subset M $ with a boundary defining function
$r : Y \rightarrow [0,2)$ (see Section \ref{App:bdefbmet} for details and notation).
When using these coordinates, we view the interior $M^\circ$  
as a manifold with cylindrical ends, and have in this picture
$M^\circ \cong (-\infty,0]\times \pl M\cup_{\pl M} \overline{M^1}$. 
All explicit calculations will be done in cylindrical coordinates
as explained in the previous sections. 
Let $E$ be a smooth hermitian vector bundle over $M$.  
Whenever convenient,
we will tacitly identify elements of $\Gamma^\infty\bigl((-\infty,0]\times \pl M;E\bigr)$,
the sections of $E$ over $(-\infty,0]\times \partial M$, with
$\Gamma^\infty(\pl M;E\rest{\pl M})$--valued smooth functions on $(-\infty,0]$,
\emph{i.e.} elements of $\cC^\infty\bigl((-\infty,0],\Gamma^\infty(\pl M;E\rest{\pl M})\bigr)$,
in the obvious way; \emph{cf.} Section \ref{Sec:ClassbPseuOp}. 
Accordingly, we define for $u\in \Gamma^\infty_{\textup{cpt}}\bigl((-\infty,0]\times \pl M;E\bigr)$ 
the Fourier transform in the cylinder variable, $\hat u(\gl)\in\Gamma^\infty(\pl M;E\rest{\pl E})$, by
\begin{equation}
     \hat u(\gl,p):=\int_{-\infty}^\infty e^{-ix\gl} u(x,p) dx.
\end{equation}

Now assume that $A\in \bPsi^{-\infty}(M;E)$
is a smoothing \textup{b}-pseudodifferential operator. 
In general,  $A$ is not trace class in the usual sense. However, since $A$ 
has a smooth Schwartz kernel it  is locally trace class, in the sense that 
$\psi A\varphi$ is trace class for any pair of smooth functions 
$\psi,\varphi : M \rightarrow \R$ having compact support in $M^\circ$. 
Using the notation from Section \ref{Sec:ClassbPseuOp}, let $\widetilde A$
be the operator induced on the cylinder $(-\infty , 0] \times \partial M$. 
We define then an operator valued symbol
$A_\partial (x,\lambda): \Gamma^\infty (E_{|\partial M}) \rightarrow \Gamma^\infty (E_{|\partial M})$
as follows: 
for $v\in \Gamma^\infty(E\rest{\pl M})$ put
\begin{equation}\label{eq:DefIndFamAlt}
          \bigl(A_\pl(x,\gl)v\bigr)(p):=\Bigl(e^{-ix\gl}\widetilde A\bigl(e^{i\cdot\gl}\otimes v\bigr)\Bigr)(x,p).
\end{equation}
For $v\in\Gamma^\infty(E\rest{\pl M})$ we have $(e^{i\cdot \gl}\otimes v)^\wedge=2\pi \delta_\gl\otimes v$, 
hence for 
$u \in \Gamma^\infty_{\textup{cpt}} \big( (-\infty , 0] \times \partial M;E \big)$
one then obtains
\begin{equation}
\label{eq:ftbdrop}
  \begin{split}
  \bigl(\widetilde A u \bigr) (x,p) & = 
  \frac{1}{2\pi } \int_{-\infty}^{\infty} e^{i x \lambda} 
  \big( A_\partial (x,\lambda ) \hat u(\lambda, -) \big)(p) \, d\lambda = \\
  & = \frac{1}{2\pi} \int_{-\infty}^{\infty}\int_{-\infty}^{\infty} 
  e^{i (x-\tilde x) \lambda} 
  \big( A_\partial (x,\lambda ) u(\tilde x, -) \big)(p) \, d\tilde x \, d\lambda.
  \end{split}
\end{equation}
We note in passing that $A_\partial (x,\lambda)$ can also be constructed
from the global symbol $\widetilde a$ as  
$A_\partial (x,\lambda) = \Op \big( \widetilde a (x,\lambda,-) \big) $
(\emph{cf.}~Sections \ref{Sec:GloSymCalPsiOp} and \ref{Sec:ClassbPseuOp}).
For the properties of $\widetilde a$ see \cite[Sec. 2.2]{Loy:DOB} and Section
\plref{Sec:ClassbPseuOp}. In the small \textup{b}-calculus $\widetilde a(x,\gl,-)$ and hence 
$A_\partial (x,\lambda)$ are entire in $\gl$ while in the full \textup{b}-calculus
they are meromorphic \cite{Mel:APSIT}. 
In any case one has
\begin{equation}\label{eq:asymptotics-b-symbol}
      A_\partial (x,\gl) = \cI(A)(\gl) +O(e^x),\quad x\to - \infty,
\end{equation}
with a family $\cI(A)(\gl), \gl\in\R,$
of classical pseudodifferential operators on $\pl M$ in the parameter dependent
calculus; 
\emph{cf. e.g.}~\cite[Sec. 2.1]{LesMosPfl:RPC} for a brief summary of the parameter dependent 
calculus.
It turns out that the operator valued function $\, \gl\in\R \mapsto \cI(A)(\gl)$
is exactly the \semph{indicial family} of $A$ as defined in Section \ref{s:IndFam}.
The reason is that in terms of the global symbol
$\widetilde a$ one has $\cI(A)(\gl) = \Op \big( \widetilde a_0 (\lambda,-) \big)$,
where $\widetilde a_0 (\lambda,-)$ is the first term in the asymptotic expansion
of $\widetilde a$  with respect to $x\rightarrow -\infty$.

Denote by $k(x,\tilde x)$, $x,\tilde x \ge 0$, the 
$\sL(L^2(\pl M;E\rest{\pl M}))$-valued kernel of $\widetilde A$. 
In terms of $A_\partial (x,\gl)$ the kernel $k(x,\tilde x)$ is given by
\begin{equation}\label{eq:KernelOnCylinder}
       k(x,\tilde x)=
       \frac{1}{2\pi} \int_{-\infty}^\infty e^{i(x-\tilde x )\gl} A_\partial (x,\gl) d\gl.
\end{equation}
Hence, as $R\to \infty$ one has in view of \eqref{eq:asymptotics-b-symbol}
\begin{equation}\label{eq:b-trace-explanation}
   \begin{split}
   \Tr(A&\rest{\{x\ge -R\}}) = \\
                 = \, &\Tr(A\rest{M^1})+\int_{-R}^0\Tr_{\pl M}(k(x,x))\;dx\\
                 = \, &\Tr(A\rest{M^1})+\frac{1}{2\pi}\int_{-\infty}^0\int_{-\infty}^\infty \Tr_{\pl M}\bigl( A_\partial (x,\gl)-\cI(A)(\gl)\bigr) \, d\gl \, dx\\
                 &+R \frac{1}{2\pi}\int_{-\infty}^\infty 
                 \Tr_{\pl M}\bigl(\cI(A)(\gl)\bigr) \, d\gl+O(e^{-R}),\quad R\to\infty.
   \end{split}
\end{equation}
The \emph{finite part} of this expansion is called the {\btrace} of $A$:
\sind{partie finie}
\begin{equation}
\label{eq:b-trace-def}
   \bTr(A):=\Tr(A\rest{M^1})+\frac{1}{2\pi}
   \int_{-\infty}^0\int_{-\infty}^\infty \Tr_{\pl M}\bigl(A_\partial(x,\gl)-
   \cI(A)(\gl)\bigr) \, d\gl \, dx.
\end{equation}
Hence
\begin{equation}\label{eq:b-trace-def-a}
\begin{split}
   \Tr(A&\rest{\{x\ge -R\}})\\
        &=\bTr(A)+R \frac{1}{2\pi}\int_{-\infty}^\infty 
        \Tr_{\pl M}\bigl(\cI(A)(\gl)\bigr) d\gl+O(e^{-R}),\quad R\to\infty.
\end{split}
\end{equation}
Its name notwithstanding, the {\btrace}  is not a trace. One has, however,
the following crucial formula.

\begin{proposition}[{\cite[Prop.~5.9]{Mel:APSIT}, \cite[Thm.~2.5]{Loy:DOB}}]
\label{p:b-trace-defect}
Assume that $A\in{\bpdo^m (M;E)}$ and $K\in\bpdo^{-\infty}(M;E)$. Then 
\begin{equation}
  \bTr(AK-KA)=\frac{-1}{2\pi i} \int_{-\infty}^\infty 
  \Tr_{\pl M}\left( \frac{d\cI(A)(\gl)}{d\gl} \, \cI(K)(\gl) \right)d\gl.
\end{equation}
\end{proposition}
%
\begin{proof}[Proof in a special case] It is instructive to prove this in the special case that $A$ is a Dirac operator $\dirac$.
We will see in Section \ref{s:b-Clifford-Dirac} below that on the cylinder $\dirac$ takes the form
$\dirac=\Gammabdy \frac{d}{dx}+\diracbdy$ and that $\cI(\dirac)(\gl)=i\Gammabdy \gl+\diracbdy$.

After choosing cut--off functions  w.l.o.g. we may assume that $K$ is supported in the interior of 
the cylinder and given by 
\begin{equation}
    \begin{split}
      (Ku)(x)&=\frac{1}{2\pi}\int_{-\infty}^\infty e^{ix\gl} k(x,\gl) \hat u(\gl)d\gl\\
           &=\frac{1}{2\pi}\int_{-\infty}^\infty\int_{-\infty}^\infty  e^{i(x-y)\gl} k(x,\gl) u(y) dyd\gl
         \end{split}
\end{equation}
with an operator valued symbol $k(x,\gl)$. Then
\begin{equation}
\begin{split}
    (\dirac Ku)(x)&= \frac{1}{2\pi}\int_{-\infty}^\infty e^{ix\gl}(i\gl\Gammabdy+\diracbdy)k(x,\gl) \hat u(\gl)d\gl\\
          &\qquad +\frac{1}{2\pi}\int_{-\infty}^\infty e^{ix\gl} \Gammabdy\partial_x k(x,\gl) \hat u(\gl) d\gl.
\end{split}
\end{equation}
Furthermore, since $(\dirac u)^\wedge=\bigl(i\Gammabdy\gl+\diracbdy)\hat u$ we have
\begin{equation}
    (K\dirac) u(x)= \frac{1}{2\pi}\int_{-\infty}^\infty e^{ix\gl}k(x,\gl)(i\gl\Gammabdy+\diracbdy) \hat u(\gl)d\gl.
\end{equation}
Consequently
\begin{equation}
    \Tr_{\pl M}\bigl((\dirac K-K\dirac)(x,x)\bigr) =\frac{1}{2\pi}\int_{-\infty}^\infty \Tr_{\pl M}\bigl(\Gammabdy \partial_x k(x,\gl)\bigr)d\gl,
\end{equation}
and hence, since by assumption $K$ is supported in $(-\infty,0)\times \pl M$  
and taking \eqref{eq:asymptotics-b-symbol} into account, we find
\begin{align}
    \int_{-R}^0\Tr_{\pl M}\bigl((&\dirac K-K\dirac)(x,x)\bigr) dx =\frac{-1}{2\pi} \int_{-\infty}^\infty \Tr_{\pl M}\bigl(\Gammabdy k(-R,\gl)\bigr)d\gl\\
         &\overset{R\to \infty}{\longrightarrow}  \frac{-1}{2\pi i}\int_{-\infty}^\infty \Tr_{\pl M}\Bigl(\frac{d\cI(\dirac)(\gl)}{d\gl} \cI(K,\gl) \Bigr)d\gl.\qedhere
\end{align}
\end{proof}
\noindent

Eq.~\eqref{eq:b-trace-def-a} immediately entails the following result.
\begin{corollary}
\label{prop:smotrcl}
  Let $A\in{\bpdo^m} (M;E)$ be a classical \textup{b}-pseudodifferential 
  operator of order $m < \dim M$. 
  If the indicial family\sind{indicial family} $\cI(A) $ vanishes, then $A$ is trace class,  and 
  \[
    \Tr (A) = \bTr(A ).
  \]
\end{corollary}

\section{The relative McKean--Singer formula and the \textup{APS} Index Theorem}
\label{s:McK-S}
\newcommand{\taubar}{\overline\tau}
In the introduction to \cite{Mel:APSIT}, the author explained in detail his elegant
approach to the {\APS} index theorem, based on the \textup{b}-calculus. 
The relative cohomology point of view allows to make this approach even more
appealing. 
Indeed, we will show that the {\APS}  index can be obtained as the pairing between a
natural relative cyclic $0$-cocycle, one of whose components is the {\btrace},
and a relative cyclic $0$-cycle constructed out of the heat kernel.\sind{heat kernel}
This pairing leads  in fact to a relative version of the McKean--Singer formula. 
\sind{pairing}\sind{McKean--Singer formula!relative}\sind{relative!McKean--Singer formula}

We start, a bit more abstractly, by considering an exact sequence of algebras
\begin{equation}\label{eq:ML200909141}
0\longrightarrow \cJ \longrightarrow \cA \overset{\sigma}{\longrightarrow} \cB \longrightarrow 0;
\end{equation}
$\cA, \cB$ are assumed to be unital, $\sigma$ is assumed to be a unital homomorphism. 
Let $\tau$ be a \semph{hypertrace} on $\cJ$, \emph{i.e.} $\tau$ satisfies
\begin{equation}\label{eq:ML200909111}
    \tau(xa)=\tau(ax) \text{ for } a\in\cA, x\in\cJ .
\end{equation}
Let $\taubar:\cA\longrightarrow \C$ be a
linear extension (\semph{regularization}) of $\tau$  
to $\cA$, which is not assumed to be tracial. Nevertheless,
 $\taubar$ induces a cyclic 1-cocycle on $\cB$ as follows:
\begin{equation}\label{eq:ML200911261}
     \mu(\sigma(a_0),\sigma(a_1)):= \taubar([a_0,a_1]).
\end{equation}
Because of Eq. \eqref{eq:ML200909111},
$\mu$ does indeed depend only on $\sigma(a_0),\sigma(a_1)$. Moreover, the pair $(\taubar,\mu)$
is a relative cyclic cocycle. Namely, in the notation of \eqref{Eq:DefCoBdrRelMixDer},
Eq.~\eqref{eq:ML200911261} translates into $\widetilde b (\taubar,\mu)=0$.

The relevant example for this paper is the exact sequence
\begin{equation}
0\longrightarrow \bPsi^{-\infty}_{\textrm{tr}}(M;E)\longrightarrow \bPsi^{-\infty}(M;E)^+
\overset{\cI}{\longrightarrow} \sS\longrightarrow 0,
\end{equation}
where $M$ is a compact manifold with boundary equipped with an exact \textup{b}-metric.
The operator trace $\Tr$ on $\bPsi^{-\infty}_{\textrm{tr}}(M;E) = \Ker \cI$ satisfies 
\eqref{eq:ML200909111}, and the
{\btrace}  $\bTr$ provides its linear extension to  $\bPsi^{-\infty}(M;E)^+$.  

The indicial map $\cI$ realizes the quotient algebra 
\[\sS=\bPsi^{-\infty}(M;E)^+/\bPsi^{-\infty}_{\textrm{tr}}(M;E)\]
as a subalgebra of the unitalized algebra
$\sS(\R,\Psi^{-\infty}(\partial M;E))^+$ of Schwartz functions with
values in the smoothing operators $\Psi^{-\infty}(\partial M;E)$ on the
boundary. That $\sS$ does not equal  $\sS(\R,\Psi^{-\infty}(\partial M;E))^+$
(which would be nicer and more intuitive here) has to do with the fine print
of the definition of the $b$--calculus which requires e.g. the analyticity
of the indicial family. For our discussion here these details are not
relevant and hence we will not elaborate further on them.

Going back to the abstract sequence \eqref{eq:ML200909141}
assume now that the algebra $\cA$ is represented as bounded operators on some Hilbert space $\sH$
and that $\cJ\subset \sL^1(\sH)$ consists of trace class operators. Let $D$ be a 
self-adjoint unbounded operator affiliated with $\cA$, \emph{i.e.}
bounded continuous functions of $D$ belong to $\cA$. 
Furthermore, we assume that we are in a graded (even) situation
and denote the grading operator by $\ga$. 
Finally we assume that $D$ is a Fredholm operator and that the orthogonal projection $P_{ \Ker D}\in \cJ$.

We note that in the case of a Dirac operator $\dirac$ on the \textup{b}-manifold $M$ it is well-known that
$\dirac$ is Fredholm if and only if the tangential operator $\diracbdy$ (see Section \plref{s:b-Clifford-Dirac}
and Eq.~\eqref{eq:essspecbd} below) is invertible.

Define
\begin{align}
	A_0(t)&:=\alpha D e^{-\frac 12 t D^2},\\
	A_1(t)&:= \int_t^\infty D e^{(\frac{t}{2}-s) D^2}\, ds .
\end{align}
Since $D$ is affiliated with $\cA$,
$A_j(t)\in\cA$ for $t>0$ and $j=1,2$.

\subsection{Case 1: $e^{-t D^2}\in\cJ$, for $t>0$}
Under this assumption $A_j(t)\in \cJ$, for $t>0$.
Moreover, $e^{-t D^2}\in C_0^\gl(\cJ)=C_0(\cJ)/((1-\gl)C_0(J))$
defines naturally a class in $H_0^\gl(\cJ)$. 

\begin{lemma}\label{l:McKeanSingerHomological} The class of $\ga e^{-t D^2}$ in $H_0^\gl(J)$ equals
	that of $\ga P_{ \Ker D}$. In particular, it is independent of $t$.
\end{lemma}
\begin{proof} We calculate
\begin{equation}\label{eq:ML200909143}
\begin{split}
b (A_0\otimes A_1)&= 2 \ga \int_t^\infty D^2 e^{-s D^2}\, ds\\
&= 2\ga \bigl (e^{-t D^2} - P_{ \Ker D} \bigr),
\end{split}
\end{equation}
	proving that $\ga e^{-t D^2}$ and $\ga P_{ \Ker D}$ are homologous.
\end{proof}

As an immediate corollary one recovers the classical McKean--Singer formula.
\sind{McKean--Singer formula} 
Indeed, since the trace $\tau$ defines a class in $H^0_\gl(\cJ)$ one finds 
\begin{equation}
\ind D = \Tr(\ga P_{ \Ker D})=\langle \tau,\ga P_{ \Ker D}\rangle = \langle \tau,\ga e^{-t D^2}\rangle=\Tr(\ga e^{-t D^2}).
\end{equation}

\subsection{Case 2: The general case} The heat operator
$e^{-t D^2}$ gives a class in $H_0^\gl(\cA)$. The pairing of $e^{-t D^2}$
with $\taubar$ cannot be expected to be independent of $t$ since $\taubar$ is not a trace.
It is however a component of the relative cyclic $0$-cocycle $(\taubar,\mu)$. 
Therefore, we are led to construct a relative cyclic homology
class from $e^{-t D^2}$. By Eq.~\eqref{Eq:DefBdrRelMixDer}
the relative cyclic complex is given by 
\begin{equation}\label{eq:ML200909144}
	C_n^\gl(\cA,\cB):=C_n^\gl(\cA)\oplus C_{n+1}^\gl(\cB),\quad \tilde b:=
	\begin{pmatrix} b & 0 \\ -\sigma_* & -b \end{pmatrix}.
\end{equation}
From \eqref{eq:ML200909143} we infer
\begin{equation}
\sigma( \ga e^{-t D^2})= \sigma (\ga e^{-t D^2}-\ga P_{ \Ker D})= \frac 12 b \bigl(\sigma(A_0)\otimes \sigma(A_1)\bigr),
\end{equation}
hence
\begin{equation}
\widetilde b { \ga e^{-t D^2} \choose -\frac 12 \sigma(A_0)\otimes \sigma(A_1)}=0,
\end{equation}								
\emph{i.e.}~the pair 
$\operatorname{EXP}_t(D):=
\bigl(\alpha e^{-t D^2}, -\frac 12\sigma(A_0)\otimes\sigma(A_1)\bigr)$ is a 
relative cyclic homology class. Furthermore, since
\begin{equation}
\widetilde b { \frac 12 A_0\otimes A_1 \choose 0}= {\ga (e^{-t D^2}-P_{ \Ker D}) \choose -\frac 12 \sigma(A_0)\otimes \sigma(A_1)},
\end{equation}
the class of $\operatorname{EXP}_t(D)$ in $H_0^\gl(\cA,\cB)$ equals that of
the pair $(\ga P_{ \Ker D},0)$ which, via \excision, corresponds to the class
of $\ga P_{ \Ker D}\in HP_0(\cJ)$. We have thus proved the following result.

\begin{lemma}\label{l:McKeanSingerHomologicalRelative}
The class of the pair $\big(\ga e^{-t D^2}, -\frac 12 \sigma(A_0)\otimes \sigma(A_1)\big)$ in $HC_0^\gl(\cA,\cB)$
equals that of $(\ga P_{ \Ker D},0)$. In particular, it is independent of $t$.
\end{lemma}

Pairing with the relative $0$-cocycle $(\taubar,\mu)$ we now obtain the following 
relative version of the McKean--Singer 
formula:
\begin{equation}
\begin{split}
\ind D&= \Tr(\ga P_{ \Ker D})=\langle (\taubar,\mu),
\bigl( \ga e^{-t D^2}, -\frac 12 \sigma(A_0)\otimes \sigma(A_1)\bigr)\rangle\\
&= \taubar ( \ga e^{-t D^2}) - \frac 12 \mu(\sigma(A_0),\sigma(A_1))\\
& = \taubar (\ga e^{-t D^2}) - \frac 12 \taubar ([A_0,A_1]).
\end{split}
\end{equation}
Once known, this identity can also be derived quite directly. Indeed, since 
$[A_0,A_1]=2\ga \int_t^\infty D^2 e^{-s D^2}ds$,  
\begin{equation}
\begin{split}
 \taubar (\ga e^{-t D^2}) - \frac 12 \taubar ([A_0,A_1])
  &=\taubar (\ga e^{-t D^2})-\taubar\bigl( \ga \int_t^\infty D^2 e^{-s D^2} \, ds \bigr)\\
& =\taubar (\ga e^{-t D^2}) +\int_t^\infty\frac{d}{ds}\taubar\bigl(\ga e^{-s D^2}\bigr)\, ds\\
& =\lim_{s\to\infty} \taubar\bigl(\ga e^{-s D^2}\bigr).
\end{split}
\end{equation}

Let us show that in the case of the Dirac operator on a \textup{b}-manifold
the second summand is nothing but the $\eta$-invariant of the tangential operator.  
Indeed, in a collar of the boundary $\dirac$ takes the form
\[
\dirac =\begin{pmatrix} 0 & -\frac{d}{dx}+\Abdy\\ \frac{d}{dx}+\Abdy & 0 \end{pmatrix}=:
\Gammabdy \frac{d}{dx}+\diracbdy,\quad \diracbdy=\begin{pmatrix} 0 & \Abdy \\ \Abdy & 0 \end{pmatrix}.
\]
Hence, one calculates using  Proposition \plref{p:b-trace-defect}
\begin{equation}
\begin{split}
\frac 12 \mu(&\cI(A_0,\gl),\cI(A_1,\gl)) \\
&= \frac{-1}{4\pi i} \int_{-\infty}^\infty \Tr_{\pl M} \bigl(\frac{d \cI(A_0,\gl)}{d\gl} \cI(A_1,\gl)\bigr) d\gl\\ 
&= \frac{-1}{4\pi}\int_{-\infty}^\infty\Tr_{\pl M}\bigl(\ga \Gammabdy \int_t^\infty (i\gl\Gammabdy+\dirac_\pl) e^{-s(\dirac_\pl^2+\gl^2)}\, ds\,\bigr) d\gl\\
&= \frac{-1}{4\sqrt{\pi}} \int_t^\infty \frac{1}{\sqrt{s}}\Tr_{\pl M} \bigl(\begin{pmatrix} -A & 0 \\ 0 & -A\end{pmatrix} e^{-s A^2}\bigr)ds\\
     &= \frac{1}{2\sqrt{\pi}}\int_t^\infty  \frac{1}{\sqrt{s}} \Tr_{\pl M} \bigl( A e^{-s A^2}\bigr)ds=:
{ \frac 12 \eta_t(A)}.
\end{split}
\end{equation}
Thus if $\dirac$ is Fredholm we have for each $t>0$
\begin{equation}
\ind \dirac = \bTr(\ga e^{-t \dirac^2})- \frac 12 \eta_t(A),
\end{equation}
and taking the limit as $t\searrow 0$ gives the {\APS}  index theorem in the \textup{b}-setting.

\section{A formula for the \textup{b}-trace} 
\label{s: b-trace formula}

In this section we give an explicit formula for the {\btrace},
based on an observation of \textnm{Loya} \cite{Loy:DOB}, which provides a
convenient tool for subsequent computations.

We first briefly review the Hadamard partie finie integral
in the special case of \textup{b}-functions. Let $f\in\bcC((-\infty,0])$. 
From the asymptotic expansion (see~Eq.~\eqref{eq:ML20090219-3})
\begin{equation}\label{eq:ML20090122-1}
   f(x)\sim_{x\to -\infty} f_0^-+f_1^- e^x + f_2^- e^{2x}+....
\end{equation}
we infer
\begin{equation}\label{eq:ML20090122-2}
\int_{-R}^0 f(x) dx = f_0^- \; R + c + O(e^{-R}),\quad R\to -\infty.
\end{equation}
The \emph{partie finie integral} of $f$ is then defined to be the constant
term in the asymptotic expansion \eqref{eq:ML20090122-2}, i.e.
\sind{partie finie}
\nind{PartieFinie@Pf$\intop$} 
\begin{equation}\label{eq:defPfInt}
  \int_{-R}^0 f(x) dx =: f_0^- \; R + \pfint_{-\infty}^0 f(x) dx + O(e^{-R}),\quad 
  R\to -\infty.
\end{equation}
The definition of the partie finie integral has an obvious extension to 
\textup{b}-functions on manifolds with cylindrical ends (see Section
\plref{App:bdefbmet}).\sind{cylindrical ends} 
Because of its importance, we single it out as a definition-proposition.

\begin{defprop}\label{defprop15}
Let $M^\circ$ be a riemannian manifold with cylindrical ends
  and $M$ the (up to diffeomorphism) unique compact manifold with boundary 
  having $M^\circ$ as its interior.  
  For a function $f \in \bcC (M^\circ)$ one has 
  \begin{displaymath}
    \int_{x \geq -R} f \, d\vol =: c \log R + \int_{\tb M} f \, d\vol + O(e^{-R})
    \quad \text{as $R\rightarrow \infty$}.
  \end{displaymath}
This means that $\int_{\tb M} f \, d\vol$ is the finite part in the asymptotic 
expansion of $\int_{x \geq -R} f \, d\vol$ as $R\rightarrow \infty$. More 
generally,  if $\omega \in \bOmega^m (M)$ is a (top degree) \textup{b}-differential form,
i.e.~a form whose coefficients are in $\bcC (M^\circ)$, then 
$\int_{\tb M} \omega$  is defined accordingly as the finite part of 
$\int_{x \geq -R} \omega$ as $R\rightarrow \infty$.
\end{defprop}

In local coordinates $y_1,\ldots,y_n$ on $\partial M$, \textup{b}-differential $p$--forms
are sums of terms of the form
\begin{equation}
 \go =f(x,y) dx\wedge dy_{i_1}\wedge \ldots\wedge dy_{i_{p-1}}+ 
      g(x,y) dy_{j_1}\wedge \ldots\wedge dy_{j_{p}},
\end{equation}
where $1\le i_1<\ldots<i_{p-1}\le n, \; 1\le j_1<\ldots<j_p\le n$ and $f,g$
are \textup{b}-smooth functions. Putting $\iota^*\go:= g_0^-(y)  dy_{j_1}\wedge \ldots\wedge dy_{j_{p}}$
(\emph{cf.}~\eqref{eq:ML20090122-2})
extends to a pullback $\iota^*:\bOmega^p(M)\to \Omega^p(\partial M).$ It is easy to see that
Stokes' Theorem holds for $\int_{\tb M}$ and $\iota$:
\begin{equation}
   \int_{\tb M} d\go = \int_{\partial M} \iota^* \go.
\end{equation}

For a \textup{b}-pseudodifferential operator $A\in\bpdo^\bullet(M;E)$ of order $<-\dim M$
the {\btrace}  is nothing but the partie finie integral of its kernel over the 
diagonal:
\sind{partie finie}
\begin{equation}\label{eq:bTrace-as-Pf}
	\bTr(A)=\int_{\tb M} \tr_{p}(K_A(p,p)) d\vol(p),
\end{equation}
where now $K_A(\cdot,\cdot)$ denotes the Schwartz kernel 
of $A$ and $\tr_p$ denotes the fiber trace on $E_p$.

Next we mention a useful formula for the partie finie integral in terms of
a convergent integral. By Eq.~\eqref{eq:ML20090219-4},
the asymptotic expansion \eqref{eq:ML20090122-1} may be differentiated,
hence $\partial_x f=O(e^x)$, $x\to -\infty$, is integrable and thus integration by parts 
yields
\begin{equation}\label{eq:ML20090122-3}
\begin{split}
  \int_{-R}^0 f(x) dx &= Rf(-R)-\int_{-R}^0 x \partial_xf(x) dx\\
                      &= R f_0^- - \int_{-\infty}^0 x \partial_x f(x)dx+ O(R e^{-R}),\quad 
     R\to -\infty.
\end{split}
\end{equation}
Hence
\nind{PartieFinie@Pf$\intop$} 
\begin{equation}\label{eq:PfIntFormula}
   \pfint_{-\infty}^0 f(x) dx = \int_{-\infty}^0 x\partial_xf(x) dx,
\end{equation}
where the integrand on the right hand side is summable in the Lebesgue sense.
Using the tools from the previous paragraphs, we can now prove the following 
theorem about the representation of the {\btrace}  as a trace of certain trace class operators. 
\begin{proposition}\label{t:bTraceAsTrace}
Let $M$ be a compact manifold with boundary and an exact \textup{b}-metric $\bmet$.
Fix a collar $(r,\eta) : Y  \rightarrow [0,2)\times \partial M$ of the boundary $\partial M$ 
as described in Section \ref{App:bdefbmet}, and let 
$(x,\eta): \clos{Y^1} \rightarrow \, (-\infty, 0] \times \partial M$ 
denote the corresponding diffeomorphism onto the cylinder $(-\infty,0]\times\partial M$.  
Assume that $A\in{\bpdo^\infty} (M;E)$ is a classical \textup{b}-pseudodifferential operator of 
order $<-\dim M$, and that 
its kernel is supported within the cylinder $(-\infty,0)\times\partial M$.
Then $x\big[ \frac{d}{dx},A \big]$ is trace class and one has
\begin{equation}\label{eq:bTraceAsTrace} 
\begin{split}
   \bTr(A)&=-\Tr \left( x \big[ \frac{d}{dx},A \big] \right) =\\
          &= -\int_{(-\infty,0)\times\partial M}\, x \frac{d}{dx}
          \tr_{x,q}\bigl(K_A(x,q;x,q)\bigr)\,  d\vol(x,q),
\end{split}
\end{equation}
where $K_A$ denotes the Schwartz kernel of $A$.
\end{proposition}
\begin{proof} The condition on the support of $A$ is necessary since
the operators $x$ and $\frac{d}{dx}$ are only defined on the cylinder.
However, Proposition \plref{t:bTraceAsTrace} can be extended to arbitrary
$A\in{\bpdo^{<-\dim M}}(M;E)$ in a straightforward way: choose a pair of
cut-off functions $\varphi,\psi\in \cinf{M}$ with support within the cylinder
$(-\infty ,0) \times \partial M$ and such that 
$\varphi(x)=1$ for $x\le -2$, $\varphi(x)=0,$ for $x\ge -3/2$,
$\psi(x)=1$ for $x\le -1$ and $\psi(x)=0$ for $x\ge -1/2$.
Finally, choose a cut-off function $\chi\in\cinfz{M\setminus\partial M}$ with
compact support and $\chi (1-\varphi)=1-\varphi$. The definition
of the {\btrace}  then immediately shows
that 
\[ 
  \bTr(\varphi A)=\bTr(\psi\varphi A)=\bTr(\varphi A \psi)
\]
and hence
\[
\begin{split}
  \bTr(A)&=\bTr(\varphi A \psi)+\Tr((1-\varphi) A\chi) =\\
         &=-\Tr(x [\frac{d}{dx},\varphi A \psi])+\Tr((1-\varphi) A\chi).
\end{split}
\] 
The fact that $x[\frac{d}{dx},A]$ is trace class follows by Prop.~\ref{prop:smotrcl},
since the indicial family of the commutator $[\frac{d}{dx},A]$ vanishes.
We provide two variants of proof for \eqref{eq:bTraceAsTrace}.
\subsection*{1st Variant.}
From equations \eqref{eq:bTrace-as-Pf} and \eqref{eq:PfIntFormula} 
we infer 
\begin{equation}
\begin{split}
\bTr(A) &= \int_{\tb (-\infty,0)\times\partial M}\tr_{x,q}\bigl(K_A(x,q;x,q)\bigr) d\vol(x,q)\\
        &= -\int_{(-\infty,0)\times\partial M}x \frac{d}{dx}\tr_{x,q}\bigl(K_A(x,q;x,q)\bigr) d\vol(x,q).
\end{split}
\end{equation}
This proves the second line of \eqref{eq:bTraceAsTrace}. The first line follows,
since the kernel of $[\frac{d}{dx},A]$ is given
by $[\frac{d}{dx},K_A](x,p;y,q)=\partial_x K_A(x,p;y,q)+\partial_y K_A(x,p;y,q)$
which for $x=y$ equals $\frac{d}{dx} K_A(x,p;x,p)$, \emph{cf.}~Eq.~\eqref{eq:PfIntFormula}.

\subsection*{2nd Variant.}
For $\Re z>0$ the operator $e^{zx}A$ is trace class and the function
\begin{equation}
   z\mapsto \Tr(e^{zx}A)
\end{equation}
is holomorphic for $\Re z>0$ and it extends meromorphically to $\Re z>-1$, $0$ 
is a simple pole and the residue at $0$ equals $\bTr(A)$  (\emph{cf.}~\cite{Loy:DOB}). Hence
\begin{equation}
\begin{split}
   \bTr(A) &= \left.\frac{d}{dz} z\Tr(e^{zx}A)\right|_{z=0}\\
           &= \left.\frac{d}{dz} \Tr \left( \big[\frac{d}{dx},e^{zx} \big]
           A \right)\right|_{z=0}\\
           &= \left.\frac{d}{dz} \Tr \left( \big[ \frac{d}{dx},e^{zx}A \big] -e^{zx}
              \big[ \frac{d}{dx},A \big] \right)\right|_{z=0}\\
           &=-\Tr \left(x \big[\frac{d}{dx},A \big] \right),
\end{split}
\end{equation}
since for $\Re z>0$ the trace of the commutator $\Tr([\frac{d}{dx},e^{zx}A])$,
thanks to the decay of $e^{zx}$, does vanish.
The last claim follows as above. 
\end{proof}

\section{\textup{b}-Clifford modules and \textup{b}-Dirac operators}
\label{s:b-Clifford-Dirac}
Let $M$ be a compact manifold with boundary,
$r: Y \rightarrow [0,2)$ a boundary defining function,  
and $\bmet $ an exact \textup{b}-metric on $M$, \emph{cf.}~Section \plref{App:bdefbmet}.
If an object is derived from a \textup{b}-metric we indicate
this notationally by giving it a \textup{b}-decoration. This applies
in particular to the various structures derived from the riemannian metric $\bmet$
as described in Section \plref{s:qDirac}, 
\emph{e.g.}~the (co)tangent bundles $\bT M, \bT^* M$,
the Levi-Civita \textup{b}-connection $\bnabla$ belonging to $\bmet$, 
the bundle of Clifford algebras $\bCl (M) := \Cl (\bT^*M)$, and the 
\emph{\textup{b}-Clifford superconnection}\sind{bCliffordsuperconnection@\textup{b}-Clifford superconnection}\sind{Clifford superconnection} 
$\bA$ on a degree $q$ \textup{b}-Clifford module\sind{bCliffordmodule@\textup{b}-Clifford module}\sind{Clifford module}  $W$ over $M$. 
For a discussion of $\bcC(M^\circ)$ vs. $\cC^\infty(M)$ we refer
to Section \plref{App:bdefbmet}.

In the remainder of this article, we assume that
a \textup{b}-Clifford superconnection is always of product form near the boundary. This 
means that over $Y^s$ for some $s$ with $0<s<2$ the superconnection 
has the form
\[
   \bA\rest{Y^s} = \eta^* \nabla^\partial + \eta^* \omega^\partial \wedge -,
\]
where $\eta : Y \rightarrow \partial M$ is the boundary projection from
Section \ref{App:bdefbmet}, $\nabla^\partial$ is a metric connection on the 
restricted bundle $W\rest{\partial M}$ and 
$\omega^\partial\in\Omega^\bullet\bigl(\partial M ;\End\big(W\rest{\partial M}\big)\bigr)$.
Recall that the pull-back covariant derivative $\big(\eta^* \nabla^\partial\big)$ 
on $W\rest{Y}$ is uniquely defined  by requiring for 
$\xi\in \Gamma^\infty \big( Y; W\big)$ that    
\[
  \big(\eta^* \nabla^\partial\big)_V  \,\xi  = 
  \begin{cases}
    r \frac{\partial \xi}{\partial r},& 
    \text{if $V= r\frac{\partial}{\partial r} $},\\
    \nabla^\partial_{\tilde V}\xi ,&  \text{if $V= \tilde V \circ \eta$ for some
    $\tilde V \in \Gamma^\infty (\partial M; T\partial M)$}.
  \end{cases}
\]

Note that the \textup{b}-metric on $M$  and the metric structure on $W$ give rise to the Hilbert 
space $\sH = L^2(M;W)$ of square integrable sections of the \textup{b}-Clifford module.
By assumption, $\Cl_q$ acts on $L^2(M;W)$, hence by Eq.~\eqref{eq:Clifford-super-trace} 
one obtains a supertrace 
$\Str_q: \sL_{\Cl_q}^1\big( L^2(M;W) \big) \rightarrow \C$. 
Similarly the {\btrace}  gives rise to a \textup{b}-supertrace
\begin{equation}\label{eq:bSuperTrace}
\begin{split}
&\bStr_q: \bpdosub{\Cl_q}^{<-\dim M}\bigl(M;W\bigr)\longrightarrow \C,\\
& \bStr_q(K):=(4\pi)^{-q/2}\;\;\bTr(\ga E_1\cdot ... \cdot E_q K).
\end{split}
\end{equation}
Here, $\bpdosub{\Cl_q}^\bullet\bigl(M;W\bigr)$ denotes the space of classical 
\textup{b}-pseudodifferential operators which lie in the supercommutant of $\Cl_q$ in $\sH$,
\emph{cf.}~\eqref{eq:def-super-commutator} \emph{supra}.

Next consider the natural embedding $T^*\partial M \hookrightarrow \bT^*_{|\partial M}M$.
By the universal property of Clifford algebras one obtains 
an embedding of Clifford bundles $\Cl \big(\partial M\big) \hookrightarrow 
\bCl \big( \bT^*_{|\partial M}M\big) $. Moreover, the 
decomposition $\bT^*\rest{\partial M}M=T^*M\oplus \R \cdot r\frac{\pl}{\pl r} $
induced by $\bmet$ even gives rise to a splitting  
$\bCl \big( \bT^*_{|\partial M}M\big) \rightarrow \Cl \big( \partial M \big)$. 
Let now $W\rightarrow M$ be a degree $q$ \textup{b}-Clifford module over $M$. 
Then $\Cl \big(\partial M\big)$ acts on $W_{|\partial M}$ via the embedding
$\Cl \big(\partial M\big) \hookrightarrow \bCl \big( \bT^*_{|\partial M}M\big)$.
We denote the resulting left action of the boundary Clifford bundle on
$W_{|\partial M}$ again by $\sfc$. 
Moreover, the action $W_{|\partial M} \otimes \Cl_q \rightarrow W_{|\partial M}$
extends to a right action 
$\sfbcr :W_{|\partial M}\otimes \Cl_{q+1} \rightarrow W_{|\partial M}$ by putting
\begin{equation}
  \label{eq:defracbdr}
   E_j = \sfbcr (w, e_j ) := 
  \begin{cases}
    \sfcr \big( w, e_j \big), & \text{for $w \in W_p$, $p\in \partial M$, 
    $j=1,\cdots , q$}, \\
    -\sfcl \big(\frac{dr}{r},w \big), & \text{for $w \in W_p$, $p\in \partial M$,  
    $j=q+1$},
  \end{cases}
\end{equation}
\emph{cf.}~the beginning of Section \ref{s:qDirac}.
It is now easy to check that $W\rest{\partial M}$ together with $\sfc$ and
$\sfbcr$ as Clifford actions becomes a degree $q+1$ Clifford-module over 
$\partial M$. 

Now we have the ingredients for the \textup{b}-supertrace of a supercommutator:

\begin{proposition}[{\cite[Cor.~5.5]{Get:CHA}}]
\label{p:b-trace-defect-graded}
Let $\dirac$ be a Dirac operator on a $q$-graded \textup{b}-Clifford bundle $W$,
and $K\in\bpdosub{\Cl_q}^{-\infty}\big( M ; W \big)$. 
On $\partial M$ put $E_{q+1}:=-\sfcl(\frac{dr}{r})=-\sfcl(dx)$. 
Then
\begin{equation}\label{eq:bTraceDefectClifford}
  \bStr_q \big([\dirac,K]_\super\big)=
  \frac{1}{\sqrt{\pi}} \int_{-\infty}^\infty \Str_{q+1,\partial M}\big( \cI(K,\gl) \big )d\gl.
\end{equation}
\end{proposition}
We do not claim here that $\cI(K,\gl)$ commutes with $\Gammabdy$ in the graded sense.
This is not necessary for the definition of $\Str_{q+1,\partial M}$.

\begin{remark}\label{r:DiracCylinderFormulas}
Another consequence of the previous considerations which we single out for future 
reference is the structure of a Dirac operator on a cylinder $\R\times \partial M$ 
(\emph{cf.}~\eqref{eq:ML20090219-1}). Since all structures are product, $\dirac$ takes the form
\begin{equation}\label{eq:dirac-cylinder}
    \dirac=\sfc(dx)\frac{d}{dx}+\diracbdy=:\Gammabdy\Bigl(\frac{d}{dx}+A\Bigr).
\end{equation}
Here $\Gammabdy=\sfc(dx)$ is Clifford multiplication by the normal vector $\frac{d}{dx}$
and $\diracbdy:=\Gammabdy \Abdy$ is the \semph{tangential operator}. 
$\diracbdy$ is a Dirac operator on the boundary. Moreover, one has the relations
\begin{equation}
\label{eq:dirac-cylinder-formulas}
  \Gammabdy^*=-\Gammabdy, \quad \Gammabdy^2=-I,\quad A^t=
  A, \quad \Gammabdy A+A\Gammabdy=\Gammabdy\diracbdy+\diracbdy\Gammabdy=0.
\end{equation}
\end{remark}
Now let $u$ be a section of the Clifford bundle $W$ over the cylinder $\R \times M$. 
Then  
\begin{equation}
  \begin{split}  
  (\dirac u) (x,p) = \, &
  \frac{1}{2\pi} \int_{-\infty}^\infty 
  \bigl(\sfc(dx)\frac{d}{dx}+\diracbdy\bigr) e^{ix\gl}  \hat u (\lambda , p) \, d\gl = \\
  = \, &\frac{1}{2\pi} \int_{-\infty}^\infty e^{ix\gl}
  \bigl(i\sfc(dx)\gl+\diracbdy\bigr) \hat u (\lambda , p) \, d\gl .
  \end{split}
\end{equation}  
By Eqs.~\eqref{eq:ftbdrop} and \eqref{eq:asymptotics-b-symbol} this proves the following:
\begin{proposition}\label{p:Dirac-indicial-family} 
Let $M$ be a compact manifold with boundary and let $\bmet$ be an
exact \textup{b}-metric on $M$. Furthermore, let $\dirac$ be a Dirac operator on $M$. Then
the indicial family\sind{indicial family} of $\dirac$ is given by $\cI(\dirac)(\gl)=i\gl \sfc(dx)+\diracbdy$.
\end{proposition}

\section{The \textup{b}-JLO cochain}
\label{s:bJLOcc}

The degree $q$ Clifford module approach outlined in Section \plref{s:qDirac}
has advantages when dealing with manifolds with boundary, because the formul\ae\, for
the \JLO-cocycle and its transgression
 (\emph{cf.}~\eqref{Eq:cocyclecond}, \eqref{Eq:transgress} below) become simpler. 
To make the connection to the standard even and odd Chern character without Clifford action,
from now on 
we will also consider \emph{ungraded}\sind{Clifford module!ungraded} 
Clifford modules without auxiliary Clifford right action. Therefore we assume that either
\begin{itemize}
\item we are in the \emph{graded}\sind{Clifford module!graded} 
case with $q$ Clifford matrices $E_1,\ldots,E_q$ where $\dirac$ is odd, 
and $\Str_q$ denotes the Clifford trace defined in Section \plref{s:qDirac},
\end{itemize}
or
\begin{itemize}
\item we are in the \emph{ungraded} case, when there are no Clifford matrices and
no grading operator; this case can be conveniently dealt with by putting
$q=-1$ (which is odd!), $\ga=1$ and $\Str_q:=\Tr=\Tr(\ga \cdot)$.
\end{itemize}

From now on, we assume that $\dirac_t$, $t\in (0,\infty),$ is a family of 
self-adjoint differential operators of the form $\dirac_t = f(t) \dirac$ with 
$\dirac$ the Dirac operator of a $q$-graded \textup{b}-Clifford module $W$
over the \textup{b}-manifold $M$ ($q\ge -1$ according to the previous explanation)
with \textup{b}-Clifford superconnection $(W,\bA)$ and $f: (0,\infty) \rightarrow \R$ a 
smooth function.
$\dirac_t$ are Dirac type operators in the sense of \cite{Tay:PDEII}.

Following Getzler \cite[Sec.~6]{Get:CHA}, we define for 
$A_0,\cdots, A_k \in \bpdo^\infty (M,W)$ (\emph{cf.}~Subsection \plref{ss:JLODO})
\begin{equation}
  \begin{split}
   \blangle A_0, \cdots,A_k \rangle_{\dirac_t}  \, := &\int_{\Delta_k} \, \bStr_q
    \big(A_0 \, e^{- \sigma_0 \dirac_t^2} \cdots A_k \, e^{- \sigma_k \dirac_t^2} \big)d\sigma \\
               = &\bStr_q \big( (A_0,\ldots,A_k )_{\dirac_t} \big).
\end{split}
\end{equation}
Put for $a_0,\ldots, a_k\in \cC^\infty(M)$
\begin{align} 
\label{Eq:DefChern}
  \bCh^k (\dirac) (a_0,\cdots,a_k) & := \blangle a_0, [\dirac,a_1],
  \cdots, [\dirac, a_k]\rangle , \\
\label{Eq:DefSlChern}
\begin{split} 
  \bslch^k(\dirac , \mathsf{V} )(a_0,\cdots,a_k) &:= \\ 
    \sum_{0\leq j \leq k} (-1)^{j \, \deg{\mathsf{V}}} \,  
    \blangle a_0, [\dirac,a_1],&
    \cdots , [\dirac, a_j], \mathsf{V} ,  [\dirac, a_{j+1}], \cdots , 
    [\dirac, a_k] \rangle.
\end{split}
\end{align}

The operation $\bslch$ will mostly be used with $V=\dot \dirac_t$ as a second argument.
Here $\dirac_t$ is considered of odd degree regardless of the value of $q$.

\begin{remark}\label{rem:ML200908282}
For $q=0$, $k$ even resp. $q=-1$, $k$ odd $\bCh^k(\dirac)$, $\bslch^k(\dirac,\dot\dirac)$
are the \textup{b}-analogues of the even and odd \JLO\ Chern character and its transgression.
\end{remark}

The following result is crucial for this paper. It is essentially due to Getzler \cite[Thm.~6.2]{Get:CHA},
although the following version is not stated explicitly in his paper.

\begin{theorem}\label{P:GETZLER} For $q\ge 0$ we have the following two equations for $\bCh^\bullet(\dirac)$
and $\bslch(\dirac,\dot\dirac)$:
\begin{align} 
\label{Eq:cocyclecond}
 b\bCh^{k-1}(\dirac_t)+B\bCh^{k+1}(\dirac_t) & =\Ch^k(\diracbdyt)\circ i^*,\\
\label{Eq:transgress} 
 \frac{d}{dt}\bCh^k(\dirac_t) + 
 b\bslch^{k-1}(\dirac_t , \dot{\dirac}_t ) +B\bslch^{k+1}(\dirac_t,\dot\dirac_t)& =
 - \slch^k(\diracbdyt , \dot{\dirac}_{\pl,t} )\circ i^*.
\end{align}
\end{theorem}

These formul\ae\, will be repeatedly used in Section \ref{s:retracted-relative-cocycle} and thereafter.
For notational convenience we will omit the symbol
$\circ i^*$ whenever the context makes clear that
this composition is required.

The theorem can be derived from \cite[Thm.~6.2]{Get:CHA} by introducing the form valued
expression
\begin{equation}\label{eq:ML200911191}
    \bllangle A_0, \cdots ,A_k \rrangle  \, := \int_{\Delta_k} \, \bStr_q
    \big( A_0 \, e^{- \sigma_0 (i \, d \dirac_t  + \dirac_t^2)} \cdots
    A_k \, e^{- \sigma_k (i \, d \dirac_t  + \dirac_t^2)} \big)d\sigma, 
\end{equation}
and the combined Chern character $\GCh^\bullet$, defined as
\begin{equation}\label{eq:combined-Chern}
\begin{split}
  \GCh^k\, & (\dirac_t) (a_0,\cdots,a_k)\\& := 
             \bllangle a_0, [\dirac_t,a_1], \cdots ,  [\dirac_t,a_k]\rrangle,
\end{split} \quad
  a_0, \cdots , a_k \in \cC^\infty ( M ).
\end{equation} 
For this, Getzler proves
\begin{equation}\label{eq:Getzler}
  (-i\, d+b+B) \GCh^\bullet (\dirac_t) = \GCh^\bullet (\diracbdyt)\circ i^* .
\end{equation}

\begin{remark}
  Note that in this paper we use self-adjoint Dirac operators while  Getzler 
  uses skew-adjoint ones in \cite{Get:CHA}. Accordingly, 
  our Dirac operators differ by a factor $-i$ from the Dirac operators in 
  \cite{Get:CHA}. This explains the appearance of such $i$-factors in our formul\ae\,,
  which are not present in \cite{Get:CHA}.
\end{remark}
By carefully tracing all the signs and $i$--factors involved in the graded
form valued Clifford calculus, as well as due to the various conventions, it turns out
that separating \eqref{eq:Getzler} into its scalar and $1$-form parts, using
\cite[Lem.~2.5]{Get:CHA}
\begin{equation}
\label{eq:split-JLO-scalar-form}
  \begin{split}
   \bllangle& A_0,\cdots,A_k \rrangle =\\
           &=\blangle A_0,\cdots,A_k \rangle -i \sum_{j=0}^k \,
           \blangle A_0,\cdots,A_{j},dt \wedge \dot \dirac_t,A_{j+1},\cdots,A_k\rangle ,
\end{split}
\end{equation}
one obtains Eqs.~\eqref{Eq:cocyclecond} and~\eqref{Eq:transgress}. 

However, for completeness,
 we will give a more direct argument in Section \ref{s:CocycleFormula},
without using operator valued forms. The proof below follows the
lines of the standard proof for the \JLO-cocycle representing the
Chern character of a $\theta$-summable Fredholm module  (\emph{cf.}~\cite{JLO:QKT},
\cite{GetSze:CCT}).

\section[Cocycle and transgression formul\ae]%
{Cocycle and transgression formul\ae\ for the even/odd \textup{b}-Chern character (without Clifford covariance)}
\label{s:CocTransgressNoClifford}
\newcommand{\warning}[1]{ {\color{red} #1}}

Recall from Remark \plref{rem:ML200908282} that for $q=0$ and $k$ even resp. $q=-1$ and $k$ odd
$\bCh^\bullet(\dirac)$ is the \textup{b}-analogue of the even, resp. odd, \JLO\ Chern character.
We shall relate the ungraded ($q=-1$) case to the graded case with $q=1$.

Starting with an ungraded Dirac operator $\dirac_t$ acting on the Hilbert space $\sH$,
put
\begin{equation}\label{eq:ML200911192}
\widetilde \sH:=\sH\oplus \sH,\quad \ga:=\begin{pmatrix} 1 & 0\\ 0 & -1\end{pmatrix},\quad
\widetilde \dirac_t:=\begin{pmatrix} 0& \dirac_t \\ \dirac_t & 0\end{pmatrix}.
\end{equation}
Then $\widetilde \dirac_t$ is odd with respect to the grading operator $\ga$ and it anti commutes
with
\begin{equation}\label{eq:ML200911193}
   E_1:=\begin{pmatrix} 0 & 1 \\ -1 & 0\end{pmatrix}.
\end{equation}
Note that
\begin{equation}\label{eq:ML200909011}
   \widetilde \dirac_t= \ga E_1 \bigl(\dirac_t\otimes I_2\bigr)
\end{equation}
with $I_2$ being the $2\times 2$ identity matrix.

\begin{proposition}\label{p:ML200909032}
Let $\dirac_t$ be ungraded {\rm ($q=-1$)}
and let $\widetilde \dirac_t=\ga E_1 (\dirac_t\otimes I_2)$ be the associated
$1$--graded {\rm ($q=1$)} operator. Then for $k$ odd 
\begin{equation}
\begin{split}
\bCh^k(\widetilde \dirac_t) &\, = \frac{1}{\sqrt{\pi}} \bCh^k(\dirac_t),\\
\bslch^{k-1}(\widetilde \dirac,\dot{\widetilde \dirac_t}) &\, =\frac{1}{\sqrt{\pi}} \bslch^{k-1}(\dirac_t,\dot\dirac_t).
\end{split}
\end{equation}
\end{proposition}
Needless to say that these formul\ae\ are valid as well for $\Ch^\bullet$ and $\slch^\bullet$.
\begin{proof}
Using Proposition \plref{p:b-trace-defect-graded} we find for $k$  odd:
\begin{equation}
\label{eq:relOddGradedUngraded}
\begin{split}
   \blangle &a_0, [\widetilde \dirac_t,a_1],\ldots,[\widetilde \dirac_t,a_k]\rangle_{\widetilde \dirac_t}\\
            &= \blangle (\ga E_1)^{k} a_0, [\dirac_t\otimes I_2,a_1],\ldots,[\dirac_t\otimes I_2,a_k]\rangle_{\dirac_t\otimes I_2}\\
	    &= \int_{\Delta_k} \frac{1}{\sqrt{4\pi}} \bTr\bigl( \underbrace{(\ga E_1)^{k+1}}_{=1}
	            a_0 e^{-\sigma_0 \dirac_t^2}[\dirac_t,a_1]\ldots [\dirac_t a_k]e^{-\sigma_k \dirac_t^2}\otimes I_2\bigr)\\
	    &= \frac{1}{\sqrt{\pi}} \blangle a_0, [\dirac_t,a_1],\ldots,[\dirac_t,a_k]\rangle_{\dirac_t}.
\end{split}
\end{equation}
The calculation for $\bslch^{k-1}(\widetilde \dirac_t,\dot{\widetilde\dirac_t})$ is completely analogous.
\end{proof}

Now we are ready to translate \eqref{Eq:cocyclecond} and \eqref{Eq:transgress} into formul\ae\ 
for the standard even and odd Chern character without Clifford action.

\subsection{$q=0$} A priori we are in the standard even situation without Clifford right action.
However, $\diracbdy$ is viewed as $\Cl_1$ covariant with respect to the Clifford action
given by $E_1=-\Gammabdy$. On the boundary, $\Gammabdy$ gives a natural identification of the even
and odd half spinor bundle and with respect to the splitting into half spinor bundles $\dirac$ takes the form:
\begin{equation}
\begin{split}
    \dirac &=\underbrace{\begin{pmatrix} 0 & -1 \\ 1 & 0 \end{pmatrix}}_{\Gammabdy}
         \frac{d}{dx} + \underbrace{\begin{pmatrix} 0 & A \\ A & 0\end{pmatrix}}_{\diracbdy} ;
\end{split}
\end{equation}
$A$ is an ungraded Dirac type operator acting on the positive half spinor bundle (it is the operator whose
positive spectral projection gives the {\APS}  boundary condition). In the notation of Eq.~\eqref{eq:ML200911192},
we have $\diracbdy=\widetilde \Abdy$, $E_1=-\Gammabdy$. Thus, Proposition \plref{p:ML200909032} and Theorem \ref{P:GETZLER}
give the following result.
\begin{proposition}\label{p:CocycleTransgressEven}
	Let $M$ be an even dimensional compact manifold with boundary with an exact \textup{b}-metric
	and let $\dirac_t=f(t)\dirac $ be as before. Writing $\dirac$ in a collar of the boundary (in cylindrical coordinates) in the form
	\begin{equation}
		\dirac=:\begin{pmatrix} 0 & -\frac{d}{dx}+A\\ \frac{d}{dx} +A\end{pmatrix},
	\end{equation}
$A$ is an ungraded Dirac type operator acting on the positive half spinor bundle restricted to the boundary.
Furthermore we have with $A_t=f(t)A$
\begin{equation} 
\label{Eq:cocyclecondEven}
b\bCh^{k-1}(\dirac_t)+B\bCh^{k+1}(\dirac_t)  =\frac{1}{\sqrt{\pi}}\Ch^k(A_t)\circ i^*,
\end{equation}
\begin{equation}\begin{split}
\label{Eq:transgressEven}
 \frac{d}{dt}\bCh^k(&\dirac_t) + 
 b\bslch^{k-1}(\dirac_t , \dot{\dirac}_t ) +B\bslch^{k+1}(\dirac_t,\dot\dirac_t)\\
 & =  -\frac{1}{\sqrt{\pi}} \slch^k(A_t, \dot{A}_t)\circ i^*.
\end{split}\end{equation}
\end{proposition}

\subsection{$q=-1$} Now let $\dirac$ be ungraded and put $\widetilde D,\ga, E_1$ as
in Eqs.~\eqref{eq:ML200911192}, \eqref{eq:ML200911193}, \eqref{eq:ML200909011}. 
Then by Proposition \plref{p:ML200909032} we have
\begin{align}
	\bCh^k(\dirac_t)(a_0,\ldots,a_k)&= \sqrt{\pi}\; \bCh^k(\widetilde \dirac_t)(a_0,\ldots,a_k), \label{eq:ML200909034}\\
	\bslch^k(\dirac_t,\dot\dirac_t)(a_0,\ldots,a_k)&= \sqrt{\pi}\; \bslch^k(\widetilde\dirac_t,
                          \dot{\widetilde\dirac_t})(a_0,\ldots,a_k).
	    \label{eq:ML200909035}
\end{align}

In the collar of the boundary, we write as usual $\dirac =\Gammabdy \frac{d}{dx}+\diracbdy$, and thus
\begin{equation}
	\widetilde\dirac=\underbrace{\begin{pmatrix} 0 & \Gammabdy \\ \Gammabdy & 0 \end{pmatrix}}_{=:\widetilde \Gammabdy}
		\frac{d}{dx}+\underbrace{\begin{pmatrix} 0 & \diracbdy\\ \diracbdy &0 \end{pmatrix}}_{=: \widetilde \diracbdy}.\
\end{equation}
$\widetilde\diracbdy$ is $2$--graded with respect to 
\begin{equation}
	E_1=
	\begin{pmatrix} 0 & 1 \\ -1 & 0 \end{pmatrix},\quad
      E_2 = -\widetilde \Gammabdy =\begin{pmatrix} 0 & -\Gammabdy \\ -\Gammabdy & 0\end{pmatrix}.
 \end{equation}
Note that
\begin{equation}
	\ga E_1 E_2 = -\Gammabdy \otimes I_2,\quad \widetilde \diracbdy = \ga E_1 (\diracbdy\otimes I_2).
\end{equation}
For even $k$ we have 
\begin{equation}
	\begin{split}
	&\Str_2\bigl( a_0 e^{-\sigma_0 \widetilde\diracbdyt^2} [\widetilde\diracbdyt,a_1]
              \cdot\ldots\cdot [\widetilde\diracbdyt,a_k] e^{-\sigma_k \widetilde\diracbdyt^2}\bigr)\\
      =& \frac{1}{4\pi} \Tr\bigl( \ga E_1 E_2 (\ga E_1)^k \bigl( a_0 e^{-\sigma_0 \diracbdyt^2} [\diracbdyt,a_1]
          \cdot\ldots\cdot [\diracbdyt,a_k] e^{-\sigma_k \diracbdyt^2}\bigr) \otimes I_2\bigr)\\
     = &- \frac{1}{2\pi}  \Tr\bigl(\Gammabdy a_0 e^{-\sigma_0 \diracbdyt^2} [\diracbdyt,a_1]
           \cdot\ldots\cdot [\diracbdyt,a_k] e^{-\sigma_k \diracbdyt^2}\bigr).
     \end{split}
\end{equation}
With respect to the grading given by $-i\Gammabdy$ we can now write
\begin{equation}
	\frac{-1}{2\pi} \Tr\bigl(\Gammabdy \cdot\bigr)= \frac{1}{2\pi i }\Str_0.
\end{equation}	

Together with \eqref{eq:ML200909034} and \eqref{eq:ML200909035} we have thus proved:
\begin{proposition}\label{p:CocycleTransgressOdd}
	Let $M$ be an odd dimensional compact manifold with boundary with an exact \textup{b}-metric
	and let $\dirac$ be an ungraded Dirac operator. Writing $\dirac$ in a collar of the boundary (in cylindrical coordinates)
in the form
	\begin{equation}
		\dirac=:\Gammabdy \frac{d}{dx}+\diracbdy,
	\end{equation}
$\diracbdy$ is a graded Dirac type operator with respect to the grading operator $-i\Gammabdy$.
Furthermore, we have
\begin{equation} 
\label{Eq:cocyclecondOdd}
b\bCh^{k-1}(\dirac_t)+B\bCh^{k+1}(\dirac_t)  =\frac{1}{2\sqrt{\pi}i}\Ch^k(\diracbdy)\circ i^*,
\end{equation}
\begin{equation}
	\begin{split}\label{Eq:transgressOdd}
 \frac{d}{dt}\bCh^k(&\dirac_t) + 
 b\bslch^{k-1}(\dirac_t , \dot{\dirac}_t ) +B\bslch^{k+1}(\dirac_t,\dot\dirac_t)\\
   & =  -\frac{1}{2\sqrt{\pi}i} \slch^k(\dirac_t^\pl, \dot{\dirac}_t^\pl)\circ i^*.
\end{split}
\end{equation}
\end{proposition}

\section{Sketch of Proof of Theorem \ref{P:GETZLER}}
\label{s:CocycleFormula}
Recall that Theorem \ref{P:GETZLER} is stated for $q\ge 0$, hence in this section all
Dirac operators will be $q$--graded with $q\ge 0$.

\begin{proposition}\label{p:shuffleRelations} Let $A_0,\ldots,A_k\in \bpdosub{\Cl_q}^{\bullet}\big( M ; W \big)$. Assume that
for all but one index $j_0$ the indicial family is independent of $\gl$ and commutes with
the actions of $E_1,\ldots,E_q$ and $E_{q+1}=-\Gammabdy$ (\emph{cf.} Section
\ref{s:qDirac}). For the possible exception $j_0$ we assume that $A_{j_0}$ is proportional
to $\dot\dirac_t$. Then
\begin{equation}\label{eq:shuffleRelationA}
   \blangle A_0,\ldots,A_k\rangle= (-1)^{\eps} \blangle A_k,A_0,\ldots,A_{k-1}\rangle,
\end{equation}
where $\eps=|A_k|(|A_0|+\ldots+|A_{k-1}|)$.
\begin{equation}\label{eq:shuffleRelationB}
\begin{split}
    \blangle A_0,\ldots,A_k\rangle &= \sum_{j=0}^k \blangle A_0,\ldots,A_j,1,A_{j+1},\ldots,A_k\rangle\\
           &=  \sum_{j=0}^k (-1)^{\eps_j} \blangle 1, A_j,\ldots,A_k,A_0,\ldots,A_{j-1}\rangle,
	   \end{split}
	   \end{equation}
where $\eps_j= (|A_0|+\ldots+|A_{j-1}|)(|A_j|+\ldots+|A_k|)$.

For $j<k$
\begin{equation}\label{eq:shuffleRelationC}
\begin{split}
   \blangle &A_0,\ldots,A_{j-1},[\dirac^2,A_j],A_{j+1},\ldots,A_k\rangle\\
        =&\blangle A_0,\ldots,A_{j-2}, A_{j-1}A_j,A_{j+1},\ldots,A_k\rangle\\ 
                               &- \blangle A_0,\ldots,A_{j-1}, A_{j}A_{j+1},A_{j+2},\ldots,A_k\rangle.
\end{split}
\end{equation}
Similarly, for $j=k$
\begin{equation}\label{eq:shuffleRelationD}
\begin{split}
   \blangle &A_0,\ldots,A_{k-1},[\dirac^2,A_k]\rangle \\
       =&\blangle A_0,\ldots,A_{k-2}, A_{k-1}A_k\rangle\\
                   &-  (-1)^ {|A_k|(|A_0|+\ldots+|A_{k-1}|)}\blangle A_k A_0,\ldots,A_{k-1}\rangle.
\end{split}
\end{equation}
\end{proposition}
Note that these formul\ae\  are the same as in Getzler-Szenes \cite[Lemma 2.2]{GetSze:CCT}. In particular there is no
boundary term. 
The proof proceeds exactly as the proofs of \cite[Lemma 6.3]{Get:CHA} (1),(2), (4), and we omit the details.
We only note that one has to make heavy use of the following lemma in order to show the vanishing
of certain terms:

\begin{lemma}[Berezin Lemma]\label{l:BerezinLemma}
 Let $K\in\sL_{\Cl_q}^1(\sH)$ (\emph{cf.}~Section \ref{s:qDirac}). Then for $j<q$
\[ \Tr(\ga E_1\cdot\ldots\cdot E_j K)=0.\]
\end{lemma}
\begin{proof}
If $j+q$ is odd then moving $\ga$ past $E_1\cdot\ldots\cdot E_j K$ and using the trace 
property gives
\begin{equation}
\begin{split}
   \Tr(\ga E_1\cdot\ldots\cdot E_j K)&= - \Tr(E_1\cdot\ldots\cdot E_j K \ga)\\
                                &= - \Tr(\ga E_1\cdot\ldots\cdot E_j K) =0.
\end{split}
\end{equation}
If $j+q$ is even then, since $j<q$, $E_q$ anti commutes with $\ga E_1\cdot \ldots\cdot E_j K$
and hence similarly
\begin{multline*}
   \Tr(\ga E_1\cdot\ldots\cdot E_j K)= - \Tr(E_q^2 \ga E_1\cdot\ldots\cdot E_j K)   
   =  \Tr(E_q \ga E_1\cdot\ldots\cdot E_j K E_q)\\ =  \Tr(E_q^2 \ga E_1\cdot\ldots\cdot E_j K) = -   \Tr(\ga E_1\cdot\ldots\cdot E_jK)=0.\qedhere
\end{multline*}
\end{proof}

We will make repeated use of the equations \eqref{eq:shuffleRelationA}--\eqref{eq:shuffleRelationD}.
Now we can proceed as for a $\theta$--summable Fredholm module. Following
\cite[p. 451]{GraVarFig:ENG} we start with the supercommutator
\begin{equation}\label{eq:SuperCommutatorA}
\int_{\Delta_k} \bStr_q\bigl(\bigl[\dirac_t,a_0 e^{-\sigma_0 \dirac_t^2}[\dirac_t,a_1]\ldots  
[\dirac_t,a_k] e^{-\sigma_k \dirac_t^2}\bigr] \bigr)d\sigma, 
\end{equation}
with $a_0,\ldots, a_k\in \bcC(M^\circ).$
As in \cite[bottom of p. 37]{Get:CHA} one shows, using Proposition \plref{p:b-trace-defect-graded}
and the fact that $\int_{-\infty}^\infty e^{-\gl^2}d\gl=\sqrt{\pi}$, that this supercommutator equals
\begin{equation}\label{eq:ML200908281}
   \langle a_{0,\pl}, [\dirac_t^\partial,a_{1,\pl}],\ldots,[\dirac_t^\partial,a_{k,\pl}]\rangle_{\dirac_t^\partial}.
\end{equation}
It is important to note that here we are in the case $q+1$, where the grading is the induced grading on the boundary
and $E_{q+1}=-\Gammabdy$.

For convenience we will write $\dirac$ instead of $\dirac_t$.
Expanding the supercommutator \eqref{eq:SuperCommutatorA} on the other hand gives
\begin{equation}\label{eq:SuperCommutatorB}
\begin{split}
    \blangle &[\dirac,a_0],\ldots,[\dirac,a_k]\rangle\\
         &+\sum_{j=1}^k (-1)^{j-1} \blangle a_0,[\dirac,a_1]\ldots,[\dirac,a_{j-1}],[\dirac^2,a_j],\ldots,[\dirac,a_k]\rangle,
	 \end{split}
\end{equation}
where we have used $[\dirac^2,a_j]=[\dirac,[\dirac,a_j]]_{\Z^2}$.

We can now calculate the effect of $b$ and $B$ on $\bCh$. 
\begin{equation}\label{eq:ML200909012}
\begin{split}
   B &\bCh^{k+1}(\dirac)(a_0,\ldots,a_k)\\
      &= \sum_{j=0}^k (-1)^{kj} \blangle 1,[\dirac,a_j],\ldots,[\dirac,a_k],[\dirac,a_0],\ldots,[\dirac,a_{j-1}]\rangle \\
      &= \sum_{j=0}^k \blangle [\dirac,a_0],\ldots,[\dirac,a_{j-1}],1,[\dirac,a_j],\ldots,[\dirac,a_k]\rangle \\
      &= \blangle  [\dirac,a_0],\ldots,[\dirac,a_k]\rangle, 
\end{split}
\end{equation}
where we used  \eqref{eq:shuffleRelationB} twice. Thus, the first summand in \eqref{eq:SuperCommutatorB}
equals $B\bCh^{k+1}(\dirac)(a_0,\ldots,a_k).$

Furthermore,
\begin{equation}\label{eq:SuperCommutatorC}
\begin{split}
   b &\bCh^{k-1}(\dirac)(a_0,\ldots,a_k)\\
      &= \blangle a_0 a_1,[\dirac,a_2],\ldots,[\dirac,a_k]\rangle\\
      &\qquad +\sum_{j=1}^{k-1} (-1)^j \blangle a_0,\ldots,[\dirac,a_j a_{j+1}],\ldots,[\dirac,a_k]\rangle\\
      &\qquad + (-1)^k \blangle a_k a_0,[\dirac,a_1],\ldots,[\dirac,a_{k-1}]\rangle\\
      &= \blangle a_0 a_1,[\dirac,a_2],\ldots,[\dirac,a_k]\rangle\\
       &\qquad - \blangle a_0, a_1[\dirac,a_2],\ldots,[\dirac,a_k]\rangle\\
       &\qquad + \sum_{j=1}^{k-2} (-1)^{j}\Biggl( \blangle a_0,[\dirac,a_1],\ldots,[\dirac,a_j]a_{j+1},\ldots,[\dirac,a_{k}]\rangle \\
       &\qquad \qquad - \blangle a_0,\ldots,[\dirac,a_j],a_{j+1}[\dirac,a_{j+2}],\ldots,[\dirac,a_{k}]\rangle \Biggr)\\
       &\qquad + (-1)^{k-1} \blangle a_0,[\dirac,a_1],\ldots,[\dirac,a_{k-1}]a_k\rangle\\
       &\qquad +(-1)^{k} \blangle a_k a_0,[\dirac,a_1],\ldots,[\dirac,a_{k-1}]\rangle\\
       &=\sum_{j=1}^k (-1)^{j-1} \blangle a_0,[\dirac,a_1],\ldots,[\dirac^2,a_j],\ldots,[\dirac,a_k]\rangle,
\end{split}
\end{equation}
where we have used \eqref{eq:shuffleRelationC} and \eqref{eq:shuffleRelationD}. The right hand side
of \eqref{eq:SuperCommutatorC} equals the sum in the second line of \eqref{eq:SuperCommutatorB}.

Summing up \eqref{eq:SuperCommutatorA}, \eqref{eq:ML200908281}, \eqref{eq:SuperCommutatorB}, \eqref{eq:ML200909012},
and \eqref{eq:SuperCommutatorC} we arrive at Eq.~\eqref{Eq:cocyclecond}.
\medskip

For additional clarity, let us perform two direct checks, for small values of $k$. 

\subsection*{Case 1: $k=0$} In this case \eqref{eq:SuperCommutatorB} equals
\begin{equation}
	\blangle [\dirac,a_0]\rangle= \blangle 1,[\dirac,a_0]\rangle= B\bCh^1(\dirac)(a_0),
\end{equation}
by \eqref{eq:shuffleRelationB} and we are done in this case.

\subsection*{Case 2: $k=1$} Then \eqref{eq:SuperCommutatorB} equals
\begin{equation}
  \blangle [\dirac,a_0],[\dirac,a_1]\rangle + \blangle a_0,[\dirac^2,a_1]\rangle.
\end{equation}
The first summand is $B\bCh^2(a_0,a_1)$ and the second summand equals in view of \eqref{eq:shuffleRelationD}
\begin{equation}
\begin{split}
  \blangle a_0 a_1\rangle - \blangle a_1 a_0\rangle = b\bCh^0(\dirac)(a_0,a_1).
\end{split}
\end{equation}

\subsection{The transgression formula}\sind{transgression formula|(}
To prove the transgression formula we proceed analogously
and start with the supercommutator
\begin{multline}\label{eq:TransgressSuperCommutatorA}
\sum_{j=0}^k (-1)^j
\int_{\Delta_{k+1}} \bStr_q\bigl(\bigl[\dirac_t,a_0 e^{-\sigma_0 \dirac_t^2}[\dirac_t,a_1]\ldots\\
   \ldots [\dirac_t,a_j] e^{-\sigma_j \dirac_t^2}
                               \dot D e^{-\sigma_{j+1}\dirac_t^2}\ldots [\dirac_t,a_k] e^{-\sigma_{k+1}\dirac_t^2}
\bigr] \bigr)d\sigma.
\end{multline}

We compute this supercommutator using Proposition \plref{p:b-trace-defect-graded}.
Note that by Proposition \ref{p:Dirac-indicial-family} $\cI(\dot\dirac_t,\gl)$ is proportional to $i\Gammabdy\gl+\dirac^\partial$.
The summand $i\Gammabdy\gl$ contributes a term proportional to $\int_{-\infty}^\infty \gl e^{-\gl^2}d\gl=0$.
The remaining summand gives, since  $\int_{-\infty}^\infty e^{-\gl^2}d\gl=\sqrt{\pi}$,
\begin{equation}\label{eq:ML200909021}
  \begin{split}
\sum_{j=0}^k (-1)^j \langle &a_{0,\pl},[\dirac_t^\pl,a_{1,\pl}],\ldots,[\dirac_t^\pl,a_{j,\pl}],\dot\dirac_t^\pl,\ldots,[\dirac_t^\pl,a_{k,\pl}]\rangle\\
      &=	  \slch^k(\diracbdy,\dot\diracbdy)(a_{0,\pl},\ldots,a_{k,\pl}).
      \end{split}
\end{equation}
Let us again emphasize that here we are in the case $q+1$, where the grading is the induced grading on the boundary
and $E_{q+1}=-\Gammabdy$.

Next we expand the commutator \eqref{eq:TransgressSuperCommutatorA}. However, we will confine ourselves
to small $k$. The calculation is basically the same as in \cite[p. 451]{GraVarFig:ENG}. The only difference is that
on a closed manifold \eqref{eq:TransgressSuperCommutatorA} is a priori $0$ while here it coincides with
the transgressed Chern character on the boundary.

\subsection*{Case 1: $k=0$} 
\eqref{eq:TransgressSuperCommutatorA} expands to
\begin{equation}\label{eq:TransgressSuperCommutatorEven}
    \blangle [\dirac_t,a_0],\dot\dirac_t]\rangle +\blangle a_0,[\dirac_t,\dot\dirac_t]\rangle.
\end{equation}
On the other hand
\begin{equation}\label{eq:ML200909022}
	\begin{split}	
	B\bslch^1(\dirac_t,\dot\dirac_t)(a_0)&=\bslch^1(\dirac_t,\dot\dirac_t)(1,a_0)\\
	         &= \blangle 1,\dot\dirac_t,[\dirac_t,a_0]\rangle
		     - \blangle 1,[\dirac_t,a_0],\dot\dirac_t \rangle\\
		 &= - \blangle [\dirac_t,a_0],1,\dot\dirac_t\rangle
		     - \blangle [\dirac_t,a_0],\dot\dirac_t,1 \rangle\\
		 &=  - \blangle [\dirac_t,a_0],\dot\dirac_t]\rangle,
	 \end{split}
\end{equation}
by \eqref{eq:shuffleRelationB}. Moreover using the well--known formula
\begin{equation}\label{eq:ML200909023}
    \frac{d}{dt} e^{-\sigma \dirac_t^2} = - \int_0^\sigma e^{(\sigma - s)\dirac_t^2} [\dirac_t,\dot\dirac_t] e^{-s\dirac_t^2} ds,
\end{equation}    
we have
\begin{equation}
	\frac{d}{dt} \bCh^0(\dirac_t)(a_0)=-\blangle a_0,[\dirac_t,\dot\dirac_t]\rangle,
\end{equation}
hence altogether
\begin{equation}
     	\frac{d}{dt} \bCh^0(\dirac_t)+B\bslch^1(\dirac_t,\dot\dirac_t)= -\slch^0(\dirac_t^\pl,\dot\dirac_t^\pl).
\end{equation}

\subsection*{Case 2: $k=1$} To be on the safe side, we also look at an example in the odd case.
We will again make repeated use of the formul\ae\ in Proposition \plref{p:shuffleRelations} without
further mentioning.
Eq. \eqref{eq:TransgressSuperCommutatorA} now expands to
\begin{equation}\label{eq:TransgressSuperCommutatorOdd}
	\begin{split}
	&	\blangle [\dirac_t,a_0],\dot\dirac_t,[\dirac_t,a_1]\rangle -
		\blangle [\dirac_t,a_0],[\dirac_t,a_1],\dot\dirac_t\rangle\\
	&	+ \blangle a_0,[\dirac_t,\dot\dirac_t],[\dirac_t,a_1]\rangle
	        + \blangle a_0,[\dirac_t,a_1],[\dirac_t,\dot\dirac_t]\rangle\\
	&       -\blangle a_0,\dot\dirac_t,[\dirac_t^2,a_1]\rangle
	        - \blangle a_0,[\dirac_t^2,a_1],\dot\dirac_t\rangle.
	\end{split}
\end{equation}

On the other hand
\begin{equation}\label{eq:ML200909024}
	\begin{split}	
	B\bslch^2&(\dirac_t,\dot\dirac_t)(a_0)= \bslch^2(\dirac_t,\dot\dirac_t)(1,a_0,a_1)-
	     \bslch^2(\dirac_t,\dot\dirac_t)(1,a_1,a_0)\\
          = &\blangle 1,\dot\dirac_t,[\dirac_t,a_0], [\dirac_t,a_1]\rangle
	    -\blangle 1,[\dirac_t,a_0],\dot\dirac_t, [\dirac_t,a_1]\rangle\\
	 & +\blangle 1,[\dirac_t,a_0], [\dirac_t,a_1],\dot\dirac_t\rangle
	    -\blangle 1,\dot\dirac_t,[\dirac_t,a_1], [\dirac_t,a_0]\rangle\\
	 & + \blangle 1,[\dirac_t,a_1],\dot\dirac_t, [\dirac_t,a_0]\rangle
	    -\blangle 1,[\dirac_t,a_1], [\dirac_t,a_0],\dot\dirac_t\rangle\\
            =& \blangle [\dirac_t,a_0], [\dirac_t,a_1],1,\dot\dirac_t\rangle
	    +\blangle [\dirac_t,a_0],1,[\dirac_t,a_1],\dot\dirac_t\rangle\\
	  &+\blangle [\dirac_t,a_0], [\dirac_t,a_1],\dot\dirac_t,1\rangle
	  -\blangle [\dirac_t,a_0],\dot\dirac_t, [\dirac_t,a_1],1\rangle\\
	  & -\blangle [\dirac_t,a_0],\dot\dirac_t,1, [\dirac_t,a_1]\rangle
          -\blangle [\dirac_t,a_0],1,\dot\dirac_t, [\dirac_t,a_1]\rangle\\
	 =& \blangle [\dirac_t,a_0],[\dirac_t,a_1],\dot\dirac_t\rangle
	    -\blangle [\dirac_t,a_0],\dot\dirac_t,[\dirac_t,a_1]\rangle,
    \end{split}
    \end{equation}
which equals the negative of the first two summands of \eqref{eq:TransgressSuperCommutatorOdd}.

Furthermore,
\begin{equation}
	b \bslch^0(\dirac_t,\dot\dirac_t)(a_0,a_1)=\bslch^0(\dirac_t,\dot\dirac_t)([a_0,a_1])
	  = \blangle [a_0,a_1],\dot\dirac_t\rangle.
\end{equation}
Applying \eqref{eq:shuffleRelationC} and \eqref{eq:shuffleRelationD} to the last
two summands of \eqref{eq:TransgressSuperCommutatorOdd} we find
\begin{equation}
    \begin{split}
	\blangle& a_0,\dot\dirac_t,[\dirac_t^2,a_1]\rangle+ \blangle a_0,[\dirac_t^2,a_1],\dot\dirac_t\rangle\\
	=& \blangle a_0,\dot\dirac_t a_1\rangle - \blangle a_1a_0,\dot\dirac_t\rangle
	 +\blangle a_0 a_1,\dot\dirac_t\rangle - \blangle a_0, a_1\dot\dirac_t\rangle \\
	=& \blangle a_0,[\dot\dirac_t,a_1]\rangle +b\bslch^0(\dirac_t,\dot\dirac_t)(a_0,a_1),
\end{split}
\end{equation}
hence adding $B\bslch^2(\dirac_t,\dot\dirac_t)(a_0,a_1)$ and $b\bslch^0(\dirac_t,\dot\dirac_t)(a_0,a_1)$
to the right hand side of \eqref{eq:TransgressSuperCommutatorOdd} we obtain
\begin{equation}
	\begin{split}
\slch^1(&\dirac_t^\pl,\dot\dirac_t^\pl)(a_{0,\pl},a_{1,\pl})+B\bslch^2(\dirac_t,\dot\dirac_t)(a_0,a_1)
   +b\bslch^0(\dirac_t,\dot\dirac_t)(a_0,a_1)\\
   =&- \blangle a_0,[\dot\dirac_t,a_1]\rangle	+ \blangle a_0,[\dirac_t,\dot\dirac_t],[\dirac_t,a_1]\rangle
	        + \blangle a_0,[\dirac_t,a_1],[\dirac_t,\dot\dirac_t]\rangle\\
   =& -\frac{d}{dt} \bCh^1(\dirac_t)(a_0,a_1)		
 \end{split}
\end{equation}
in view of \eqref{eq:ML200909023}.

With more effort but in a similar manner, the previous considerations can be extended to arbitrary $k$,
thus proving Eq.~\eqref{Eq:transgress}.

    
\sind{transgression formula|)}

\chapter{Heat Kernel and Resolvent Estimates}
\label{s:cylinder-estimates}
%
%
This is the most technical chapter of the paper. It is devoted to prove
some crucial estimates for the heat kernel\sind{heat kernel} of a \textup{b}-Dirac operator. 
These estimates will be used to analyze the short and long time 
behavior of the Chern character. Throughout this chapter we will mostly
work in the cylindrical context.

For the convenience of the reader we start by summarizing some
basic estimates for the resolvent and the heat operator associated
to an elliptic operator. These estimates will then be applied
in Section \ref{s:comparison-results} to prove comparison results for
the heat kernel\sind{heat kernel} and {\JLO} integrand of a Dirac operator on a general manifold with
cylindrical ends to those of a corresponding Dirac operator on the model
cylinder. In the remainder of the Chapter we will then prove short and
large time estimates for the \textup{b}-Chern character. This is in
preparation for proving heat kernel asymptotics
in the \textup{b}-setting in Section \ref{s: heat expansion}.

\section{Basic resolvent and heat kernel estimates on general manifolds}
\label{sec:basic-estimates} 

During the whole section $M$ will be a riemannian manifold without boundary and 
$D_0:\Gamma^\infty(M;W)\longrightarrow \Gamma^\infty(M;W)$ 
will denote a first order formally self-adjoint elliptic differential operator acting between sections
of the hermitian vector bundle $W$. 
We assume that there exists a self-adjoint extension, $D$, of $D_0$.
E.g. if $M$ is complete and $D_0$ is of Dirac type
then $D_0$ is essentially self-adjoint; if $M$ is the interior of a compact manifold
with boundary then $D$ can be obtained by imposing an appropriate boundary condition. 
For the following
considerations it is irrelevant which self-adjoint extension is chosen. We just fix one.

\subsection{Resolvent estimates}
\label{ss:ResolventEstimates}
We fix an open sector 
$\Lambda:=\bigsetdef{z\in\C\setminus\{0\}}{0< \eps < \arg z < 2\pi-\eps}\subset \C\setminus \R_+$ in the complex plane.

We introduce the following notation: for a function $f:\gL\to \C$ we write 
$f(\gl)= O(|\gl|^{\ga+0}), \gl\to\infty, \gl\in\gL$
if for every $\delta>0, \gl_0\in\gL,$ there is a constant $C_{\delta,\gl_0}$ such that 
$|f(\gl)|\le C_\delta |\gl|^{\ga+\delta}$ for $\gl\in\gL, |\gl|\ge |\gl_0|$.

We write $f(\gl)=O(|\gl|^{-\infty})$, $\gl\to\infty,\gl\in\gL$ if
$f(\gl)=O(|\gl|^{-N})$ for every $N$; the $O$--constant may depend on $N$.

\sind{Schatten class}
$L^2_s(M;W)$ denotes the Hilbert space of sections of $W$ which are of Sobolev
class $s$. The Sobolev norm of an element $f\in L^2_s(M;W)$ is denoted by
$\|f\|_s$. For a linear operator $T:L^2_s(M;W)\to L^2_t(M;W)$ its operator 
norm is denoted by $\|T\|_{s,t}$.

For an operator $T$ in a Hilbert space $\sH$ we denote by $\|T\|_p$
the $p$--th Schatten norm. To avoid confusions the letter $p$ will not
be used for Sobolev orders. 
Note that the operator norm of $T$ in $\sH$ coincides with $\|T\|_\infty$.

\begin{proposition}\label{p:ML20080403-A1}
Let $A,B\in\pdo^\bullet(M,W)$ be pseudodifferential operators of order $a,b$ 
with compact support.\footnote{This means that their
Schwartz kernels are compactly supported in $M\times M$.}


\textup{1.} If $k>(\dim M)/4+a/2$ then $A (D^2-\gl)^{-k}, (D^2-\gl)^{-k}A$ are Hilbert--Schmidt operators for $\gl\not\in\spec D^2$
and we have
\begin{equation}\label{eq:Hilbert-Schmidt-resolvent-estimate}
      \| A(D^2-\gl)^{-k}\|_2=O(|\gl|^{a/2+(\dim M)/4-k+0}),\quad \text{as } \gl\to\infty\text{ in }\gL.
\end{equation}
The same estimate holds for $\| (D^2-\gl)^{-k}A\|_2$.

\textup{2.} If $k>(\dim M+a+b)/2$ then $A (D^2-\gl)^{-k}B$ is of trace class for $\gl\not\in\spec D^2$
and
\begin{equation}\label{eq:trace-class-resolvent-estimate}
      \| A(D^2-\gl)^{-k}B\|_1= O(|\gl|^{(\dim M+a+b)/2-k+0}),\quad \text{as } \gl\to\infty\text{ in }\gL.
\end{equation}

\textup{3.} Denote by $\pi_1, \pi_2:M\times M\to M$ the projection onto the first resp. second factor
and assume that $\pi_2(\supp A)\cap \pi_1(\supp B)=\emptyset$. Then $A(D^2-\gl)^{-k}B$ is a trace class
operator for any $k\ge 1$ and 
\begin{equation}\label{eq:trace-class-smoothing-estimate}
        \| A(D^2-\gl)^{-k}B\|_1=O(|\gl|^{-\infty}), \quad \text{as } \gl\to\infty\text{ in }\gL.
\end{equation}
\end{proposition}
\begin{proof}
1. Sobolev embedding and elliptic regularity implies that for $f\in L^2(M;W)$ the section $A(D^2-\gl)^{-k}f$ is
continuous. Moreover, for $r>\dim M/2, |\gl|\ge |\gl_0|,$ and $x$ in the compact set $\supp A=:K$ 
\begin{equation}
   \begin{split}
       \|\bigl(A (D^2-\gl)^{-k}f\bigr)(x)\|&\le C \| (D^2-\gl)^{-k}f\|_{a+r,K}\\
                           &\le C \|(D^2+I)^{(a+r)/2}(D^2-\gl)^{-k} f\|_0\\
                           &\le C |\gl|^{-k+(a+r)/2} \|f\|.
   \end{split}
\end{equation}
For the Schwartz--kernel this implies the estimate
\begin{equation}
        \sup_{x\in\supp A} \int_M \|A(D^2-\gl)^{-k}(x,y)\|^2 d\vol(y)\le C |\gl|^{-2k+a+r},
\end{equation}
and since $A$ has compact support, integration over $x$ yields
\begin{equation}
    \begin{split}
        \| A&(D^2-\gl)^{-k}\|_2^2\\
         &\le \int_{\supp A} \int_M \|A(D^2-\gl)^{-k}(x,y)\|^2d\vol(x)d\vol(y)\\
         &\le C |\gl|^{-2k+a+r},
       \end{split}
\end{equation}       
proving the estimate \eqref{eq:Hilbert-Schmidt-resolvent-estimate}.
The estimate for $(D^2-\gl)^{-k}A$ follows by taking the adjoint.

2. The second claim follows from the first one using the H\"older inequality.  

3. To prove the third claim we choose cut--off functions $\varphi,\psi\in\cinfz{M}$
with $\varphi=1$ on $\pi_2(\supp A)$, $\psi=1$ on $\pi_1(\supp B)$ and $\supp\varphi\cap\supp \psi=\emptyset$.

Then $A(D^2-\gl)^{-k}B=A\varphi (D^2-\gl)^{-k}\psi B$ and
$\varphi (D^2-\gl)^{-k}\psi$ is a smoothing operator in the 
\semph{parameter dependent calculus} (\emph{cf.}~Shubin \cite[Chap. II]{Shu:POS}).
Hence for any real numbers $s,t,N$ we have
\begin{equation}
       \|\varphi (D^2-\gl)^{-k}\psi\|_{s,t}\le C(s,t,N)\; |\gl|^{-N}, \quad \text{as } \gl\to\infty\text{ in }\gL.
\end{equation}
Since the Sobolev orders $s,t$ are arbitrary this implies the claim.
\end{proof}

\begin{proposition}\label{p:ML20080403-A2}
Let $A\in\pdo^a(M,W)$ be a pseudodifferential operator with compact support.


\textup{1.} Let $\varphi\in\cinf{M}$ be a smooth function such that $\supp d\varphi$ is compact,
i.e. outside a compact set $\varphi$ is locally constant. Moreover suppose that
$\supp \varphi\cap \pi_1(\supp A)=\emptyset$. Then $\varphi(D^2-\gl)^{-k}A$ is a trace class
operator for any $k\ge 1$ and the estimate \eqref{eq:trace-class-smoothing-estimate}
holds for $\varphi(D^2-\gl)^{-k}A$.


\textup{2.} If $k>(\dim M+a)/2$ then $A (D^2-\gl)^{-k}, (D^2-\gl)^{-k}A$ are trace class operators
and the estimate \eqref{eq:trace-class-resolvent-estimate} holds with $B=I$. 
\end{proposition}
\begin{proof}
From
\begin{equation}\label{eq:ML20080403-A7}
       (D^2-\gl) \varphi (D^2-\gl)^{-k} A= [D^2,\varphi] (D^2-\gl)^{-k}A+\varphi (D^2-\gl)^{-k+1}A
\end{equation}
we infer since $\varphi A=0$
\begin{equation}\begin{split}
           \varphi (&D^2-\gl)^{-k} A\\
     &= (D^2-\gl)^{-1}
                 \begin{cases} [D^2,\varphi] (D^2-\gl)^{-k}A,& k=1,\\
                               [D^2,\varphi] (D^2-\gl)^{-k}A+\varphi (D^2-\gl)^{-k+1}A,& k>1.
  \end{cases}
  \end{split}
 \end{equation}
Applying Proposition \plref{p:ML20080403-A1}.3 to the right hand side we 
inductively obtain the first assertion.

To prove the second assertion we choose a cut--off function $\varphi\in\cinfz{M}$ with $\varphi=1$ on
$\pi_1(\supp A)$. Then we apply Proposition \plref{p:ML20080403-A1}.2 to $\varphi(D^2-\gl)^{-k}A$ and
the proved first assertion to $(1-\varphi)(D^2-\gl)^{-k}A$ to reach the conclusion.
\end{proof}

For the following Proposition it is crucial that we are precise about
{\domain}s of operators:

\begin{definition}\label{d:commutator-compact} By $\Diff^d(M,W)$
 we denote the space of differential operators of order $d$
 acting on the sections of $W$.
Given a differential operator $A\in\Diff^a(M,W)$ we say
that the commutator $[D^2,A]$ has compact support if
\begin{enumerate}
\item $A$ and $A^t$ map the domain $\dom(D^k)$ into the domain $\dom(D^{k-a})$ for $k\ge a$ and
\item the differential expression $[D^2,A]$ has compact support.
\end{enumerate}
\end{definition}

The main example we have in mind is where $D$ is a Dirac type operator
on a complete manifold and $A$ is multiplication by a smooth function
$\varphi$ such that $d\varphi$ has compact support. Then $[D^2,\varphi]$
has compact support in the above sense.

\begin{proposition}\label{p:ML20081104-A19} 
Let $A\in\Diff^a(M,W)$ be a differential operator
such that $[D^2,A]$ has compact support and is of order $\le a+1$. 
Then for $k>\dim M+a$ the
commutator $[A,(D^2-\gl)^{-k}]$ is trace class and
\begin{equation}\label{eq:commutator-resolvent-estimate}
      \| [A,(D^2-\gl)^{-k}]\|_1= O(|\gl|^{(\dim M+a-1)/2-k+0}),\quad \text{as } \gl\to\infty\text{ in }\gL.
\end{equation}
\end{proposition}
\begin{proof} Note first that since $A$ and $A^t$ map $\dom(D^k)$ into $\dom(D^{k-a})$
the commutator  $[A,(D^2-\gl)^{-k}]$ is defined as a linear operator on $L^2(M;W)$
and we have the identity
\begin{equation}\label{eq:ML20081104-A20}
[A,(D^2-\gl)^{-k}] = \sum_{j=1}^k (D^2-\gl)^{-j} [D^2,A](D^2-\gl)^{-k+j-1}.
\end{equation}
Since $k>\dim M+a$ we have in each summand $j>(\dim M+a+1)/2$ or
$k-j+1>(\dim M+1+1)/2$. Say in the first case we apply 
Proposition \plref{p:ML20080403-A2}
to $\|(D^2-\gl)^{-j}[D^2,A]\|_1$ and the Spectral Theorem
to estimate $\|(D^2-\gl)^{-k+j-1}\|$ and find
\begin{equation*}
  \begin{split}
      \|(D^2-\gl)^{-j} &[D^2,A](D^2-\gl)^{-k+j-1}\|_1\\
     &\le \|(D^2-\gl)^{-j} [D^2,A]\|_1 \|(D^2-\gl)^{-k+j-1}\|_{\infty}\\
     &\le O(|\gl|^{(\dim M+a+1)/2-j+0})\;\cdot\; O(|\gl|^{-k+j-1})\\
     &=O(|\gl|^{(\dim M+a-1)/2-k+0}).\qedhere
  \end{split}
\end{equation*}
\end{proof}

\subsection{Heat kernel estimates}
\label{ss:HeatKernelEstimates}
From Propositions \plref{p:ML20080403-A1}, \plref{p:ML20080403-A2} we can
derive short and large times estimates for the heat operator $e^{-tD^2}$.
We write 
\begin{equation}\label{eq:contour-integral}
   \begin{split}
         e^{-tD^2}&=\frac{1}{2\pi i}\int_\gamma e^{-t\gl} (D^2-\gl)^{-1} d\gl\\
                 &=\frac{t^{-k} k!}{2\pi i}\int_{\gamma} e^{-t \gl} (D^2-\gl)^{-k-1} d\gl,
    \end{split}
\end{equation}
where integration is over the contour sketched in Figure \ref{fig:ContourNonFredholm}.
The notation $O(t^{\ga-0})$, $O(t^\infty)$ as $t\to 0+$ resp. 
$O(t^{\ga-0}), O(t^{-\infty})$ as $t\to\infty$
is defined analogously to the corresponding notation for $\gl\in\gL$ in
the previous Section.
\FigContourNonFredholm
 
We infer from Propositions 
\plref{p:ML20080403-A1}, \plref{p:ML20080403-A2}

\begin{proposition}\label{p:ML20080403-A3}
Let $A,B\in\pdo^\bullet(M,W)$ be pseudodifferential operators of order $a,b$ 
with compact support.


\textup{1.} For $t>0$ the operators $A e^{-tD^2},  e^{-tD^2} A$ 
are trace class operators. For $t_0,\eps>0$ there is a constant
$C(t_0,\eps)>0$ such that for all $1\le p\le \infty$
we have the following estimate in the Schatten $p$--norm
\begin{equation}\label{eq:trace-heat-estimate}
      \| Ae^{-tD^2}\|_{p}\le C(t_0,\eps)\; t^{-a/2-\frac{\dim M+\eps}{2p}},\quad 0<t\le t_0.
\end{equation}
Note that $C(t_0,\eps)$ is independent of $p$. 
The same estimate holds for $\|e^{-tD^2}A\|_p$.
%


\textup{2.} Denote by $\pi_1, \pi_2:M\times M\to M$ the projection onto the first resp. second factor
and assume that $\pi_2(\supp A)\cap \pi_1(\supp B)=\emptyset$. Then 
\begin{equation}\label{eq:heat-smoothing-estimate}
        \| Ae^{-tD^2}B\|_1=O(t^\infty), \quad 0<t<t_0,
\end{equation}
with $N$ arbitrarily large.

\textup{3.} Let $\varphi\in\cinf{M}$ be a smooth function such that $\supp d\varphi$ is 
compact. Moreover suppose that
$\supp \varphi\cap \pi_1(\supp A)=\emptyset$. Then 
the estimate \eqref{eq:heat-smoothing-estimate} also holds
for $\varphi e^{-tD^2}A$.
\end{proposition}

\begin{proof} 1. From Proposition \ref{p:ML20080403-A2} and the contour integral
\eqref{eq:contour-integral} we infer the inequality \eqref{eq:trace-heat-estimate}
for $p=1$. For $p=\infty$ it follows from the Spectral Theorem.
The H\"older inequality implies the following interpolation inequality\sind{interpolation inequality}
for Schatten norms
\begin{equation}\label{eq:Hoelder-interpolation}
     \|T\|_p=\Tr(|T|^p)^{1/p}\le \|T\|_\infty^{1-1/p}\|T\|_1^{1/p}, \quad 1\le p\le \infty.
\end{equation}
From this we infer \eqref{eq:trace-heat-estimate}.

The remaining claims follow immediately from the contour integral 
\eqref{eq:contour-integral} and the corresponding resolvent estimates.
\end{proof}

For the next result we 
assume additionally that $D$ is a Fredholm operator and we denote by $H$ the orthogonal
projection onto $ \Ker D$. $H$ is a finite rank smoothing operator. Let $c:=\min \spec_{\ess} D^2$
be the bottom of the essential spectrum of $D^2$.
Then $e^{-tD^2}(I-H)=e^{-tD^2}-H$ can again be expressed in terms of a contour integral as in
\eqref{eq:contour-integral} where the contour is now depicted in Figure \plref{fig:ContourFredholm}.
\FigContourFredholm

This allows to make large time estimates. The result is as follows:

\begin{proposition}\label{p:ML20080403-A4} 
Assume that $D$ is Fredholm and let $A\in \pdo^a(M,W)$ be a pseudodifferential
operator with compact support. Then for any $0<\delta<\inf\spec_{\ess} D^2$
and any $\eps>0$ there is a constant $C(\delta,\eps)$ such that
for $1\le p\le \infty$
\begin{equation}
        \| A e^{-tD^2}(I-H)\|_{p} \le C(\delta,\eps)\; t^{-a/2-\frac{\dim M+\eps}{2p}} e^{-t\delta},\quad 0<t<\infty.
\end{equation}
\end{proposition}      
\begin{proof} For $t\to 0+$ the estimate follows from Proposition \plref{p:ML20080403-A3}.1. 

For $t\to\infty$ and $p=1$ the estimate follows from Proposition \plref{p:ML20080403-A1}
and \eqref{eq:contour-integral} by taking the contour as in Figure \ref{fig:ContourFredholm}.
For $p=\infty$ the estimate is a simple consequence of the Spectral Theorem.
The general case then follows again from 
the interpolation inequality \eqref{eq:Hoelder-interpolation}.
\end{proof}

Finally we state the analogue of Proposition \ref{p:ML20081104-A19}
for the heat kernel.\sind{heat kernel} 

\begin{proposition}\label{p:ML20081104-A20} 
Let $A\in\Diff^a(M,W)$ be a differential operator
such that $[D^2,A]$ has compact support
(in the sense of Definition \textup{\ref{d:commutator-compact}})
and is of order $\le a+1$. 

Then for $t>0$ the operator $[A,e^{-tD^2}]$ is of trace class. For $t_0,\eps>0$
there is a constant $C(t_0,\eps)$ such that for all $1\le p\le \infty$
we have the following estimate in the Schatten $p$--norm
\begin{equation}\label{eq:commutator-heat-estimate}
      \| [A,e^{-tD^2}]\|_{p}\le C(t_0,\eps)\; t^{-a/2-\frac{\dim M-1+\eps}{2p}},\quad 0<t\le t_0;
\end{equation}
$C(t_0,\eps)$ is independent of $p$. 

If $D$ is a Fredholm operator then for any $0<\delta<\inf\spec_{\ess} D^2$
and any $\eps>0$ there is a constant $C(\delta,\eps)$ such that
for $1\le p\le \infty$
\begin{equation}
        \| [A, e^{-tD^2}(I-H)]\|_{p} \le C(\delta,\eps)\; 
       t^{-a/2-\frac{\dim M-1+\eps}{2p}} e^{-t\delta},\quad 0<t<\infty.
\end{equation}
\end{proposition}      
\begin{proof} For $p=1$ this follows from Proposition
\ref{p:ML20081104-A19} and the contour integral representation
\eqref{eq:contour-integral} by taking the contours as in Figure \ref{fig:ContourNonFredholm}
for $t\to 0+$ and as in Figure \ref{fig:ContourFredholm} in the Fredholm case
as $t\to\infty$.
For $p=\infty$ the estimates are a simple consequence of the Spectral Theorem.
The general case then follows from 
the interpolation inequality \eqref{eq:Hoelder-interpolation}.
\end{proof}

\subsection{Estimates for the \textup{JLO} integrand}
\label{ss:EstimatesJLOIntegrand}
Recall that we denote the standard $k$--simplex by 
$\Delta_k:=\bigsetdef{(\sigma_0,...,\sigma_k)\in\R^{k+1}}{\sigma_j\ge 0, \sigma_0+...+\sigma_k=1}$. 
Furthermore, recall the notation \eqref{eq:ML20090128-3}.

\begin{proposition}\label{p:multiple-heat-estimate}
Let $A_j\in\Diff^{d_j}(M;W)$, $j=0,...,k$, be $D^{d_j}$--bounded differential operators on
of order $d_j$ on $M$; let $d:=\sum_{j=0}^k d_j$ be the sum of their
orders.
Furthermore, assume that $\supp A_{j_0}$ is compact for at least one index $j_0$.

\textup{1.} For $t_0, \eps>0$ there is a constant $C(t_0,\eps)$ such that
for all $\sigma=(\sigma_0,...,\sigma_k)\in\Delta_k, \sigma_j>0$,
\begin{equation}\label{eq:multiple-heat-estimate-short}
\begin{split}
     \| A_0 &e^{-\sigma_0 t D^2}A_1\cdot\ldots\cdot A_k e^{-\sigma_k t D^2}\|_1\\
       &\le C(t_0,\eps)\; \Biggl(\prod_{j=0}^k \sigma_j^{-d_j/2} \Biggr)
       t^{-d/2-(\dim M)/2-\eps},\quad 0<t\le t_0.
     \end{split}
\end{equation}
In particular, if $d_j\le 1, j=0,...,k,$ then 
\begin{equation}\label{eq:integrated-multiple-heat-estimate-short}
     \|(A_0,...,A_k)_{\sqrt{t}D}\|=O( t^{-d/2-(\dim M)/2-0}),\quad t \to 0+.
\end{equation}

\textup{2.} Assume additionally that $D$ is Fredholm and denote by $H$ the orthogonal
projection onto $ \Ker D$. Then for $\eps>0$ and any $0<\delta<\inf\spec_{\ess} D^2$
there is a constant $C(\delta,\eps)$ such that for all $\sigma\in\Delta_k, \sigma_j>0$
\begin{equation}\label{eq:multiple-heat-estimate-large}
   \begin{split}
     \| A_0 &e^{-\sigma_0 t D^2}(I-H)A_1\cdot\ldots\cdot A_k e^{-\sigma_k t D^2}(I-H)\|_1\\
      &\le C(\delta,\eps)\; \Biggl(\prod_{j=0}^k \sigma_j^{-d_j/2} \Biggr)
       t^{-d/2-(\dim M)/2-\eps} e^{-t\delta},\quad \text{for \emph{all} } 0<t<\infty.
    \end{split}
  \end{equation}
In particular, $d_j\le 1, j=0,...,k,$ then
\begin{equation}\label{eq:integrated-multiple-heat-estimate-large}
    \begin{split} 
    \|(A_0&(I-H),...,A_k(I-H))_{\sqrt{t}D}\|\\
      &=O(t^{-d/2-(\dim M)/2-0} e^{-t\delta}),\quad \text{for \emph{all }} 0<t<\infty.
    \end{split}
\end{equation}
\end{proposition}
\begin{proof} We first reduce the problem to the case that all $A_j$ are compactly
supported. To this end choose $\varphi_{j_0-1},\varphi_{j_0}\in\cinfz{M}$ such that
$\supp \varphi_{j_0}\cap \supp(1-\varphi_{j_0-1})=\emptyset$ and such that
$\varphi_{j_0}A_{j_0}=A_{j_0}\varphi_{j_0}=A_{j_0}$. Decompose
$A_{j_0-1}=A_{j_0-1}\varphi_{j_0-1}+A_{j_0-1}(1-\varphi_{j_0-1})$.

First we show that the estimates \eqref{eq:multiple-heat-estimate-short}, \eqref{eq:multiple-heat-estimate-large}
hold if we replace $A_{j_0-1}$ by $A_{j_0-1}(1-\varphi_{j_0-1})$:

Case 1. Proposition \ref{p:ML20080403-A3}.3 gives
\begin{equation}\label{eq:ML20080428-A18}
 \|A_{j_0-1}(1-\varphi_{j_0-1})e^{-\sigma_{j_0-1} t D^2} \varphi_{j_0}\|_1\le C_{t_0} \sigma_{j_0-1}^Nt^{N},\quad \text{ for } \sigma_{j_0-1}t\le t_0.
\end{equation}
The operator norm of the other factors can be estimated using 
the Spectral Theorem, taking into account the $D^{d_j}$-boundedness of $A_j$:
\begin{equation}
   \| A_j e^{-\sigma_j tD^2}\|\le C_{t_0} (\sigma_j t)^{-d_j/2},
\quad \text{ for } \sigma_{j}t\le t_0.
 \end{equation}
Hence by the H\"older inequality
\begin{equation}
   \begin{split}
     \| A_0 &e^{-\sigma_0 t D^2}A_1\cdot\ldots\cdot A_{j_0-1}(1-\varphi_{j_0-1})e^{-\sigma_{j_0-1}tD^2}
                                \cdot A_k e^{-\sigma_k t D^2}\|_1\\
      &\le C_{t_0}  \Biggl(\prod_{j=0}^k \sigma_j^{-d_j/2} \Biggr)
       t^{-\sum\limits_{j=0}^k d_j+N},\quad 0<t\le t_0,
    \end{split}
  \end{equation}
which is even better than \eqref{eq:multiple-heat-estimate-short}.

Case 2 ($D$ Fredholm). 
From \eqref{eq:ML20080428-A18}, Proposition \ref{p:ML20080403-A4} 
and the fact that $H$ is a finite rank operator with $e^{-\xi D^2}H=H$ we infer
\begin{equation}\label{eq:ML20080425-1}
  \begin{split}
 \|A_{j_0-1}(1-\varphi_{j_0-1})&e^{-\sigma_{j_0-1} t D^2} (I-H)\varphi_{j_0}\|_1\\
       &   \le C_\delta e^{-\sigma_{j_0-1}t\delta},\quad \text{ for all } 0<t<\infty.
  \end{split}
\end{equation}
To the other factors we apply 
Proposition \ref{p:ML20080403-A4} with $p=\infty$:
\begin{equation}\label{eq:ML20080425-2}
   \| A_j e^{-\sigma_j tD^2}(I-H)\|\le C_{\delta}  
 (\sigma_jt)^{-d_j/2} e^{-\sigma_jt\delta},\quad 0<t<\infty.
 \end{equation}
The H\"older inequality combined with 
\eqref{eq:ML20080425-1},\eqref{eq:ML20080425-2} gives 
\eqref{eq:multiple-heat-estimate-large}.

Altogether we are left to consider $A_0,...,A_{j_0-1}\varphi_{j_0-1},A_{j_0},...A_k$
where now $A_{j_0-1}\varphi_{j_0-1}$ and $A_{j_0}$ are compactly supported. 
Continuing this way, also to the right of $j_0$, 
it remains to treat the case where \emph{each} $A_j$ 
has compact support.

Case 1. We apply H\"older's inequality for Schatten norms and Proposition \ref{p:ML20080403-A3}:
\begin{equation}\label{eq:ML20080426}
    \begin{split}
    \| A_0 e^{-\sigma_0 t D^2} &A_1\cdot\ldots\cdot A_k e^{-\sigma_k t D^2}\|_1\\
                          &\le \prod_{j=0}^k \|A_j e^{-\sigma_j t D^2}\|_{\sigma_j^{-1}}\\
       & \le C(t_0,\eps)\; \prod_{j=0}^k (\sigma_j t)^{-d_j/2-\frac{\dim M+\eps}{2}\sigma_j}\\
       & \le C(t_0,\eps)\; \Biggl(\prod_{j=0}^k \sigma_j^{-d_j/2}\Biggr)
               t^{-d/2-\frac{\dim M+\eps}{2}}, \quad 0<t\le t_0,
     \end{split}
\end{equation}
thanks to the fact that $\sigma\mapsto \sigma^{-\frac{\dim M+\eps}{2}\sigma}$ is bounded as $\sigma\to 0$.

Case 2. If $D$ is Fredholm we estimate 
\[\| A_0 e^{-\sigma_0 t D^2}(I-H)A_1\cdot\ldots\cdot A_k e^{-\sigma_k t D^2}(I-H)\|_1\]
using H\"older as in \eqref{eq:ML20080426} and apply Proposition \ref{p:ML20080403-A4}
to the individual factors:
\begin{equation}
\|A_j e^{-\sigma_j t D^2}(I-H)\|_{\sigma_j^{-1}}\le 
    C_\delta  (\sigma_j t)^{-d_j/2-\frac{\dim M+\eps}{2}\sigma_j}e^{-\sigma_jt\delta},
    \quad 0<t<\infty,
\end{equation}
to reach the conclusion.

Finally we remark that the inequalities \eqref{eq:integrated-multiple-heat-estimate-short}, \eqref{eq:integrated-multiple-heat-estimate-large}
follow by integrating the inequalities \eqref{eq:multiple-heat-estimate-short}, \eqref{eq:multiple-heat-estimate-large}
over the standard simplex $\Delta_k$. Note that 
$\int_{\Delta_k} (\sigma_0\cdot\ldots\cdot\sigma_k)^{-1/2}d\sigma<\infty$.
\end{proof}


\section{Comparison results}\label{s:comparison-results}
Let $M_j$, $j=1,2,$ be complete riemannian manifolds with cylindrical ends,
\cf~Proposition \ref{p:bsmooth}. Assume that $M_1$ and $M_2$ share a common 
cylinder component $(-\infty,0]\times Z$.
That is, if $M_j= (-\infty,0]\times Z_j\cup_{Z_j} X_j, j=1,2$, 
then $Z$ is (after a suitable identification) a common (union of) 
connected component(s) of $Z_1, Z_2$.  A typical example will be 
$M_2=\R\times Z$.

Suppose that $\dirac_j$ are formally self-adjoint Dirac 
operators (\emph{cf.}~Section \plref{s:qDirac}) 
on $M_j$, $j=1,2$ with 
${\dirac_1}\rest{(-\infty,0]\times Z}={\dirac_2}\rest{(-\infty,0]\times Z}
=:\dirac=\sfc(dx)\frac{d}{dx}+\diracbdy$. 
The operators $\dirac_j$ are supposed to act on sections of the hermitian
vector bundles $W_j$ such that 
${W_1}\rest{(-\infty,0]\times Z}={W_2}\rest{(-\infty,0]\times Z}$.
As explained in Section \ref{App:bdefbmet}, we identify 
$\Gamma^\infty \big( (-\infty,c]\times Z;W \big)$, $c\in \R$,
with the completed tensor product
$\cC^\infty \big( (-\infty,c] \big)\hatotimes\Gamma^\infty(Z;W)$ 
and accordingly
$\bcptGamma \big( (-\infty,c)\times Z;W \big)$ with 
$\bcptC \big( (-\infty,c) \big) \hatotimes \Gamma^\infty(Z;W)$.
We want to compare the
resolvents and heat kernels of $\dirac_j$ on the common cylinder 
$(-\infty,0]\times Z$.

To this end, we will make repeated use of Remark \ref{r:DiracCylinderFormulas} 
without mentioning it every time. There is an intimate relation between
the \emph{spectrum}, $\spec \diracbdy$, of the boundary operator $\diracbdy$
and the \emph{essential spectrum}, $\specess \dirac$, of $D$. We will only
need that  
\begin{equation}
  \label{eq:essspecbd}
  \inf\specess \dirac^2=\inf\spec\diracbdy^2.
\end{equation}
A proof of this can be found in \cite[Sec. 4]{Mul:EIM}.
Concerning notation, $\|T\|_p$ stands for the $p$-th Schatten norm
of an operator $T$ acting on a Hilbert space $\sH$ unless otherwise stated.

\begin{proposition}\label{t:heat-resolvent-comparison}
\textup{1.} Fix an open sector 
$\Lambda:=\bigsetdef{z\in\C^*}{\eps < \arg z < 2\pi-\eps}
\subset \C\setminus\R_+$ where $\eps > 0$.
Then on $(-\infty,c]\times Z,c<0,$ 
the difference of the resolvents $(\dirac_1^2-\gl)^{-1}-(\dirac_2^2-\gl)^{-1}, \gl\in\Lambda,$
is trace class. Moreover, for $N>0$ and $\gl_0\in\Lambda$ there is a constant
$C(c,N,\gl_0)$ such that
\begin{multline}\label{eq:ML20081028-2}
     \Bigl\|\bigl(  (\dirac_1^2-\gl)^{-1}-(\dirac_2^2-\gl)^{-1}\bigr)
     \rest{(-\infty,c]\times Z}  \Bigr\|_1\\
           \le C(c,N,\gl_0)\; |\gl|^{-N},\quad \text{ for }\gl\in \gL,\,
           |\gl|\ge |\gl_0|.
\end{multline}

\textup{2.} On $(-\infty,c]\times Z, c<0$, the difference of the heat kernels
\begin{equation}
     \bigl(\dirac_1^l e^{-t\dirac_1^2}-\dirac_2^l e^{-t\dirac_2^2}\bigr)
     \rest{(-\infty,c]\times Z}, \quad l\in\Z_+,
\end{equation}
is trace class for $t>0$. Moreover, for $N, t_0>0$ there is a constant
$C(c,l,N,t_0)$ such that 
\begin{equation}
     \Bigl\|\bigl(\dirac_1^l e^{-t\dirac_1^2}-\dirac_2^l e^{-t\dirac_2^2}\bigr)
     \rest{(-\infty,c]\times Z} \Bigr\|_1\\
         \le C(c,l,N,t_0)\; t^N,\quad 0<t\le t_0.
\end{equation}

\textup{3. } Assume in addition that $\dirac_1, \dirac_2$ are Fredholm operators and denote by $H_j$
the orthogonal projections onto $ \Ker \dirac_j, j=1,2$. 
Then for $0<\delta<\inf\specess \dirac^2$ there is a constant $C(c,\delta)$ such
that 
\begin{equation}
     \Bigl\|\bigl(\dirac_1^le^{-t\dirac_1^2}(I-H_1)- \dirac_2^le^{-t\dirac_2^2}(I-H_2)\bigr)\rest{(-\infty,c]\times Z}  \Bigr\|_1
         \le C(c,\delta) e^{-t\delta},
\end{equation}
for $0<t<\infty,\; l\in\Z_+$.
\end{proposition}

\begin{proof}
\textup{1. } We choose cut--off functions $\varphi,\psi\in\cinf{\R}$
such that
\begin{equation}\label{eq:cut-off-functions}
   \varphi(x)=\begin{cases} 1, & x\le 4/5 c,\\
                            0, & x\ge 3/5 c,\\
               \end{cases}
\qquad
    \psi(x)=\begin{cases}   1, & x\le 2/5 c,\\
                            0, & x\ge 1/5 c,\\
               \end{cases}
\end{equation}
see Figure \ref{fig:CutOff}.
\FigCutOff
We have $\psi\varphi=\varphi$, $\supp d\psi\cap\supp \varphi=\emptyset$. Consider 
\begin{equation}
     R_{\psi,\varphi}(\gl):= \psi \bigl((\dirac_1^2-\gl)^{-1}-(\dirac_2^2-\gl)^{-1}\bigr) \varphi,
\end{equation}
viewed as an operator acting on sections over $M_1$. Then
\begin{equation}\label{eq:ML20081028-1}
       (\dirac_1^2-\gl)R_{\psi,\varphi}(\gl)=[\dirac_1^2,\psi](\dirac_1^2-\gl)^{-1}\varphi-[\dirac_2^2,\psi](\dirac_2^2-\gl)^{-1}\varphi,
\end{equation}
where again $[\dirac_2^2,\psi](\dirac_2^2-\gl)^{-1}\varphi$ is considered as acting on sections over $M_1$.
Since the operators $[D_j^2,\psi], j=1,2$ have compact support which is disjoint from
the support of $\varphi$ we may apply
Proposition \plref{p:ML20080403-A2} to the r.h.s. of \eqref{eq:ML20081028-1}
to infer that $R_{\psi,\varphi}(\gl)$ is trace class and that the estimate \eqref{eq:ML20081028-2}
holds.

\textup{2. } 
This follows from 1. and the contour integral representation \eqref{eq:contour-integral}
of the heat kernel. Cf. Proposition \plref{p:ML20080403-A3}.

\textup{3. } This follows from 1. and \eqref{eq:contour-integral} by taking the contour as in Figure \ref{fig:ContourFredholm},
page \pageref{fig:ContourFredholm}.
\end{proof}

We recall from Section \ref{App:bdefbmet} the notation $\bcptdiff^a((-\infty,0)\times Z;W)$
\eqref{eq:defbcptdiff}. In what follows, the subscript \emph{cpt} indicates that
the objects are supported away from $\{0\}\times Z$; it does \emph{not} indicate compact support.
The support of objects in $\bcptdiff^a$, and other spaces having the $\emph{cpt}$ decoration, may 
be unbounded towards $\{-\infty\}\times Z$. \emph{Compactly} supported
functions resp. sections are written with a $c$ decoration, e.g. $\cC_c^\infty$ resp. $\Gamma_c^\infty$. 

\begin{theorem}\label{t:JLO-comparison}
Let $A_j\in \bcptdiff^{d_j}((-\infty,0)\times Z;W)$ be
$b$--differential operators of order $d_j, j=0,...,k$ which are supported
away from $\{0\}\times Z$. Let $d:=\sum_{j=0}^kd_j$ be the sum of their
orders.

\textup{1. } For $t_0,N>0$ there is a constant $C(t_0,N)$ such that for all
$\sigma\in \Delta_k, \sigma_j>0$
\begin{multline}\label{eq:ML20090128-1}
   \Bigl\| A_0 e^{-\sigma_0 t \dirac_1^2}\cdot...\cdot A_l \bigl( e^{-\sigma_l t \dirac_1^2}-e^{-\sigma_l t \dirac_2^2}
\bigr) A_{l+1}\cdot...\cdot A_k e^{-\sigma_k t \dirac_1^2}\Bigr\|_1\\
    \le C(t_0,N)\Biggl(\prod_{j=0, j\not=l}^k \sigma_j^{-d_j/2} \Biggr) (\sigma_l t)^N,\quad 
             0< t<t_0.
\end{multline}

\textup{2. }  Assume in addition that $\dirac_1, \dirac_2$ are Fredholm operators and denote by $H_j$
the orthogonal projections onto $ \Ker \dirac_j, j=1,2$. 
Then for $0<\delta<\inf \specess \dirac^2$ 
and all $\sigma\in\Delta_k,\sigma_j>0,$
\begin{multline}\label{eq:ML20090128-2}
   \Bigl\| A_0 e^{-\sigma_0 t \dirac_1^2}(I-H_1)\cdot...\cdot A_l \bigl( e^{-\sigma_l t \dirac_1^2}(I-H_1)-e^{-\sigma_l t \dirac_2^2}(I-H_2)\bigr)A_{l+1}\cdot...\\ ...\cdot A_k e^{-\sigma_k t \dirac_1^2}(I-H_1)\Bigr\|_1\\
    \le C(\delta)\Biggl(\prod_{j=0, j\not=l}^k \sigma_j^{-d_j/2} \Biggr)  e^{-t\delta},\quad 
             0< t<\infty.
\end{multline}
\end{theorem}
\begin{remark}\label{rem:EstimateOptimality}
With some more efforts one can show that the factors $\Bigl(\prod_{j=0}^k \sigma_j^{-d_j/2}\Bigr)$ on the right
hand sides of the estimates \eqref{eq:ML20090128-1}, \eqref{eq:ML20090128-2},
and also below in \eqref{eq:b-multiple-heat-estimate-short}, \eqref{eq:b-multiple-heat-estimate-large}
are obsolete. But since this is not needed for our purposes we prefer a less
cumbersome presentation.
\end{remark} 
\begin{proof} 
First note that by Proposition \ref{Prop:PropbSob} 
the operator $A(i+\dirac)^{-a}$ is bounded 
for $A\in \bcptdiff^a((-\infty,0)\times Z;W)$ and hence
the Spectral Theorem implies that for $t_0>0$ there is a $C(t_0)$ such that
\begin{equation}\label{eq:ML20081028-4}
      \|A e^{-t\dirac^2}\|_\infty\le C(t_0) t^{-a/2},\quad 0<t<t_0.
\end{equation}
If $D$ is Fredholm then for $0<\delta<\inf\specess D^2$ and $t_0>0$ 
there is a $C(t_0,\delta)$ such that 
\begin{equation}\label{eq:ML20081028-5}
      \|A e^{-t\dirac^2}(I-H)\|_\infty\le C(\delta) e^{-t\delta},\quad t_0<t<\infty.
\end{equation}
\eqref{eq:ML20081028-4} and \eqref{eq:ML20081028-5} together imply that
for $0<\delta<\inf\specess D^2$ there is a $C(\delta)$ such that
\begin{equation}\label{eq:ML20081028-6}
      \|A e^{-t\dirac^2}(I-H)\|_\infty\le C(\delta) t^{-a/2} e^{-t\delta},\quad 0<t<\infty.
\end{equation}

The first claim now follows
from the second assertion of Theorem \ref{t:heat-resolvent-comparison}.
Namely, with some $c<0$ such that $\supp A_j\subset (-\infty,c]\times Z$ we find
\begin{equation}\label{eq:ML20081028-7}
\begin{split}
     \Bigl\| A_0 &e^{-\sigma_0 t \dirac_1^2}\cdot...\cdot A_l \bigl( e^{-\sigma_l t \dirac_1^2}-e^{-\sigma_l t \dirac_2^2}
\bigr) A_{l+1}\cdot...\cdot A_k e^{-\sigma_k t \dirac_1^2}\Bigr\|_1\\
     &\le \Biggl(\prod_{j=0,j\not=l}^k \| A_j e^{-\sigma_j t \dirac_1^2} \|_\infty \Biggr) \;
             \bigl\| A_l (i+\dirac)^{-d_l}\bigr\|_\infty \;\cdot\\
     & \qquad \cdot \Bigl\|\bigl((i+\dirac_1)^{d_l} e^{-\sigma_l t \dirac_1^2}-(i+\dirac_2)^{d_l}e^{-\sigma_l t \dirac_2^2}\bigr)_{|(-\infty,c]\times Z}\Bigr\|_1\\
     &\le C(t_0,N) \Biggl(\prod_{j=0,j\not=l}^k \sigma_j^{-d_j/2} \Biggr) \; \sigma_l^N t^{N-d/2},\quad 0<t<t_0.
\end{split}
\end{equation}
Here we have used \eqref{eq:ML20081028-4}. Since $\sigma_j<1$ an inequality which
is valid for $0<\sigma_j t<t_0$ is certainly also valid for $0<t<t_0$.
Since $N$ is arbitrary \eqref{eq:ML20081028-7} proves the first claim.

2. Similarly, using the third assertion of Theorem \ref{t:heat-resolvent-comparison} and
\eqref{eq:ML20081028-6}
\begin{equation}
\begin{split} 
 \Bigl\|& A_0 e^{-\sigma_0 t \dirac_1^2}(I-H_1)\cdot...\cdot A_l \bigl( e^{-\sigma_l t \dirac_1^2}(I-H_1)-e^{-\sigma_l t \dirac_2^2}(I-H_2)\bigr)A_{l+1}\cdot...\\ 
     &\qquad ...\cdot A_k e^{-\sigma_k t \dirac_1^2}(I-H_1)\Bigr\|_1\\
     &\le \Biggl(\prod_{j=0,j\not=l}^k \| A_j e^{-\sigma_j t \dirac_1^2} (I-H_1)\|_\infty \Biggr)\;
         \bigl\| A_l(i+\dirac)^{-d_l}\bigr\|_\infty \cdot\\
     &\quad \cdot  \Bigl\|\bigl((i+\dirac_1)^{d_l} e^{-\sigma_l t \dirac_1^2}(I-H_1)-
            (i+\dirac_2)^{d_l} e^{-\sigma_l t \dirac_2^2}(I-H_2)\bigr)_{|(-\infty,c]\times Z}\Bigr\|_1\\
     &\le C(\delta) \Biggl(\prod_{j=0,j\not=l}^k \sigma_j^{-d_j/2} \Biggr) \; t^{-d/2} e^{-t\delta}.
\end{split}
\end{equation}
Together with the proved short time estimate and the fact that the $H_j$ are of finite rank and thus of trace class
we reach the conclusion.
\end{proof}

\section{Trace class estimates for the model heat kernel}
\label{s:estimates-model-heat-kernel}
\nind{Delta@$\Delta_\R$}
We consider the heat kernel\sind{heat kernel} of the Laplacian $\Delta_\R$ on the real
line
\begin{equation}
  \label{eq:heat-kernel}
  k_t(x,y)=\frac{1}{\sqrt{4\pi t}} e^{-(x-y)^2/4t}.
\end{equation}
By slight abuse of notation we will denote the operator of multiplication
by $\id_\R$ by $X$. We want to estimate the Schatten norms
of $e^{\ga |X|}e^{-t\Delta_\R}e^{\beta |X|}$.
Before we start with this let us note for future reference:
\begin{equation}\label{eq:gaussian-integral}
   \begin{split}
   \int_\R e^{-\frac{z^2}{\gl t}-\beta z} dz&= \sqrt{\pi \gl t}\; e^{\gl \beta^2 t/4},\\
   \int_\R e^{-\frac{z^2}{\gl t}+\beta |z|} dz&\le 2\sqrt{\pi \gl t}\; e^{\gl \beta^2 t/4},
   \end{split}
\quad \beta\in\R;\; \gl,t>0.
\end{equation}
Furthermore, we will need the well-known \sindp{Schur's test}:
\begin{lemma}[{\cite[Thm.~5.2]{HalSun:BIO}}]\label{l:Schurs-test}
Let $K$ be an integral operator on a measure space $(\Omega,\mu)$ with kernel
$k:\Omega\times \Omega\to \C$. Assume that there are positive measurable functions 
$p,q: \Omega \to (0,\infty)$ such that
\begin{equation}\label{eq:Schurs-test}
  \begin{split}
       \int_X |k(x,y)| p(y) d\mu(y)&\le C_p\; q(x),\\      
       \int_X |k(x,y)| q(x) d\mu(x)&\le C_q\; p(y).
  \end{split}
\end{equation}
Then $K$ is bounded in $L^2(\Omega,\mu)$ with $\|K\|\le \sqrt{C_pC_q}$.
\end{lemma}
\comment{\begin{proof} 
For $f,g\in L^2(\Omega ,\mu)$ we estimate using Cauchy--Schwarz and 
\eqref{eq:Schurs-test}
\begin{equation}
  \begin{split}
    \int_X\int_X &|f(x)| |k(x,y)| |g(y)| d\mu(x)d\mu(y)\\
    &=  \int_X\int_X \Big( \frac{|f(x)|}{\sqrt{q(x)}}  \sqrt{p(y)} |k(x,y)|^{1/2}\Big)\cdot\\ 
    &\qquad  
      \cdot \Big( \sqrt{q(x)}|k(x,y)|^{1/2} \frac{|g(y)|}{\sqrt{p(y)}}\Big) d\mu(x)d\mu(y)\\
    &\le \Big( \int_X\int_X \frac{|f(x)|^2}{q(x)}  |k(x,y)| p(y) d\mu(y) d\mu(x)\Big)^{1/2}
      \cdot\\
    &\qquad \cdot\Big( \int_X\int_X 
      \frac{|g(y)|^2}{p(y)}  |k(x,y)| q(x) d\mu(x) d\mu(y)\Big)^{1/2}\\
    &\le \sqrt{C_pC_q} \|f\| \; \|g\|.\qedhere
  \end{split}
\end{equation}
\end{proof}} 
Now we can prove the following estimate. 
\begin{proposition}\label{p:heat-estimate-scalar-laplace}
Let $\Delta_\R=-\frac{d^2}{dx^2}$ be the Laplacian on the real line.
Then for $\ga>\gb>0,t>0, l\in \Z_+$, the integral operator
$e^{-\ga|X|}\bigl(\frac{d}{dx}\bigr)^l e^{-t\Delta_\R} e^{\beta|X|}$
with the (not everywhere smooth) kernel
\begin{equation}\label{eq:heat-estimate-scalar-laplace-1}
   \frac{1}{\sqrt{4\pi t}}e^{-\ga |x|} \partial_x^l e^{-(x-y)^2/4t +\gb |y|}
\end{equation}
is $p$-summable for $1\le p\le \infty$ and we have the estimate
   \begin{multline}\label{eq:heat-estimate-scalar-laplace-2}
  \|e^{-\ga|X|}\bigl(\frac{d}{dx}\bigr)^l e^{-t\Delta_\R} e^{\beta|X|}\|_p\\
      \le \bigl( c_1 t^{-\frac{l}{2}-\frac{1}{2p}}+c_2 t^{-\frac{1}{2p}}\bigr)
              (\ga-\gb)^{-1/p} e^{(\ga^2+\gb^2)t},\quad 0<t<\infty
   \end{multline}
with (computable) constants $c_1,c_2$ independent of $\ga,\gb,p,t$.
\end{proposition}
\begin{remark}\label{rem:EstimateOptimality-a}We are not striving to make these estimates optimal. 
We chose to formulate them in such a way that they are sufficient 
for our purposes and such that the proofs 
do not become too cumbersome.
\end{remark}
\begin{proof} We will prove this estimate for $p=\infty$ using Schur's test
Lemma \plref{l:Schurs-test} and for $p=2$ by estimating the $L^2$-norm
of the kernel. The case $p=1$ will then follow from the semigroup property of
the heat kernel. The result for general $p$ follows from the cases $p=1$ and $p=\infty$
and the interpolation inequality \eqref{eq:Hoelder-interpolation}.

The case $l\ge 2$ can easily be reduced to the case $l\in\{0,1\}$ in view of the 
identity
\begin{equation}\label{eq:ML20090127-1}
  \partial_x^{2k} e^{-t\Delta_\R} = (-\Delta_\R)^k e^{-t\Delta_\R}= 
  \partial_t^k e^{-t\Delta_\R}.
\end{equation}

\subsection*{The case $p=\infty$} We apply Schur's test with $p(x)=q(x)=1$. We will make frequent
use of \eqref{eq:gaussian-integral} without explicitly mentioning it all the time.
\begin{equation}\label{eq:heat-estimate-scalar-laplace-3}
   \begin{split}
        \frac{1}{\sqrt{4\pi t}} &\int_\R e^{-\ga|x| -(x-y)^2/4t +\gb |y|}dy\\
       &\le \frac{1}{\sqrt{4\pi t}}e^{-\ga |x|}\int_\R e^{-z^2/4t +\gb |x|+\gb|z|}dz\\
       & \le e^{(\gb-\ga)|x|}2 e^{\gb^2 t}\le 2 e^{\gb^2 t}.
   \end{split}
\end{equation}
Reversing the roles of $x$ and  $y$ one gets similarly
\begin{equation}\label{eq:heat-estimate-scalar-laplace-4}
   \begin{split}
        \frac{1}{\sqrt{4\pi t}} &\int_\R e^{\gb|y| -(x-y)^2/4t -\ga |x|}dx\\
       &\le \frac{1}{\sqrt{4\pi t}}e^{\gb |y|}\int_\R e^{-z^2/4t -\ga |y|+\ga |z|}dz\\
       & \le e^{(\gb-\ga)|y|}2 e^{\ga^2 t}\le 2 e^{\ga^2 t}.
   \end{split}
\end{equation}
This proves the result for $l=0$ and $p=\infty$.
In the case $l=1$ the integral
\begin{equation}\label{eq:heat-estimate-scalar-laplace-5}
         \frac{1}{\sqrt{4\pi t}} \int_\R \frac{|x-y|}{2t}e^{-\ga|x| -(x-y)^2/4t +\gb |y|}dy
\end{equation}
is estimated similarly.

\subsection*{The case $p=2$} We estimate the $L^2$-norm of the  kernel on $\R\times\R$ by:
\begin{equation}\label{eq:heat-estimate-scalar-laplace-6}
  \begin{split}
     \frac{1}{4\pi t} &\int_\R\int_\R e^{-2\ga|x| -\frac{(x-y)^2}{2t}+2\beta|y|} dx dy\\
           &\le   \frac{1}{4\pi t} \int_\R  e^{-2\ga|x|}\int_\R e^{-\frac{z^2}{2t}+2\beta|z|+2\beta|x|} dz dx\\
       &\le \frac{1}{\sqrt{2\pi t}}\int_\R e^{-2(\ga-\gb)|x|} e^{2\gb^2 t} dx\\
       &= \frac{1}{\sqrt{2\pi t}}\frac{1}{\ga-\gb}e^{2\gb^2 t},
  \end{split}
\end{equation}
proving the result for $l=0$ and $p=2$. Again, the case $l=1$ is similar.

\subsection*{The case $p=1$} Put $c=(\ga+\gb)/2$. Then the semigroup property of the heat 
kernel gives
\begin{equation}
  \begin{split}
      \| e^{-\ga|X|} &e^{-t\Delta_\R} e^{\gb|X|}\|_1\\
           &\le   \| e^{-\ga|X|} e^{-t/2\Delta_\R} e^{c|X|}\|_2 
            \| e^{- c |X|} e^{-t/2\Delta_\R} e^{\gb |X|}\|_2,
  \end{split}
\end{equation}
and using the proved case $p=2$ gives the result for $p=1$.
\end{proof}

The previous Proposition and standard estimates for the heat kernel\sind{heat kernel} 
on closed manifolds
(\cf~the Section \ref{sec:basic-estimates}) immediately give the following result 
for the heat kernel of the model Dirac operator on the cylinder.

\begin{proposition}\label{p:heat-estimate-model-laplace} 
Let $Z$ be a compact closed manifold and $\dirac=\Gammabdy\bigl(\frac{d}{dx}+\Abdy\bigr)$ 
a Dirac operator on the cylinder $M=\R\times Z$ 
(\cf~Remark \plref{r:DiracCylinderFormulas}). 
Furthermore, let $Q\in \bcptdiff^q((-\infty,0)\times Z;W)$ a \textup{b}-differential operator
of order $q$ with support in some cylindrical end $(-\infty, c)\times Z$.
Then for $\ga>\gb>0,t>0$ the integral operator $e^{-\ga|X|}Q e^{-t\dirac^2} e^{\beta|X|}$
with kernel
\begin{equation}\label{eq:heat-estimate-model-laplace-1}
   \frac{1}{\sqrt{4\pi t}}e^{-\ga |x|} Q_x e^{-(x-y)^2/4t +\gb |y|}e^{-tA^2}
\end{equation}
is $p$-summable for $1\le p\le \infty$.
Furthermore, for $\eps>0$, $t_0>0$, there is a constant $C(\eps,t_0)$, such that
for $1\le p\le\infty$, $0<t<t_0$, $0<\beta<\ga$
\begin{equation}
   \|e^{-\ga|X|} Q e^{-t \dirac^2} e^{\beta|X|}\|_p
       \le  C(\eps,t_0) (\ga-\gb)^{-1/p} 
             t^{-\frac{\dim M+\eps}{2p}-\frac q2}.
\end{equation}
If in addition the operator $A$ is invertible then for $0<\delta<\inf\spec A^2$
and $\eps>0$ there are constants $C_j(\delta,\eps)$, $j=1,2$ 
such that for $1\le p\le \infty,$
$0<t<\infty$, $0<\gb<\ga$ we have the estimate
\begin{multline}
\label{eq:heat-estimate-model-laplace-2}
   \|e^{-\ga|X|} Q e^{-t \dirac^2} e^{\beta|X|}\|_p \\
       \le \bigl( C_1(\delta,\eps) t^{-\frac{q}{2}}+C_2(\delta,\eps)\bigr)(\ga-\gb)^{-1/p} 
          t^{-\frac{\dim M+\eps}{2p}} e^{(\ga^2+\gb^2-\delta)t}.
\end{multline}
\end{proposition}
For the definition of $\bcptdiff^q, \bdiff^q$ see Proposition \plref{prop:DefEqbDiff} and Eq.~\eqref{eq:defbcptdiff}.

\begin{proof} 
By Proposition \ref{prop:DefEqbDiff} we may write $Q$ as a sum of operators of the form 
\begin{equation}\label{eq:dDiffOp-normalForm}
f(x,p) P \Bigl(\frac{d}{dx}\Bigr)^l 
\end{equation}
with
\begin{itemize}
\item $f\in\bcptGamma((-\infty,c)\times Z;\End W)$,
\item $P\in\Diff^{b-l}(Z;W\rest{Z})$ a differential operator of order $b-l$
on $Z$ which is \emph{constant} in $x$-direction.
\end{itemize}
Note that $e^{-\ga|X|}$ commutes with $f$. Furthermore, $f$ is uniformly bounded.
Thus
\begin{equation}\label{eq:ML20090127-2}
  \Bigl\| e^{-\ga|X|} f P \partial_x^l e^{-t\dirac^2} e^{\gb |X|}\Bigr\|_p
  \le \|f\|_\infty \Bigl\|e^{-\ga|X|} P\partial_x^l e^{-t\dirac^2} e^{\gb |X|}\Bigr\|_p.
\end{equation}
Inside the $p$-norm is now a tensor product of operators
\begin{equation}
  e^{-\ga|X|}  P \partial_x^l e^{-t\dirac^2} e^{\gb |X|}
  = \bigl( e^{-\ga|X|} \partial_x^l e^{-t\Delta_\R} e^{\gb|X|}\bigr)
                 \otimes  \bigl(P e^{-tA^2}\bigr).
\end{equation}
Since the $p$-norm of a tensor product is the product of the $p$-norms 
the claim follows from Proposition \ref{p:heat-estimate-scalar-laplace} and 
standard elliptic estimates for the closed manifold $Z$
(Propositions \ref{p:ML20080403-A3}, \ref{p:ML20080403-A4}).
\end{proof}

\section[Trace class estimates]{
Trace class estimates for the \textup{JLO} integrand on manifolds with cylindrical ends}
\label{s:estimates-JLO-cylindrical}

The heat kernel estimate for the Dirac operator on the model cylinder from Proposition 
\plref{p:heat-estimate-model-laplace}
together with the comparison result in Section \ref{s:comparison-results} 
will now be used to obtain trace class estimates for 
\textup{b}-differential operators similar to the one in Proposition 
\plref{p:multiple-heat-estimate} if the indicial operator of at least
one of the operators $A_0,\ldots,A_k$ vanishes. Let us mention here that
in the following we will use the notation introduced in
Subsection \ref{ss:JLODO}, in particular Eq.~\eqref{eq:ML20090128-3}.

\begin{proposition}\label{p:multiple-heat-estimate-b}
Let $M=(-\infty,0]\times Z\cup_Z X$, where $X$ is a compact manifold with
boundary,  be a complete manifold with cylindrical ends 
and let $\dirac$ be a Dirac operator on $M$. 
Let $A_0,...,A_k\in\bdiff(M;W)$ be \textup{b}-differential operators of order
$d_0,...,d_k; d:=\sum_{j=0}^k d_j$.
Assume that for at least one index
$l\in \{0,...,k\}$ the indicial family of $A_l$ vanishes.
Then for $t>0$, $\sigma\in\Delta_k$ the operator
\begin{equation}
   A_0 e^{-\sigma_0 t \dirac^2}A_1\cdot\ldots\cdot A_k e^{-\sigma_k t \dirac^2}
\end{equation}
is trace class. Furthermore, there are the following estimates:

\paragraph*{1} For $t_0>0$, $\eps>0$ there is a constant $C(t_0,\eps)$ such that
for all $\sigma=(\sigma_0,...,\sigma_k)\in\Delta_k$, $\sigma_j>0$,
\begin{equation}\label{eq:b-multiple-heat-estimate-short}
\begin{split}
     \| A_0 &e^{-\sigma_0 t \dirac^2}A_1\cdot\ldots\cdot A_k e^{-\sigma_k t \dirac^2}\|_1\\
       &\le C(t_0,\eps)\; \Biggl(\prod_{j=0}^k \sigma_j^{-d_j/2} \Biggr)
       t^{-d/2-(\dim M)/2-\eps},\quad 0<t\le t_0.
     \end{split}
\end{equation}
In particular, if $d_j\le 1, j=0,...,k,$ then 
\begin{equation}\label{eq:b-integrated-multiple-heat-estimate-short}
     \|( A_0,...,A_k)_{\sqrt{t}\dirac}\|=O( t^{-d/2-(\dim M)/2-0}),\quad t \to 0+.
\end{equation}

\paragraph*{2} Assume additionally that $\dirac$ is Fredholm and denote by $H$ the orthogonal
projection onto $ \Ker \dirac$. Then for $\eps>0$ and any $0<\delta<\inf\spec_{\ess} \dirac^2$
there is a constant $C(\delta,\eps)$ such that for all $\sigma\in\Delta_k, \sigma_j>0$
\begin{equation}\label{eq:b-multiple-heat-estimate-large}
   \begin{split}
     \| A_0 &e^{-\sigma_0 t \dirac^2}(I-H)A_1\cdot\ldots\cdot A_k e^{-\sigma_k t \dirac^2}(I-H)\|_1\\
      &\le C(\delta,\eps)\; \Biggl(\prod_{j=0}^k \sigma_j^{-d_j/2} \Biggr)
       t^{-d/2-(\dim M)/2-\eps} e^{-t\delta},\quad \text{for \emph{all} } 0<t<\infty.
    \end{split}
  \end{equation}
In particular, if $d_j\le 1, j=0,...,k,$ then 
\begin{equation}\label{eq:b-integrated-multiple-heat-estimate-large}
    \begin{split} 
    \|( A_0&(I-H),...,A_k(I-H))_{\sqrt{t}\dirac}\|\\
                &\le \tilde C(\eps,\delta)\; 
      t^{-d/2-(\dim M)/2-\eps} e^{-t\delta},\quad \text{for \emph{all }} 0<t<\infty.
    \end{split}
\end{equation}
\end{proposition} 
\begin{proof}
We first reduce the problem to a problem on the cylinder $(-\infty,0]\times Z$.

Write $A_j=A_j^{(0)}+A_j^{(1)}$ where $A_j^{(0)}$ has compact support and $A_j^{(1)}$
is supported on $(-\infty,c]\times Z$ for some $c<0$. 

Then we split $A_0e^{-\sigma_0t \dirac^2}A_1\cdot\ldots\cdot A_ke^{-\sigma_k t\dirac^2}$
(resp. $A_0e^{-\sigma_0t \dirac^2}(I-H)A_1\cdot\ldots\cdot A_ke^{-\sigma_k t\dirac^2}(I-H)$)
into a sum of terms obtained by decomposing  $A_j=A_j^{(0)}+A_j^{(1)}$.

To the summands involving at least one term $A_j^{(0)}$ we apply Proposition
\ref{p:multiple-heat-estimate}. To the remaining summand involving only $A_j^{(1)}$
we first apply the comparison Theorem \ref{t:JLO-comparison}
with $\dirac_1=\dirac$ and $\dirac_2=\Gammabdy \frac{d}{dx}+\diracbdy
=\Gammabdy\bigl(\frac{d}{dx}+A\bigr)$,
where $\dirac_2$ acts on sections over the cylinder $\R\times Z$;
\emph{cf.}~Remark \ref{r:DiracCylinderFormulas}.

Hence it remains to prove the claim for the cylinder $M=\R\times Z$ where
each $A_j$ is supported on $(-\infty,c]\times Z$ for some $c<0$.

For definiteness it is not a big loss of generality if
we assume that the indicial family of $A_0$ vanishes (I.e. $l=0$). 
Write $A_0=e^{-|X|}\tilde A_0$ with $A_0\in\bcptdiff^{d_0}((-\infty,0)\times Z;W)$.
Let $\gb_0,...,\gb_{k+1}$ be real numbers with $1\ge \gb_0>\gb_1>...>\gb_k>\gb_{k+1}=0$.

Let us assume that $\dirac$ is Fredholm and prove the claim 2. The proof of claim 1.
is similar and left to the reader.
H\"older's inequality yields
\begin{multline}
   \bigl \|A_0 e^{-\sigma_0 t \dirac^2}(I-H)A_1\cdot...\cdot A_k e^{-\sigma_k t \dirac^2}(I-H)\bigr\|_1         \\
      \le C \bigl\| e^{-\gb_0 |X|} \tilde A_0 e^{-\sigma_0 t\dirac^2}(I-H)e^{\gb_1|X|}\bigr\|_{\sigma_0^{-1}}\cdot   \\ 
          \cdot \prod_{j=1}^k \bigl \| e^{-\gb_j|X|} A_j e^{-\sigma_j t \dirac^2}(I-H)  e^{\gb_{j+1}|X|}\bigr\|_{\sigma_j^{-1}}.
\end{multline}
The individual factors are estimated by Proposition 
\ref{p:heat-estimate-model-laplace} and we obtain
for $0<t<\infty$:
\begin{equation}
\begin{split}
    ...&\le  \prod_{j=0}^k 
       \bigl( C_{1,j}(\delta,\eps,\gb) 
            (t\sigma_j)^{-d_j/2}+C_{2,j}(\delta,\eps,\gb)\bigr)\cdot\\
        &\qquad \cdot     
             t^{-\frac{\dim M+\eps}{2}\sigma_j} e^{(\gb_j^2+\gb_{j+1}^2-\delta)\sigma_j t}\\
        &\le C(\delta,\eps,\gb,\gamma) \Bigl(\prod_{j=0}^k \sigma_j^{-d_j/2}\Bigr) 
          t^{-d/2-\frac{\dim M+\eps}{2}} e^{(2\sum \gb_j^2 +\gamma-\delta)t},
\end{split}
\end{equation}
for any $\gamma>0$. The $\gamma>0$ is introduced to compensate $t^{-d/2}$ as $t\to\infty$.
Since we may choose $2\sum \gb_j^2 +\gamma$ as small as we please, the claim is proved.

The remaining cases are treated similarly.
\end{proof}

\section{Estimates for \textup{b}-traces}
\label{s:EstbTraces}

\comment{\marginpar{Due to the commutator formula for the \textup{b}-trace Thm. \eqref{t:bTraceAsTrace}
this section can be presented in a much shorter way. However, since the results
are secured now this steamlining does not have high priority; I wouldn't mind by the
way if somebody else would do it}

Finally we are going to estimate expressions of the form $\bTr(A_0e^{-\sigma_0t \dirac^2}A_1\cdot...\cdot
A_ke^{-\sigma_kt \dirac^2})$. 
As a preparation we first discuss the one dimensional case.

\begin{proposition}\label{p:1d-b-estimate}
Let $\Delta_\R=-\frac{d^2}{dx^2}$ on the real line. Furthermore, let 
$\varphi_j\in \bcptC(-\infty,0), j=0,...,k$, with $\varphi_j(x)=1$ for $x\le x_0$.
Moreover, given $d_0,...,d_k\in\Z_+$, $d=\sum_{j=0}^k d_j$.

Then for $\eps,\delta>0$ there is a constant $C(\eps,\delta)$ 
such that for all $\sigma\in\Delta_k, \sigma_j>0,$ 
\begin{equation}\label{eq:ML20081104-1}
   \begin{split}
\bTr(&\varphi_0 \partial_x^{d_0}e^{-\sigma_0t\Delta_\R}\varphi_1\cdot...\cdot
  \varphi_k \partial_x^{d_k}e^{-\sigma_kt\Delta_\R})\\
    &\le C(\eps,\delta)\Biggl(\prod_{j=0}^k \sigma_j^{-d_j/2-\eps} \Biggr)
           t^{-(a+1)/2-\eps}e^{t\delta},\quad 0<t<\infty.
   \end{split}
\end{equation}
\end{proposition}
\begin{proof} Fix a $c>0$ and consider the invertible operator $\Delta_\R+c$.
Although $\Delta_\R+c$ is not the square of a first order differential operator
it is clear that the basic estimates of Section \ref{sec:basic-estimates}
hold verbatim for $\Delta_\R+c$ instead of $D^2$. Note that since $\Delta_\R+c$
is invertible, the long time estimates of Section \ref{sec:basic-estimates}
hold with $H=0$. Finally, note that
\begin{equation}
\begin{split}
e^{-t c} \bTr((&\varphi_0 \partial_x^{d_0}e^{-\sigma_0t\Delta_\R}\varphi_1\cdot...\cdot
  \varphi_k \partial_x^{d_k}e^{-\sigma_kt\Delta_\R})\\
 \bTr((&\varphi_0 \partial_x^{d_0}e^{-\sigma_0t(\Delta_\R+c)}\varphi_1\cdot...\cdot
  \varphi_k \partial_x^{d_k}e^{-\sigma_kt(\Delta_\R+c)}),
\end{split}
\end{equation}
and hence we need to show that for $\eps>0$ and $0<\delta<c$ there is
a $C(\eps,\delta)$ such that for all $\sigma\in\Delta_k, \sigma_j>0,$ 
\begin{equation}\label{eq:ML20081104-2}
   \begin{split}
\bTr(&\varphi_0 \partial_x^{d_0}e^{-\sigma_0t(\Delta_\R+c)}\varphi_1\cdot...\cdot
  \varphi_k \partial_x^{d_k}e^{-\sigma_kt(\Delta_\R+c)})\\
    &\le C(\eps,\delta)\Biggl(\prod_{j=0}^k \sigma_j^{-d_j/2-\eps} \Biggr)
           t^{-(a+1)/2-\eps}e^{-t\delta},\quad 0<t<\infty.
   \end{split}
\end{equation}
We want to move the $\varphi_j$ to the left. If we commute $\varphi_j$
with a term $e^{-\sigma_{j-1}t(\Delta_\R+c)}$ we apply Proposition
\ref{p:ML20081104-A20} to the term
\begin{equation}
\varphi_0 \partial_x^{d_0}e^{-\sigma_0t(\Delta_\R+c)}\varphi_1...
[\varphi_j,e^{-\sigma_{j-1}t(\Delta_\R+c)}]...
  \varphi_k \partial_x^{d_k}e^{-\sigma_kt(\Delta_\R+c)}).
\end{equation}
This expression is of trace class and the trace norm can be estimated by
the r.h.s. of \eqref{eq:ML20081104-2}.

Similarly, if we commute $\varphi_j$ and a derivative $\partial_x$ we need
to estimate the trace norm of a term of the form
\begin{equation}
\varphi_0 \partial_x^{d_0}e^{-\sigma_0t(\Delta_\R+c)}\varphi_1...[\varphi_j,\partial_x]...
  \varphi_k \partial_x^{d_k}e^{-\sigma_kt(\Delta_\R+c)}).
\end{equation}
Since $[\varphi,\partial_x]$ is compactly supported the desired estimate
follows from Proposition \plref{p:multiple-heat-estimate}.

Moving all the $\varphi_j$ to the left leaves us to estimate
\begin{equation}
    \bTr(\underbrace{\varphi_0...\varphi_k}_{=:\chi(x)} \partial_x^a e^{-t(\Delta_\R+c)}).
\end{equation}
If $a$ is odd this is easily seen to be $0$ since the kernel of $\chi \partial_x^a 
e^{-t(\Delta_\R+c)}$ then vanishes on the diagonal. If $a$ is even we find
\begin{equation}
\begin{split}
    \bTr(\chi \partial_x^a e^{-t(\Delta_\R+c)})&= e^{-ct}\partial_t^{d/2} \bTr(\chi e^{-t\Delta_\R})\\
   &= e^{-ct}\partial_t^{d/2} \textup{p.f.-}\int_{-\infty}^0 \chi(x) \frac{1}{\sqrt{4\pi t}} dx\\
   &=e^{-ct}\partial_t^{d/2} (4\pi t)^{-1/2} \int_{-\infty}^0 \chi(x)-1 dx.
\end{split}
\end{equation}
This can certainly be estimated by the r.h.s. of \eqref{eq:ML20081104-2} and
the proof is complete.    
\end{proof}
}

Now we come to the main result of this chapter.

\begin{theorem}\label{t:main-b-estimate}
In the notation of Proposition \textup{\ref{p:multiple-heat-estimate-b}}
we now drop the assumption that the indicial family of one of the $A_l$ vanishes.
Then the following estimates hold:

\paragraph*{1} For $\eps>0, t_0>0$ there is a constant $C(\eps,t_0)$ such that
for all $\sigma=(\sigma_0,...,\sigma_k)\in\Delta_k, \sigma_j>0$,
\begin{equation}\label{eq:multiple-btrace-estimate-short}
\begin{split}
     \bTr\bigl(A_0 &e^{-\sigma_0 t \dirac^2}A_1\cdot\ldots\cdot A_k e^{-\sigma_k t \dirac^2}\bigr)\\
      &\le C(\eps,t_0) \Biggl(\prod_{j=0}^k \sigma_j^{-d_j/2-\eps} \Biggr)
       t^{-d/2-(\dim M)/2-\eps},\quad 0<t\le t_0.
     \end{split}
\end{equation}
In particular,
\begin{equation}\label{eq:integrated-multiple-btrace-estimate-short}
     |\blangle A_0,...,A_k\rangle_{\sqrt{t}\dirac}|=O( t^{-d/2-(\dim M)/2-\eps}),\quad t \to 0+.
\end{equation}
\paragraph*{2} If $\dirac$ is Fredholm then for $\eps>0$ and any $0<\delta<\inf\spec_{\ess} \dirac^2$
there is a constant $C(\eps,\delta)$ such that for all $\sigma\in\Delta_k, \sigma_j>0$
\begin{equation}\label{eq:multiple-btrace-estimate-large}
   \begin{split}
     \bTr\bigl(A_0 &e^{-\sigma_0 t \dirac^2}(I-H)A_1\cdot\ldots\cdot A_k e^{-\sigma_k t \dirac^2}(I-H)\bigr)\\
      &\le C(\eps,\delta) \Biggl(\prod_{j=0}^k \sigma_j^{-d_j/2-\eps} \Biggr)
       t^{-d/2-(\dim M)/2-\eps} e^{-t\delta},\quad \text{for \emph{all} } 0<t<\infty.
    \end{split}
  \end{equation}
In particular,
\begin{equation}\label{eq:integrated-multiple-btrace-estimate-large}
    \begin{split} 
    |\blangle A_0&(I-H),...,A_k(I-H)\rangle_{\sqrt{t}\dirac}|\\
                &\le \tilde C_{\delta,\eps} 
      t^{-d/2-(\dim M)/2-\eps} e^{-t\delta},\quad \text{for \emph{all }} 0<t<\infty.
    \end{split}
\end{equation}
\end{theorem}
\begin{proof} 
Arguing as in the proof of Proposition \ref{p:multiple-heat-estimate-b} 
we may assume that $\dirac$ is the model Dirac operator and that
$A_0,\dots,A_k\in\bcptdiff\bigl((-\infty,0)\times\partial M;W\bigr)$.

By Proposition \plref{t:bTraceAsTrace} we have
\begin{multline}\label{eq:ML200909171}
  \bTr\bigl(A_0 e^{-\sigma_0 t \dirac^2} A_1\cdot\ldots\cdot A_k e^{-\sigma_k t \dirac^2}\bigr)\\
= -\sum_{j=0}^k \Tr\bigl(x A_0 e^{-\sigma_0 t \dirac^2}\dots [\frac{d}{dx},A_j]\dots e^{-\sigma_k t \dirac^2}\bigr).
\end{multline}
Although multiplication by $x$ is not a \textup{b}-differential operator it is easy to
see that Proposition \plref{p:multiple-heat-estimate-b} still holds true for the summands on the
right of \eqref{eq:ML200909171}. The reason is that for any fixed $\eps>0$ the function
$x e^{-\eps |x|}$ is bounded. In fact any $0<\eps<1$ will do since the coefficients of
$[\frac{d}{dx},A_j]$ are $O(e^x)$ as $x\to -\infty$.
\end{proof}

\section{Estimates for the components of the entire \textup{b}-Chern character} \label{estim Chern}

We continue to work in the setting of a complete riemannian manifold with
cylindrical ends $M$, which is equivalent to a compact manifold 
with boundary with an exact \textup{b}-metric, \cf~Section \ref{App:bdefbmet}.
Furthermore, let $\dirac$ be a Dirac operator on $M$. 

\subsection{Short time estimates}
\begin{proposition} 
\label{p:ShortTimeEst}
The Chern characters $\bCh^k$ and $\bslch^k$ defined in 
\eqref{Eq:DefChern} and \eqref{Eq:DefSlChern} satisfy the following estimates
for $k\in\Z_+$:
\begin{equation}
\begin{split}
   \bCh^k (t \dirac)(a_0,\cdots,a_k) & = O ( t^{k-\dim M-0} ), \\
   \bslch^k(t \dirac,\dirac)(a_0,\cdots,a_k) & = O ( t^{k-\dim M-0} ),
\end{split}\quad t\to 0+,
\end{equation}
for $a_j\in \bcC(M), j=0,...,k$.

In particular,
\begin{enumerate}
\item $\lim\limits_{t\to 0+} \bCh^k(t\dirac)=0$ for $k>\dim M$.
\item The function $t\mapsto \bslch^k(t\dirac,\dirac)(a_0,...,a_k)$ is integrable on
$[0,T]$ for $T>0$ and $k>\dim M-1$.
\end{enumerate}
\end{proposition}
\begin{proof}
This follows immediately from Theorem \plref{t:main-b-estimate}.
\end{proof}

\subsection{Large time estimates}

Unless otherwise said we assume in this Subsection that 
$\diracbdy$ is invertible. Then $\inf\specess \dirac^2=\inf\spec\diracbdy^2>0$
(\cf~Eq.~\eqref{eq:essspecbd}) and hence $\dirac$ is a Fredholm operator.
We denote by $H$ the finite rank orthogonal projection onto the kernel of $\dirac$. 

\begin{lemma}\label{l:H1} Let $A_j\in\bdiff(M;W)$
 be \textup{b}-differential operators of order $d_j, j=0,\ldots,k;
 d:=\sum\limits_{j=0}^k d_j$ the total order.
Furthermore let $H_j=H$ or $H_j=I-H$, $j=0,...,k$ and assume that $H_j=H$ for at least one index $j$. 
Then for each $0<\delta<\inf\specess \dirac^2$
\begin{equation}
   \begin{split}
     \Bigl\|&A_0H_0e^{-\sigma_0 t\dirac^2} A_1 H_1e^{-\sigma_1 t \dirac^2}\ldots A_k H_k e^{-\sigma_k t \dirac^2}\Bigr\|_1\\
     &\le C(\delta) \Bigl(\prod_{l\in \{j_1,...,j_q\}} \sigma_{l}^{-d_l/2}\Bigr) t^{-d/2} e^{-(\sigma_{j_1}+...+\sigma_{j_q}) t\delta},\quad 0<t<\infty,
     \end{split}
\end{equation}
where $j_1,...,j_q$ are those indices with $H_j=I-H$, $d=\sum_{l\in \{j_1,...,j_q\}} d_l$.
\end{lemma}   
\begin{proof}  We pick an index $l$ with $H_l=H$. Then H\"older's inequality gives
\begin{equation} \begin{split}
\Bigl\|&A_0H_0e^{-\sigma_0 t\dirac^2} A_1 H_1e^{-\sigma_1 t \dirac^2}\ldots A_k H_k e^{-\sigma_k t \dirac^2}\Bigr\|_1\\
&\le \bigl\| A_l H_l e^{-\sigma_l t \dirac^2}\bigr\|_1 \quad \prod_{j\not=l} \bigl\| A_j H_j e^{-\sigma_j t \dirac^2} \bigr\|_\infty.
\end{split}
\end{equation}
The individual factors are estimated as follows: if $H_j=H$ then 
\begin{equation}
    \| A_j H e^{-\sigma_j t \dirac^2} \|_p\le \| A_j H\|_p,\quad \text{for } p\in\{1,\infty\}.
\end{equation}
If $H_j=I-H$ then by the Spectral Theorem and Proposition \ref{Prop:PropbSob} 
we have for  $0<\delta<\inf\spec_\ess \dirac^2$
\[
    \| A_j (I-H) e^{-\sigma_j t \dirac^2} \|_\infty \le  C(\delta,A_j) (\sigma_j t)^{-d_j/2}  e^{-\sigma_jt \delta}, \quad 0<t<\infty.\qedhere
\]
\end{proof}

The next Lemma is extracted from the proof of \cite[Prop. 2]{ConMos:TCC}.

\begin{lemma}\label{l:H2} Let $f:\R_+^q\to \C$ be a (continuous) rapidly decreasing function
of $q\le n$ variables. Then
\begin{equation}
   \begin{split}
    \lim_{t\to\infty} t^{2q}&\int_{\Delta_n} f(t^2\sigma_1,...,t^2\sigma_q)d\sigma\\
      &=\frac{1}{(n-q)!} \int_{\R_+^q} f(u) du
    \end{split}
\end{equation}    
\end{lemma}
\begin{proof} 
    Changing variables $u_j=t^2\sigma_j, j=1,...,q; u_j=\sigma_j, j=q+1,...,n$ we find
\begin{equation}
   \begin{split}
     t^{2q}\int_{\Delta_n}& f(t^2\sigma_1,...,t^2\sigma_j)d\sigma\\
         &= \int_{\{t^{-2}(u_1+...+u_q)+u_{q+1}+...+u_n)\le 1\}}
                 f(u_1,...,u_q)du\\
         &= \int_{t^2\Delta_q } f(u_1,...,u_q) \int_{\bigl(1-t^{-2}(u_1+...+u_q)\bigr)\Delta_{n-q}}du\\
         &= \frac{1}{(n-q)!}\int_{t^2\Delta_q}\bigl(1-t^{-2}(u_1+...+u_q)\bigr)^{n-q} f(u) du.
   \end{split}
\end{equation}
By assumption $f$ is rapidly decreasing, hence we may apply the Dominated Convergence Theorem
to reach the conclusion.
\end{proof}

\subsection{Estimating the transgressed \textup{b}-Chern character} 

\begin{proposition}\label{p:estimate-tChern-infty} For $k\ge 1$ and
 $a_0,...,a_k\in \bcC(M)$ we have
\begin{equation}
      \bslch^k(t\dirac,\dirac)(a_0,...,a_k)=\begin{cases}O(t^{-2}),&  k \text{ even },\\
                                           O(t^{-3}), & k \text{ odd },
                            \end{cases} t\to\infty.
\end{equation}
\end{proposition}
\newcommand{\roott}{\sqrt{t}}
\begin{proof} $\slch^k(t\dirac,\dirac)(a_0,...,a_k)$ is a sum of terms
of the form 
\begin{equation}
T=\blangle a_0, [t \dirac,a_1],...,[t \dirac,a_{i-1}],\dirac,[t \dirac,a_i],...,[t \dirac,a_k] \rangle_{t \dirac}.
\end{equation}
Writing $A_0=a_0$, $A_j=[\dirac,a_j], j=1,...,i-1,$ $A_j=[\dirac,a_{j-1}], j=i+1,...,k+1$,
$A_i:=\dirac$ we find
\begin{equation}
    T=t^k\sum_{H_j\in\{H,I-H\}} \blangle A_0 H_0,...,A_{k+1}H_{k+1}\rangle_{t \dirac},
\end{equation}
where the sum runs over all sequences $H_0,...,H_{k+1}$ with $H_j\in\{H,I-H\}$.
Since $H[\dirac,a_j]H=0, H\dirac=\dirac H=0$ (note $A_i=\dirac$ !) only terms containing no more than $[k/2]+1$ copies of $H$
can give a non-zero contribution. 

Consider such a nonzero summand with at least one index $j$ with $H_j=H$ and denote by $q$ the number
of indices $l$ with $H_l=I-H$. Then $q\ge [\frac{k+1}{2}]+1$ and we infer from Lemmas \plref{l:H1},\plref{l:H2}
\begin{equation}
      t^k \;\blangle A_0H_0,...,A_{k+1}H_{k+1} \rangle_{t\dirac}
=O(t^{k-2q})=\begin{cases}O(t^{-2}),& k \text{ even, }\\
                                          O(t^{-3}),  &k \text{ odd }.
                            \end{cases}
\end{equation}
We infer from Theorem \ref{t:main-b-estimate} that
the remaining summand with $H_j=I-H$ for all $j$ decays exponentially as $t\to\infty$ and we are done.
\end{proof}

\subsection{The limit as $t\to\infty$ of the \textup{b}-Chern character}
As in \cite{ConMos:TCC} we put
\begin{align}
            \varrho_H(A)&:=HAH,\\
\intertext{and}
             \go_H(A,B)&:= \varrho_H(AB)-\varrho_H(A)\varrho_H(B).
\end{align}

\begin{proposition}\label{p:ML20090928} Let $a_0,...,a_k\in\bcC(M)$. If $k$ is odd then
\begin{equation}
      \lim_{t\to\infty}\bCh^k(t \dirac)(a_0,...,a_k)=0.
\end{equation}
If $k=2q$ is even then
\begin{equation}
\begin{split}
    \lim_{t\to\infty}&\bCh^k(t \dirac)(a_0,...,a_{2q})\\
&= \frac{(-1)^q }{q!} \Str\bigl(\varrho_H(a_0)\go_H(a_1,a_2)\dots\go_H(a_{2q-1},a_{2q})\bigr)\\
&=: \kappa^{2q}(\dirac)(a_0,\dots,a_{2q}).
\end{split}
\end{equation}
\end{proposition}
Of course, since $H$ is of finite rank $\varrho_H(a_j)$ and $\go_H(a_j,a_{j+1})$ are of trace class.
\begin{proof}
As in the previous proof we abbreviate $A_0=a_0, A_j=[\dirac,a_j], j\ge 1$ and decompose
\begin{equation}
    \blangle A_0,A_1,...,A_k \rangle_{t \dirac}=t^{k}\sum_{H_j\in\{H,I-H\}} 
    \blangle A_0 H_0,...,A_{k}H_{k} \rangle_{t \dirac}.
\end{equation}
Since $H[\dirac,a_j]H=0$ only terms containing no more than $[k/2]+1$ copies of $H$ can give a nonzero
contribution. 

The term containing no copy of $H$ decreases exponentially in view of Theorem
\ref{t:main-b-estimate}.

Consider a term containing $q$ copies of $I-H$.
If the number $k+1-q$ of copies of $H$ is at least one but less than $k/2+1$, which is always the case
if $k$ is odd, then $q>k/2$ and hence in view of Lemmas  \plref{l:H1}, \plref{l:H2} 
\begin{equation}
t^k \blangle A_0 H_0,...,A_{k}H_{k} \rangle_{t \dirac}=O(t^{k-2q})=O(t^{-1}), \quad t\to \infty.
\end{equation}

If $n=2q$ is even there is exactly one term containing $k/2+1$ copies of $H$, namely
\begin{equation}
  \begin{split}
   t^{2q}\;\blangle &A_0 H,A_1(I-H),...,A_{2q}H \rangle_{t \dirac}\\
        &=t^{2q}\int_{\Delta_{2q}}
      \Tr\bigl(\gamma a_0 H e^{-\sigma_0 t^2 \dirac^2}[\dirac,a_1](I-H)\cdot...\cdot e^{-\sigma_{2q}t^2 \dirac^2}H\bigr)d\sigma.
    \end{split}
  \end{equation}
The integrand depends only on the $q=k/2$ variables $\sigma_1,\sigma_3,...,\sigma_{2q-1}$ and
so we infer from Lemma \plref{l:H2} that the limit as $t\to\infty$ equals
\begin{equation}
   \frac{1}{q!}\int_{\R_+^q}\Tr\bigl(\gamma a_0H[\dirac,a_1]e^{-u_1 \dirac^2}(I-H)[\dirac,a_2]H...e^{-u_{2q-1}\dirac^2}(I-H)[\dirac,a_{2q}]H\bigr) du.
\end{equation}
As in \cite[2.2]{ConMos:TCC} one shows that this equals
\[
    \frac{(-1)^q }{q!} \Str\bigl(\varrho_H(a_0)\go_H(a_1,a_2)...\go_H(a_{2q-1},a_{2q})\bigr).\qedhere
\]
\end{proof}

\chapter{The Main Results}
\label{chap:Main}
We are now in a position to establish the main results of this paper. 
After discussing in Section \ref{s: heat expansion} asymptotic
expansions for the \textup{b}-analogues of the Jaffe-Lesniewski-Osterwalder components, 
we construct in Section \ref{s:retracted-relative-cocycle}
the retracted relative cocycle representing the Connes--Chern character
in relative cyclic cohomology and compute its small and large scale limits.
Section \ref{s: geom pairing} derives the ensuing pairing formula 
with the $K$-theory, and discusses the geometric consequences. 
The final remark (Section \ref{s:conclude}) offers an explanation
for the restrictive eta-pairing which appears in the work of Getzler and  Wu.

\section{Asymptotic \textup{b}-heat expansions} \label{s: heat expansion}

\subsection{\textup{b}-Heat expansion}
Let $M$ be a complete riemannian manifold with cylindrical ends and let $\dirac$
be a Dirac operator on $M$ (\cf~Remark \plref{r:DiracCylinderFormulas},
Section \ref{App:bdefbmet}).

Let $Q\in\bdiff^q(M;W)$ be an auxiliary \textup{b}-differential operator of order $q$. It
is well-known (\emph{cf.~e.g.}~\cite{Gil:ITH}) that the Schwartz-kernel of the operator
$Qe^{-t\dirac^2}$ has a \emph{pointwise} asymptotic expansion
\begin{equation}\label{eq:ML20090122-4}
    (Qe^{-t\dirac^2})(x,p;x,p)
       \sim_{t\to 0+} \sum_{j=0}^\infty a_j(Q,\dirac)(x,p)\; t^{\frac{j-\dim M-q}{2}}.
\end{equation}
The problem is that in general neither $Qe^{-t\dirac^2}$ is of trace class nor
are the local heat invariants $a_j(Q,\dirac)$ integrable over the manifold.
Nevertheless we have the following theorem, which has been used implicitly by 
Getzler \cite{Get:CHA}. However, we could not find a reference where the result is 
cleanly stated and proved. Therefore, we provide here a proof for the convenience of the 
reader. 
\begin{theorem}\label{thm:bHeatExpansion}
Under the previously stated assumptions the \textup{b}-heat trace of $Qe^{-t\dirac^2}$ has
the following asymptotic expansion:\sind{bHeatexpansion@\textup{b}-heat expansion}
\begin{equation}
    \bTr\bigl(Q e^{-t\dirac^2}\bigr)
    \sim_{t\to 0+} \sum_{j=0}^\infty \int_{\tb M} \tr_{p}\bigl(a_j(Q,\dirac)(p)\bigr) d\vol(p) \; t^{\frac{j-\dim M-q}{2}}.
\end{equation}
\end{theorem}
The $\textup{b}$-integral $\int_{\tb M}$ was defined in
Section \plref{s: b-trace formula}, \emph{cf.}~Definition-Proposition
\plref{defprop15}.
\begin{proof}
We first write the operator $Q$ as a sum $Q=Q^{(0)}+Q^{(1)}$ of differential operators with
$Q^{(0)}\in\bcptdiff^q((-\infty,0)\times \pl M;W)$ and $Q^{(1)}$ a differential operator supported in the interior.
By standard elliptic theory (\cite{Gil:ITH}) $Q^{(1)}e^{-t\dirac^2}$ is trace class and
since the asymptotic expansion \eqref{eq:ML20090122-4} is uniform on compact subsets of $M$
the claim follows for $Q^{(1)}$ instead of $Q$. 

So it remains to prove the claim for an operator $Q\in \bcptdiff^q((-\infty,0)\times\pl M;W)$; for convenience
we write from now on again $Q$ instead of $Q^{(0)}$. 
Next we apply the comparison Theorem \plref{t:heat-resolvent-comparison} which allows us to assume
that $M=\R\times\partial M$ is the model cylinder, $\dirac=\sfc(dx)\frac{d}{dx}+D^\partial$, and $Q$ is supported
on $(-\infty,c)\times\partial M$ for some $c>0$.

Furthermore, we may assume that $Q$ is of the form \eqref{eq:dDiffOp-normalForm}.
Since the heat kernel\sind{heat kernel} of the model operator is explicitly known
(\cf~\eqref{eq:heat-kernel}) we have
\begin{equation}
\begin{split}
      \Bigl( f(x,p) &P \partial_x^l e^{-t\dirac^2}\Bigr)(x,p;y,q)\\
        &= \frac{1}{\sqrt{4\pi t}} \bigl(\partial_x^l e^{-(x-y)^2/4t}\bigr) 
         \bigl(P e^{-t A^2}\bigr)(p,q),\qquad \Abdy:=\Gammabdy \diracbdy.
\end{split}
\end{equation}
If $l$ is odd then by induction one easily shows that this kernel vanishes on the diagonal and
hence the {\btrace}  $\bTr(Qe^{-t\dirac^2})$ as well as all local heat coefficients
vanish, proving the Theorem in this case. So let $l=2k$ be even. Then
using \eqref{eq:ML20090127-1} and
since on the diagonal $\partial_t^k e^{-t\Delta_\R}(x,x)=\partial_t^k (4\pi t)^{-1/2}=:c_k t^{-1/2-k}$,
we have
\begin{equation}\label{eq:ML20090123-2}
\begin{split}
      \Bigl( f(x,p) &P \partial_x^l e^{-t\dirac^2}\Bigr)(x,p;x,p)\\
       &= f(x,p)\bigl(P e^{-t A^2}\bigr)(p,p) c_k \; t^{-1/2-k}\\
       &\sim_{t\to 0+} \sum_{j=0}^\infty f(x,p) a_j(P,A)(p) c_k \; t^{\frac{j-\dim M-q}{2}}.
\end{split}
\end{equation}
Comparing with \eqref{eq:ML20090122-4} we find for the heat coefficients $a_j(Q,D)$
\begin{equation}\label{eq:ML20090123-3}
     a_j(Q,D)(x,p)= f(x,p) a_j(P,A)(p) c_k .
\end{equation}
Furthermore, we have using Theorem \eqref{t:bTraceAsTrace}
\begin{equation}\label{eq:ML20090123-1}
      \bTr\bigl(Qe^{-t\dirac^2}\bigr) \\
       =\int_{-\infty}^0\int_{\partial M} \tr_{x,p}\Bigl(x\partial_x f(x,p) \bigl(Pe^{-tA^2}\bigr)(p,p)\Bigr)
        d\vol_{\partial M}(p) dx. 
\end{equation}
$\bigl(Pe^{-tA^2}\bigr)(p,p)$ has an \emph{$x$-independent} asymptotic expansion as $t\to 0+$.
Since $x\partial_x f(x,p)=O(e^{(1-\delta)x}), x\to -\infty$, uniformly in $p$, we can plug
the asymptotic expansion \eqref{eq:ML20090123-2} 
into \eqref{eq:ML20090123-1} and use \eqref{eq:ML20090123-3} to find
\begin{equation*}
\begin{split}
     \bTr&\bigl(Qe^{-t\dirac^2}\bigr)\\
      &\sim_{t\to 0+}\sum_{j=0}^\infty  \int_{-\infty}^0  \int_{\partial M} \tr_{x,p}\Bigl(x\partial_x f(x,p) a_j(P,A)(p)\Bigr)
        d\vol_{\partial M}(p) dx\; c_k\; t^{\frac{j-\dim M-q}{2}}\\
      &\sim_{t\to 0+}\sum_{j=0}^\infty \int_{\tb (-\infty,0)\times \partial M} \tr_{x,p}\bigl( a_j(Q,D)(x,p) \bigr) d\vol(x,p)
         \; t^{\frac{j-\dim M-q}{2}}.
\end{split}
\end{equation*}
The claim is proved.
\end{proof}

\subsection{The \textup{b}-trace of the \textup{JLO} integrand}
\newcommand{\nablad}{\nabla_{\mathsf{D}}}

To extend Theorem \plref{thm:bHeatExpansion} to expressions of the form 
$\bTr\Bigl(A_0 e^{-\sigma_0 t\dirac^2} A_1 e^{-\sigma_1 t\dirac^2} ... A_k e^{-\sigma_k t \dirac^2}\Bigr) $
we use a trick which was already applied successfully in the proof of the local index formula in noncommutative
geometry \cite{ConMos:LIF}. Namely, we successively commute $A_j e^{-\sigma_j t \dirac^2}$ and control
the remainder. We will need the estimates proved in Sections \plref{s:estimates-JLO-cylindrical}
and \plref{s:EstbTraces}. 

We first need to introduce some notation (\cf~\cite[Lemma 4.2]{Les:NRP}).  
For a \textup{b}-differential operator $B\in\bdiff(M;W)$ we put inductively
\begin{equation}\label{eq:IterCommutator}
   \nablad^0B:=B,\qquad \nablad^{j+1}B:=[\dirac^2,\nabla_D^jB].
\end{equation}
Note that since $\dirac^2$ has scalar leading symbol we have $\ord(\nabla_D^jB)\le j+\ord B$.
The following formula can easily be shown by induction. 
\begin{equation}\label{eq:HeatOpCommutator}
\begin{split}
     e^{-t\dirac^2}B& = \sum_{j=0}^{n-1} \frac{(-t)^j}{j!} \bigl(\nablad^j B \bigr) e^{-t\dirac^2}+\\
           &\qquad +\frac{(-t)^n}{(n-1)!}\int_0^1 (1-s)^{n-1} e^{-st\dirac^2}\bigl(\nablad^n B\bigr)
                           e^{-(1-s)t\dirac^2}ds.
\end{split}
\end{equation}

The identity \eqref{eq:HeatOpCommutator} easily allows to prove the following statement about
\semph{local heat invariants}, \cf~\cite{Wid:STC}, \cite{ConMos:CCN}, \cite{BloFox:APO}:

\begin{proposition}\label{p:JLOCommutatorAsymptotic}
 Let $A_0,...,A_k\in\bdiff(M;W)$ of order $d_0,...,d_k; d:=\sum_{j=0}^k d_j$. Then the Schwartz kernel
of $A_0 e^{-\sigma_0 t\dirac^2}A_1 e^{-\sigma_1 t\dirac^2}... A_k e^{-\sigma_k t \dirac^2}$
has a pointwise asymptotic expansion
\begin{equation}\label{eq:JLOCommutatorAsymptotic}
\begin{split}
    \Bigl(A_0 &e^{-\sigma_0 t\dirac^2} A_1 e^{-\sigma_1 t\dirac^2} ... A_k e^{-\sigma_k t \dirac^2}\Bigr)(p,p) \\
        &= \sum_{\ga\in\Z_+^k, |\ga|\le n} \frac{(-t)^{|\ga|}}{\ga!} \sigma_0^{\ga_1}(\sigma_0+\sigma_1)^{\ga_2}...(\sigma_0+...+\sigma_{k-1})^{\ga_k}
                       \cdot\\
             &\qquad\cdot \bigl( A_0 \nablad^{\ga_1}A_1 ... \nablad^{\ga_k} A_k e^{-t\dirac^2}\bigr)(p,p)+O_p(t^{(n+1-d-\dim M)/2}),\\
        &=: \sum_{j=0}^{n} a_j(A_0,...,A_k,\dirac)(p)\; t^{\frac{j-\dim M -d}{2}}+O_p(t^{(n+1-d-\dim M)/2}),
\end{split}
\end{equation} 
where $d=\sum\limits_{j=0}^k d_j$.
The asymptotic expansion is locally uniformly in $p$. Furthermore, it is uniform for $\sigma\in \Delta_k$.
\end{proposition}
Again we are facing the problem explained before Theorem \ref{thm:bHeatExpansion}. Still we will be able
to show that one obtains a correct formula by taking the {\btrace}  on the left and
partie finie integrals on the right of \eqref{eq:JLOCommutatorAsymptotic}:
\sind{partie finie}

\begin{theorem}\label{t:JLObCommutatorAsymptotic}
Under the assumptions of the previous Proposition \plref{p:JLOCommutatorAsymptotic} we have an asymptotic
expansion
\begin{equation}\label{eq:JLObCommutatorAsymptotic}
\begin{split}
    \bTr\Bigl(A_0 &e^{-\sigma_0 t\dirac^2} A_1 e^{-\sigma_1 t\dirac^2} ... A_k e^{-\sigma_k t \dirac^2}\Bigr) \\
        &= \sum_{\ga\in\Z_+^k, |\ga|\le n} \frac{(-t)^{|\ga|}}{\ga!} \sigma_0^{\ga_1}(\sigma_0+\sigma_1)^{\ga_2}...(\sigma_0+...+\sigma_{k-1})^{\ga_k}
                       \cdot\\
             &\qquad\cdot \bTr\bigl( A_0 \nablad^{\ga_1}A_1 ... \nablad^{\ga_k} A_k e^{-t\dirac^2}\bigr)+\\
             &\qquad+ O\Bigl(\bigl(\prod_{j=1}^k \sigma_j^{-d_j/2}\bigr)t^{(n+1-d-\dim M)/2}\Bigr),\\
        &= \sum_{j=0}^{n} \int_{\tb M} a_j(A_0,...,A_k,\dirac)d\vol \; t^{\frac{j-\dim M -d}{2}}+\\
              &\qquad +O\Bigl(\bigl(\prod_{j=1}^k \sigma_j^{-d_j/2}\bigr)t^{(n+1-d-\dim M)/2}\Bigr).
\end{split}
\end{equation} 
\end{theorem}
\begin{remark}
The O-constant in \eqref{eq:JLObCommutatorAsymptotic} is independent of $\sigma\in\Delta_k$. However, the
factor $\bigl(\prod_{j=1}^k \sigma_j^{-d_j/2}\bigr)$ inside the $O()$ causes some trouble because it is
integrable over the standard simplex $\Delta_k$ only if $d_1,...,d_k\le 1$.
We do not claim that this factor is necessarily there. It might be an artifact of the inefficiency of
our method. Cf. also Remarks \plref{rem:EstimateOptimality}, \plref{rem:EstimateOptimality-a}.
\end{remark}

\begin{proof} 
The strategy of proof we present here can also be used to prove Proposition \plref{p:JLOCommutatorAsymptotic}. 

Again by the comparison Theorem \plref{t:JLO-comparison} we may
assume that $\dirac$ is
the model Dirac operator and $A_0,...,A_k\in \bcptdiff((-\infty,0)\times\partial M;W)$.

Using Proposition \plref{t:bTraceAsTrace} we have
\begin{equation}\label{eq:ML20090126-1}
\begin{split}
    \bTr\Bigl( A_0 &e^{-\sigma_0t\dirac^2}A_1... A_k e^{-\sigma_k t\dirac^2}\Bigr)\\
                   &= -\bTr\Bigl(x\bigl[\frac{d}{dx}, A_0 e^{-\sigma_0t\dirac^2}A_1... A_k e^{-\sigma_k t\dirac^2}\bigr]\Bigr)\\
        &= -\sum_{j=0}^k \Tr\Bigl(x A_0 e^{-\sigma_0t\dirac^2}A_1...[\frac{d}{dx},A_j]... A_k e^{-\sigma_k t\dirac^2}\Bigr).
\end{split}
\end{equation}
$[\frac{d}{dx},A_j]$ is again in $\bcptdiff((-\infty,0)\times\partial M;W)$ and its indicial family vanishes.
Hence by Proposition \plref{p:multiple-heat-estimate-b}
all summands on the right are of trace class. Cf. also the comment at the end of the proof of
Theorem \plref{t:main-b-estimate}.

It therefore suffices to prove the claim for the summands on the right
of \eqref{eq:ML20090126-1}, i.e. for $\Tr\Bigl(x A_0 e^{-\sigma_0t\dirac^2}A_1... A_k e^{-\sigma_k t\dirac^2}\Bigr)$ 
where at least one of the $A_j$ has vanishing indicial family.

Applying \eqref{eq:HeatOpCommutator} to $A_1$ we get 
\begin{align}\label{eq:HeatOpCommutator-1}
     e^{-\sigma_0 t\dirac^2}A_1& = \sum_{j=0}^{n-1} \frac{(-\sigma_0 t)^j}{j!} \bigl(\nablad^j A_1 \bigr) e^{-\sigma_0t\dirac^2}+\\
           &\qquad +\frac{(-\sigma_0 t)^n}{(n-1)!}\int_0^1 (1-s)^{n-1} e^{-s\sigma_0 t\dirac^2}\bigl(\nablad^n A_1\bigr)
                           e^{-(1-s)\sigma_0 t\dirac^2}ds.\nonumber
\end{align}
Therefore we need to estimate the expression
\begin{equation}\label{eq:ML20090126-2}
  x(\sigma_0 t)^n (1-s)^{n-1} A_0 e^{-s\sigma_0 t\dirac^2}\bigl(\nablad^n A_1\bigr)
                           e^{-(1-s)\sigma_0 t\dirac^2} e^{-\sigma_1t\dirac^2}...A_k e^{-\sigma_kt\dirac^2}
\end{equation}
in the trace norm.

If the index $l$ for which the indicial family of $A_l$ vanishes is $0$ we write $A_0$ as $e^{x}\tilde A_0$
with $\tilde A_0\in\bcptdiff((-\infty,0)\times \partial M;W)$ and move $x e^x$ under the trace to the right.
This assures that Proposition \plref{p:multiple-heat-estimate-b} applies to 
$\bigl(\nablad^n A_1\bigr)e^{-(1-s)\sigma_0 t\dirac^2} e^{-\sigma_1t\dirac^2}...A_k e^{-\sigma_kt\dirac^2}x e^x$.

If $l\ge 1$ we just move $x$ under the trace to the right. After all w.l.o.g. we may assume that $l\ge 1$. 

Next we choose an integer $\gb$ such that $A_0(\dirac^2+I)^{-\gb}$ has order $\in \{ 0,1\}$. 
Then H\"older's inequality yields
\begin{multline}
   (\sigma_0 t)^n (1-s)^{n-1} \Bigl\| A_0 (I+\dirac^2)^{-\gb} e^{-s\sigma_0 t\dirac^2}(I+\dirac^2)^\gb\bigl(\nablad^n A_1\bigr)\\
                           e^{-(1-s)\sigma_0 t\dirac^2} e^{-\sigma_1t\dirac^2}...A_k e^{-\sigma_kt\dirac^2}x\Bigr\|_1\\
       \le  (\sigma_0 t)^n (1-s)^{n-1} (s\sigma_0t)^{-d_0/2+\gb}\Bigl\|(I+\dirac^2)^\gb\bigl(\nablad^n A_1\bigr)\\
                           e^{-(1-s)\sigma_0 t\dirac^2} e^{-\sigma_1t\dirac^2}...A_k e^{-\sigma_kt\dirac^2}x\Bigr\|_1
\end{multline}
To the remaining trace we apply Proposition \plref{p:multiple-heat-estimate-b} and obtain
\begin{equation}
\begin{split}
 ...&\le \sigma_0^n t^n (1-s)^{n-1} (s\sigma_0t)^{-d_0/2+\gb} C(t_0,\eps) \bigl((1-s)\sigma_0\bigr)^{-\gb-n/2-d_1/2}\\
    &\qquad \bigl(\prod_{j=2}^k \sigma_j^{-d_j/2}\bigr) t^{-d/2+d_0/2-\dim M/2-\eps -n/2-\gb}\\
    &\le C(t_0,\eps) s^{-1/2}(1-s)^{n/2-1-\gb-d_1/2} \sigma_0^{\frac{n-d_0-d_1}{2}} \bigl(\prod_{j=2}^k \sigma_j^{-d_j/2}\bigr)
           t^{\frac{n-d-\dim M}{2}-\eps}.
\end{split}
\end{equation}
If we choose $n$ large enough the right hand side is integrable in $s$ and we obtain the desired estimate. 

In the next step we apply \eqref{eq:HeatOpCommutator} to  $e^{-(\sigma_0+\sigma_1)t\dirac^2}$ and $A_2$.
Continuing this way we reach the conclusion after $k$ steps.
\end{proof}

%
%
\section{The Connes--Chern character of the relative Dirac class}
\label{s:retracted-relative-cocycle}
\subsection{Retracted Connes--Chern character} \label{ss:retracted-CC}
In this section we assume that $\dirac$ is a Dirac operator on a \textup{b}-Clifford 
bundle  $W\rightarrow M$  over the \textup{b}-manifold $M$ and 
$\dirac_t = t \dirac$ is a family of Dirac type operators.
We now have all tools to apply the method of \cite{ConMos:TCC} to convert
the entire relative {\CoChch}, which was constructed using the {\btrace},
into a finitely supported cocycle.

By integrating Eq.~\eqref{Eq:transgress}, one obtains for 
$0 < \varepsilon <t$ 
\begin{equation}
\begin{split} 
 \bCh^k & \, (\varepsilon \dirac) - \bCh^k (t \dirac)  
 = \, b \int_\varepsilon^t \bslch^{k-1} (s \dirac, \dirac) ds \\
 & + B  \int_\varepsilon^t \bslch^{k+1} (s \dirac, \dirac) ds
 + \int_\varepsilon^t \slch^k (s \diracbdy, \diracbdy)\circ i^* ds .
\end{split}
\end{equation}
$\Ch^\bullet(\diracbdy)$ satisfies the cocycle and transgression formul\ae\
Eq.~\eqref{eq:cocycle}, \eqref{eq:transgression}. Integrating
these we obtain
\begin{equation}\label{eq:ML200911262}
\begin{split} 
 \Ch^k & \, (\varepsilon \diracbdy) - \Ch^k (t \diracbdy)  
 = \, b \int_\varepsilon^t \slch^{k-1} (s \diracbdy, \diracbdy) ds \\
 & + B  \int_\varepsilon^t \slch^{k+1} (s \diracbdy, \diracbdy) ds.
\end{split}
\end{equation}

By Proposition \ref{p:ShortTimeEst} (1), 
the limit $\varepsilon \searrow 0$ exists for $k > \dim M$, and
\begin{equation}
\label{Eq:limit}
\begin{split}
&\lim_{\varepsilon \searrow 0} \bCh^k(\varepsilon \dirac)= 0,\\ 
&\lim_{\varepsilon \searrow 0} \Ch^{k-1}(\varepsilon \diracbdy)= 0,
\end{split}
\qquad \text{for all $k > \dim M $}.
\end{equation}
The second limit statement follows either from an obvious adaption of our calculations to the ordinary
trace or from \cite{ConMos:TCC}.
Hence one gets for $k > \dim M$
\begin{equation}
\begin{split}
\label{Eq:LittlebCapitalBChern}
 - \bCh^k (t \dirac) &=  
   \,  b \bTslch_t^{k-1} (\dirac)+
 B  \bTslch_t^{k+1} (\dirac)
 + \Tslch_t^k (\diracbdy)\circ i^*,\\
 - \Ch^{k-1} (t \diracbdy) &= \,  b \Tslch_t^{k-2} (\diracbdy)+
 B  \Tslch_t^{k} (\diracbdy),\\
\end{split}
\end{equation}
where
\begin{equation}\label{eq:ML200909263}
\begin{split}
  \bTslch_t^k (\dirac) &:= \int_0^t \slch ^k (s\dirac, \dirac) \, ds,\\
   \Tslch_t^{k-1} (\diracbdy) &:= \int_0^t \slch^{k-1} (s\diracbdy, \diracbdy) \, ds.
\end{split}
\end{equation}
The above integrals exist in view of Proposition \plref{p:ShortTimeEst} (2) even for $k\ge \dim M$.
From Eq.~\eqref{eq:ML200911262} and Theorem \ref{P:GETZLER}
we obtain for $k\ge \dim M$:
\begin{equation}
\begin{split}
   b\Bigl( &\bCh^k(t\dirac)+B\int_{\eps}^t \bslch^{k+1}(s\dirac,\dirac) ds\Bigr)\\
     &=-B \bCh^{k+2}(\eps \dirac) +\Ch^{k+1}(\eps\diracbdy)\circ i^*-b\int_{\eps}^t \slch^k(s\diracbdy,\diracbdy)ds\\
   &\longrightarrow -b \Tslch^k_t(\diracbdy)\circ i^*, \quad \eps\to 0+.
\end{split}
\end{equation}
Thus
\begin{equation}
\label{Eq:RellbcB}
\begin{split}
  b\Big( \bCh^k (t \dirac) + B\bTslch_t^{k+1} (\dirac) \Big) &= - b \Tslch_t^k (\diracbdy )\circ i^*,\\
  b\Big( \Ch^{k-1} (t \diracbdy) + B\bTslch_t^{k} (\diracbdy) \Big) &=0,
\end{split}\quad k\ge \dim M.
\end{equation}

Following \textnm{Connes-Moscovici} \cite{ConMos:TCC}, we define for $k\ge \dim M$
the Chern characters 
$\bch^k_t (\dirac)$, $\btch^k(\dirac)$ and $\ch^{k-1}_t (\diracbdy)$ by
\begin{align}
  \bch^k_t (\dirac) & = \sum_{j\geq 0} \bCh^{k-2j} (t \dirac) +
  B \bTslch^{k+1}_t (\dirac),\\
  \ch^{k-1}_t (\diracbdy) & = 
  \sum_{j\geq 0} \Ch^{k-2j-1} (t \diracbdy) +
  B \Tslch^{k}_t (\diracbdy), \\
  \btch^k_t(\dirac) \, &=\, \bch^k_t(\dirac) +\Tslch^k_t(\diracbdy)\circ i^*.
\end{align}

Let us now compute $(b+B) \bch^\bullet_t (\dirac)$. 
Using Eq.~\eqref{Eq:cocyclecond} and Eq.~\eqref{Eq:RellbcB} above, we write
\begin{equation}
\label{Eq:RelativityCheck}
\begin{split} 
  b\bch^k_t & (\dirac) + B\bch^k_t (\dirac)  = \\
  = \, & 
  \sum_{j \geq 1}   \Big( b\bCh^{k-2j} (t\dirac) + 
  B\bCh^{k-2j + 2} (t \dirac) \Big) \\
  & + b \Big( \bCh^k (t\dirac ) + B \bTslch_t^{k+1} (\dirac) \Big) \\
  = \, &
  \sum_{j \geq 1} \Ch^{k-2j +1} (t \diracbdy)\circ i^* - b \Tslch^k_t (\diracbdy)\circ i^* \\
  = \, &
  \sum_{j \geq 0} \Ch^{k-2j -1} (t \diracbdy)\circ i^* - b \Tslch^k_t (\diracbdy)\circ i^* \\
  = \, &
  \ch^{k-1}_t (\diracbdy)\circ i^* - B \Tslch^{k}_t (\diracbdy)\circ i^*
   - b \Tslch^k_t (\diracbdy) \circ i^*
    \\ 
   = \, & \ch^{k+1}_t (\diracbdy)\circ i^*,
\end{split}
\end{equation}
where the last equality follows from the second line of Eq.~\eqref{Eq:LittlebCapitalBChern}.

In conclusion
\begin{equation}
\begin{split}
      (b+B) \bch_t^k(\dirac)&= \ch_t^{k+1}(\diracbdy)\circ i^*\\
      (b+B) \btch_t^k(\dirac)&= \ch_t^{k-1}(\diracbdy)\circ i^*.
\end{split}
\end{equation}
Denoting by $\widetilde b,\widetilde B$ the relative Hochschild resp. Connes'
coboundaries, \emph{cf.}~Eq.~\eqref{Eq:DefCoBdrRelMixDer}, we thus infer
\sind{Hochschild (co)homology}
\begin{equation}
\begin{split}
    (\widetilde b+\widetilde B)\bigl(\bch_t^k(\dirac),\ch_t^{k+1}(\diracbdy)\bigr)&=0,\\
    (\widetilde b+\widetilde B)\bigl(\btch_t^k(\dirac),\ch_t^{k-1}(\diracbdy)\bigr)&=0, 
\end{split}
\end{equation}
\emph{i.e.} the pairs $(\bch_t^k(\dirac),\ch_t^{k+1}(\diracbdy))$ and $(\btch_t^k(\dirac),\ch_t^{k-1}(\diracbdy))$ are
relative cocycles in the direct sum of total complexes
\begin{equation}\label{eq:ML200909264}
\begin{split}
  \tot^k  & \, \mathcal B C^{\bullet,\bullet}
  (\mathcal C^\infty (M), \mathcal C^\infty (\partial M) ) := \\
  & \, := 
  \tot^k  \, \mathcal B C^{\bullet,\bullet}
  (\mathcal C^\infty (M)) \oplus 
  \tot^{k+1} \mathcal B C^{\bullet,\bullet}
  (\mathcal C^\infty (\partial M)).
\end{split}
\end{equation}

By Eq.~\eqref{Eq:RelativityCheck} we have
\begin{equation}
\begin{split}
     \Bigl(&\bch_t^k(\dirac)-\btch_t^k(\dirac),\ch_t^{k+1}(\diracbdy)-\ch_t^{k-1}(\diracbdy)\Bigr)\\
       &= \Bigl(-\Tslch_t^k(\diracbdy)\circ i^*,-(b+B)\Tslch_t^k(\diracbdy)\Bigr)\\
       &=\bigl(\widetilde b+\widetilde B\bigr)\Bigl(0,\Tslch_t^k(\diracbdy)\Bigr),
\end{split}
\end{equation}       
hence the two pairs differ only by a coboundary.

Next let us compute
$\big(\bch^{k+2}_t (\dirac),\ch^{k+3}_t (\diracbdy) \big)
  - S \big(\bch^k_t (\dirac),\ch^{k+1}_t (\diracbdy) \big)$
in the above relative cochain complex. Using Eq.~\eqref{Eq:LittlebCapitalBChern}
one checks immediately that
\begin{displaymath}
\begin{split}
  \bch^{k+2}_t  (\dirac) - \bch^k_t  (\dirac) = \, &
  \bCh^{k+2} (t \dirac) + B \bTslch^{k+3}_t (\dirac)
  -  B \Tslch^{k+1}_t (\dirac)\\
  = \, &
  - (b +B ) \Tslch^{k+1}_t (\dirac ) - \Tslch^{k+2}_t (\diracbdy)\circ i^*.
\end{split}
\end{displaymath}
From the second line of Eq.~\eqref{Eq:LittlebCapitalBChern} (or from \cite[Sec.~2.1]{ConMos:TCC})
\begin{displaymath}
  \ch^{k+3}_t  (\diracbdy) - \ch^{k+1}_t  (\diracbdy)
  = - (b+B) \Tslch^{k+2}_t (\diracbdy ),
\end{displaymath}
one thus gets 
\begin{equation}
\begin{split}
\big(\bch^{k+2}_t  &(\dirac),\ch^{k+3}_t (\diracbdy) \big)
  - S \big(\bch^k_t (\dirac),\ch^{k+1}_t (\diracbdy) \big)\\
 &= (\widetilde b + \widetilde B) 
 \big(-\bTslch^{k+1}_t (\dirac ),\Tslch^{k+2}_t (\diracbdy ) \big).
\end{split}
\end{equation}
Hence, the relative cocycles 
$\big(\bch^{k+2}_t  (\dirac),\ch^{k+3}_t (\diracbdy) \big)$
and $S\big(\bch^k_t (\dirac),\ch^{k+1}_t (\diracbdy) \big)$
are cohomologous.  Similarly, one gets
\begin{displaymath}
\begin{split}
  \bch^k_t & (\dirac) - \ch^k_\tau (\dirac) = \\
  = \, & \sum_{j\geq 0}
  \big( \bCh^{k-2j} (t \dirac) - \bCh^{k-2j} (\tau \dirac) \big)
  + B \int_\tau^t \bslch^{k+1} (s\dirac,\dirac)\, ds \\
  = \, & - (b+B) \sum_{j\geq 0} \int_\tau^t \bslch^{k-2j-1} 
  (s\dirac,\dirac)\, ds - \sum_{j\geq 0} 
  \int_\tau^t \slch^{k-2j} (s\diracbdy,\diracbdy)\, ds  ,
\end{split} 
\end{displaymath}
resp.
\begin{displaymath}
\begin{split}
  \ch^{k+1}_t & (\diracbdy) - \ch^{k+1}_\tau (\diracbdy) = 
   - (b+B) \sum_{j\geq 0} \int_\tau^t \slch^{k-2j} 
  (s\diracbdy,\diracbdy)\, ds,  
\end{split} 
\end{displaymath}
hence 
$\big(\bch^k_t (\dirac),\ch^{k+1}_t (\diracbdy) \big)$
and $\big(\bch^k_\tau (\dirac),\ch^{k+1}_\tau (\diracbdy) \big)$
are cohomologous in the total relative complex as well.
Thus, we have proved (1)-(3) of the following result.

\begin{theorem} \label{t: CC-character}
\begin{enumerate}
\item The pairs of retracted cochains 
 $\big(\bch^k_t (\dirac),\ch^{k+1}_t (\diracbdy) \big)$,  
$\big(\btch^k_t (\dirac),\ch^{k-1}_t (\diracbdy) \big)$, $t>0$, 
 $t>0$, 
 $k\ge m=\dim M, k-m\in 2\Z$ are cocycles in the relative total complex  
 $\tot^\bullet  \, \mathcal B C^{\bullet,\bullet}
  (\mathcal C^\infty (M), \mathcal C^\infty (\partial M) )$.
\item  They represent the same class 
in $HC^n(\mathcal C^\infty (M), \mathcal C^\infty (\partial M) )$ which
is independent of $t>0$.

\item They represent the same class
in $HP^\bullet(\mathcal C^\infty (M), \mathcal C^\infty (\partial M) )$
which is independent of $k$.

\item Denote by $\bomega_{\dirac}, \go_{\diracbdy}$ the local index forms of
$\dirac$ resp.~$\diracbdy$ \cite[Thm.~4.1]{BerGetVer:HKD}, \emph{cf.} Eq.~\eqref{eq: 0lim cocycle},
 \eqref{eq: A-currents} and see Eq.~\eqref{eq:ML200909265} below. 
Then one has a pointwise limit
\[
\lim_{t\to 0+}\big(\btch^k_t (\dirac),  \ch^{k-1}_t (\diracbdy) \big)
 =  \Big(\int_{\tb M} \bomega_{\dirac}\wedge \bullet, 
    \int_{\pl M} \omega_{\diracbdy} \wedge \bullet\Big).
\]
Moreover, 
$\big(\btch^k_t (\dirac),\ch^{k-1}_t (\diracbdy) \big)$
represents the {\CoChch} of 
$[\dirac] \in KK_m(C_0 (M) ; \C)= K_m(M,\partial M)$.
\end{enumerate}
\end{theorem}
The pointwise limit will be explained in the proof below. Up to normalization 
constants $\bomega_{\dirac}$ is the $\hat A$ form $\hat A(\bnabla^2_g)$ and 
$\omega_{\diracbdy}$ is the $\hat A$ form $\hat A(\nabla^2_{g_\pl})$ on the 
boundary. Note also that $\iota^* \go_{\dirac}=\go_{\diracbdy}$ .
\comment{does anyone have the energy
to make $\go_{\dirac}$ more explicit, I guess no}
\begin{proof}
It remains to prove (4). So consider 
\[a_0, a_1,\dots, a_j\in \bcC(M^\circ).\]
Using
Getzler's asymptotic calculus (\emph{cf.}~\cite{Get:POS}, \cite[\S 3]{ConMos:CCN}, and
\cite[Thm.~4.1]{BloFox:APO}) one shows that the \emph{local} heat invariants \sind{local heat
invariants}
of 
\[ a_0 e^{-\sigma_0 t\dirac^2}[\dirac,a_1] e^{-\sigma_1 t\dirac^2}\ldots 
[\dirac,a_j] e^{-\sigma_j t \dirac^2}\]
(\emph{cf.} Proposition \plref{p:JLOCommutatorAsymptotic}) satisfy 
\begin{multline}
 t^j\; \int_{\Delta_j}\str_{q,W_p}\Bigl( a_0 e^{-\sigma_0 t\dirac^2}[\dirac,a_1] 
 e^{-\sigma_1 t\dirac^2}\ldots [\dirac,a_j] e^{-\sigma_j t \dirac^2}\Bigr)(p,p) \,
 d\vol_{\bmet} (p) \\
 = \frac{1}{j!}\big(\bomega_{\dirac} \wedge a_0 d a_1\dots\wedge d a_j\big)_{|p}  
 + O(t^{1/2}),\quad t\to 0+.\label{eq:ML200909265}
\end{multline}
Here $\str_{q,W_p}$ denotes the fiber supertrace in $W_p$, $q$ indicates the Clifford 
degree of $\dirac$, \emph{cf.}~Section \plref{s:qDirac},
the factor $\frac{1}{j!}$ is the volume of the simplex $\Delta_j$.
This statement holds \emph{locally} on any riemannian manifold for any choice 
of a self-adjoint extension of a Dirac operator. So it holds for $\dirac$ and 
accordingly for $\diracbdy$.

From Theorem \plref{t:JLObCommutatorAsymptotic} and its well-known analogue for 
closed manifolds, \emph{cf.}~\cite[Sec.~4]{ConMos:TCC}, we thus infer
\begin{multline}\label{eq:ML200909261}
  \lim_{t\to 0+} \bCh^j(t\dirac)(a_0,\dots,a_j)\\
  =\frac{1}{j!}\int_{\tb M} \bomega_{\dirac} \wedge a_0 d a_1\dots\wedge d a_j,
  \qquad a_0,\ldots, a_j\in\bcC(M^\circ),
\end{multline}
resp. 
\begin{multline}\label{eq:ML200909262}
  \lim_{t\to 0+} \Ch^{j-1}(t\diracbdy)(a_0,\dots,a_{j-1})\\
  =\frac{1}{(j-1)!}\int_{\pl M} \omega_{\diracbdy} 
  \wedge a_0 d a_1\dots\wedge d a_{j-1},\qquad a_0,\ldots, a_{j-1}\in \cC^\infty(M).
\end{multline}
Furthermore, in view of \eqref{eq:ML200909263} we have for $k\geq \dim M -1$
\begin{equation}
\begin{split}
    \lim_{t\to 0+}  \bTslch^{k+1}_t(\dirac)(a_0,\dots,a_{k+1})&=0, \qquad  
    a_0,\ldots,a_{k+1}\in\bcC(M^\circ),\\
    \lim_{t\to 0+} \Tslch^{k}_t(\diracbdy)(a_0,\dots,a_k)&=0, \qquad  
    a_0,\ldots,a_k\in\cC^\infty(M).
\end{split}
\end{equation}

\sind{de Rham!current}
To interpret these limit results we briefly recall the relation between 
\deRham\ currents and relative cyclic cohomology classes over 
$(\bcC(M),\cC^\infty(M))$, \emph{cf.}~also \cite[Sec.~2.2]{LesMosPfl:RTK}.

Given a \deRham\ current $C$ of degree $j$ then $C$ defines naturally
a cochain $\widetilde C\in C^j(\bcC(M))$ by putting
$\widetilde C(a_0,\dots,a_j):= \frac{1}{j!}\langle C,a_0da_1\wedge\dots\wedge da_j\rangle$.
One has $b\widetilde C=0$ and $B\widetilde C=\widetilde{\pl C}$, where $\pl$
is the codifferential. Because of this identification we will from now on
omit the $\sim$ from the notation if no confusion is possible.

Given a closed \textup{b}-differential form $\go$ on $M$ of even degree. 
By $C_\go$ we denote the \deRham\ current $\int_{\tb M} \go\wedge -$.
There is a natural pullback $\iota^*\go$ at $\infty$
(\emph{cf.} Definition and Proposition \plref{defprop15}),
which is a closed even degree form on $\pl M$. We now find
\begin{equation}
\begin{split}
 \langle \pl C_{\go} , \alpha \rangle &= \int_{\tb M} \go \wedge d\alpha = \int_{\tb M} d (\go \wedge \alpha)
          = \int_{\pl M} \iota^*(\go \wedge \alpha) = \\
&= \int_{\pl M} \iota^*(\go) \wedge \iota^*(\alpha) = \langle  C_{\iota^*\go} ,  \iota^*(\alpha) \rangle .
\end{split}
\end{equation}
In view of Section \plref{SubSec:RelCycCoh} this means that the \emph{pair}
$(C_\go,C_{\iota^* \go})$ is a relative \deRham\ cycle or via the above mentioned
identification between \deRham\ currents and cochains that $(\widetilde{C_\go}, \widetilde{C_{\iota^* \go}})$ 
is a relative cocycle in the relative total complex Eq.~\eqref{eq:ML200909264}.

If $\go=\sum_{j \ge 0} \go_{2j}$ with closed \textup{b}-differential forms of degree $2j$
then the pair $(\go,\iota^*\go)$ still gives rise to a relative cocycle of degree 
$\dim M$ in the relative total complex.
These considerations certainly apply to the even degree forms 
$\bomega_\dirac, \omega_{\diracbdy}$ which satisfy
$\iota^*(\bomega)=\omega_{\diracbdy}$. 
The limit results \eqref{eq:ML200909261}, \eqref{eq:ML200909262}, 
\eqref{eq:ML200909263} can then be summarized as
\begin{equation}\label{eq:ML200909265b}
\lim_{t\to  0+}\big(\btch^k_t (\dirac),  \ch^{k-1}_t (\diracbdy) \big)
 =  \Big( \int_{\tb M} \bomega_{\dirac} \wedge \bullet , 
   \int_{\pl M} \omega_{\diracbdy}\wedge \bullet \Big).
\end{equation}
The limit on left is understood pointwise for each component of pure degree.

Finally we need to relate $\big( \btch^k_t(\dirac),\ch^{k-1}_t(\diracbdy) \big)$ to 
the Chern character of $[\dirac]\in K_m(M,\partial M)$. First recall from 
Eq.~\eqref{Eq:DefRelCycCoh} that 
$HP^\bullet \big( \cC^\infty (M),\cC^\infty (\partial M) \big)$
is naturally isomorphic to $HP^\bullet \big( \cJ^\infty (\partial M, M)\big)$.
Under this isomorphism, the class of the pair 
$\big( \btch^k_t(\dirac),\ch^{k-1}_t(\diracbdy) \big)$ is mapped to 
${\btch^k_t(\dirac)}_{|\cJ^\infty (\partial M, M)}$, just because elements 
of $\cJ^\infty (\partial M, M)$ vanish on $\partial M$. 
We note in passing that by \eqref{eq:ML20090219-4}
a smooth function $f$ on $M^\circ$ lies in $ \cJ^\infty (\partial M, M)$ 
iff in cylindrical coordinates one has for  all $l,R$ and every 
differential operator $P$ on $\partial M$
\begin{displaymath}
  \partial_x^l Df (x,p) = O(e^{Rx}) , \quad x \mapsto -\infty .
\end{displaymath}
In view of \eqref{eq:ML200909265}, the class of 
$\big( \btch^k_t(\dirac),\ch^{k-1}_t(\diracbdy) \big)$ in  
$HP^\bullet \big( \cJ^\infty (\partial M, M)\big) $ equals that of 
$\widetilde{C_{\bomega_\dirac}}$. As explained in Section \ref{Sec:RelConCheDirac}, 
$HP^\bullet \big( \cJ^\infty (\partial M, M)\big) $ is naturally isomorphic to 
$H^\textrm{dR}_\bullet (M,\pl M; \C)$. Under this isomorphism, 
$\widetilde{C_{\bomega_\dirac}}$ corresponds to the relative \deRham\ cycle
$\big( C_{\bomega_\dirac} , C_{\omega_{\diracbdy}} \big)$. Finally, note that under
the Poincar\'e duality isomorphism $H^\textrm{dR}_\bullet (M,\pl M; \C)\cong 
H^\bullet_\textrm{dR} (M\setminus \pl M;\C) $, the relative \deRham\ cycle
$\big( C_{\bomega_\dirac} , C_{\omega_{\diracbdy}} \big)$ is mapped onto the closed form 
$\bomega_\dirac$. This line of argument shows that the class of 
$\big( \btch^k_t(\dirac),\ch^{k-1}_t(\diracbdy) \big)$
depends only on the absolute \deRham\ cohomology class of the closed form
$\bomega_\dirac = c \cdot \hat{A} \big( \bnabla^2_g\big)$ on the open manifold 
$M\setminus \pl M$. The transgression formula in Chern--Weil theory shows that the 
(absolute) \deRham\ cohomology class of $\hat{A}\big( \nabla^2_g\big)$ is independent 
of the metric $g$ on $M^\circ$. Thus the \deRham\ class of $\bomega_\dirac$ 
equals that of $\omega_\dirac = c\cdot \hat{A}\big( \nabla^2_{g_0}\big)$ for any 
smooth metric $g_0$. Choosing $g_0$ to be smooth up to the boundary we infer from
Section \ref{Sec:RelConCheDirac} and Proposition \ref{p:ML200909292} that the class of  
$\big( \btch^k_t(\dirac),\ch^{k-1}_t(\diracbdy) \big)$ in 
$HP^\bullet \big( \cC^\infty (M),\cC^\infty (\pl M) \big) \cong 
HP^\bullet \big( \cJ^\infty (\partial M, M)\big) $ equals that of the
{\CoChch} of $[\dirac]$.
\end{proof}

\subsection{The large time limit and higher $\eta$--invariants}

Let us now assume that the boundary Dirac $\diracbdy$ is invertible.
In view of Proposition \plref{p:estimate-tChern-infty} and
Proposition \plref{p:ShortTimeEst} we can now, for $k\ge\dim M$,
form the transgressed cochain
\begin{equation}
	\bTslch_\infty^k(\dirac)(a_0,...,a_k)=\int_0^\infty \bslch^{k} (s\dirac,
 \dirac)(a_0,\ldots,a_k) \, ds, 
\end{equation}
for $a_0,\ldots,a_k\in \bcC(M^\circ)$.
In view of Eq.~\eqref{eq:ML200909024} we arrive at 
\begin{align}
   B &\bTslch_\infty^{k+1}(\dirac)(a_0,...,a_k)\nonumber\\
          &= \sum_{j=0}^k (-1)^{j+1} \int_0^\infty s^{k+1}\; \blangle [\dirac,a_0],\ldots ,[\dirac,a_j],\dirac,
                   [\dirac,a_{j+1}],\ldots ,[\dirac,a_k]\rangle\, ds\displaybreak[1]\nonumber \\
          &=\sum_{j=0}^k(-1)^{j+1} \int_0^\infty s^{k+1} \int_{\Delta_{k+1}}
        \bStr_q\bigl([\dirac,a_0]e^{-\sigma_0 s^2\dirac^2} \ldots \\
          &\quad [\dirac,a_j] e^{-\sigma_j s^2\dirac^2}
            \dirac e^{-\sigma_{j+1}s^2\dirac^2}
            \ldots [\dirac,a_k]e^{-\sigma_{k+1}s^2\dirac^2}\bigr)\, d\sigma \,ds.\nonumber
\end{align}

Together with Proposition \plref{p:ML20090928} we have proved the analogue of \cite[Thm.~1]{ConMos:TCC} in the relative setting:

\begin{theorem}\label{t:ML20081215} Let $k\ge\dim M$ be of
the same parity as $q$ and assume that $\diracbdy$ is invertible.
Then the pair of retracted cochains 
$\big(\bch^k_t (\dirac),\ch^{k+1}_t (\diracbdy) \big)$, $t>0$, has a limit as $t\to\infty$.
For $k=2l$ even we have
\begin{equation}\label{eq:ML20081215-1}
\begin{split}
    \bch_\infty^k(\dirac)&=\sum_{j=0}^l \kappa^{2j}(\dirac)+B\bTslch_\infty^{k+1}(\dirac)\, ,\\
         \ch_\infty^{k+1}(\diracbdy)&=          B\Tslch_\infty^{k+2}(\diracbdy)\, .
\end{split}
\end{equation}
If $k=2l+1$ is odd then
\begin{equation}\label{eq:ML20081215-2}
\begin{split}
    \bch_\infty^k(\dirac)&=B\bTslch_\infty^{k+1}(\dirac)\, ,\\
         \ch_\infty^{k+1}(\diracbdy)&=          B\Tslch_\infty^{k+2}(\diracbdy)\, . 
\end{split}
\end{equation}
\end{theorem}



\section{Relative pairing formul\ae\, and geometric consequences} \label{s: geom pairing}
\sind{pairing|(}

Let us briefly recall some facts from the theory of
boundary value problems for Dirac operators \cite{BooWoj:EBP}.
Given a Dirac operator $\dirac$ acting on sections
of the bundle $W$ on a compact riemannian manifold with boundary
$(M,g)$. In contrast to the rest of the paper $g$ is
a "true" riemannian metric, smooth and non-degenerate up to the boundary,
and not a \textup{b}-metric. We assume that all structures
are product near the boundary, that is there is a collar
$U=[0,\eps)\times\pl M$ of the boundary such
that $g\rest{U}=dx^2\oplus g\rest{\pl M}$ is a product metric.
In particular the formul\ae\, of Remark \plref{r:DiracCylinderFormulas}
hold.

We assume furthermore that we are in the even situation. That is,  
$\dirac$ is odd with respect to a $\Z_2$-grading. 
Then in a collar of the boundary $\dirac$ takes the form
\begin{equation}\label{eq:ML20090218-1}
    D=\left[\begin{matrix}  0 & D^-\\
                             D^+ & 0
             \end{matrix}\right]
      = \left[\begin{matrix}  0 & -\frac{d}{dx}+A^+\\
                             \frac{d}{dx}+A^+ & 0
             \end{matrix}\right]
      =\sfc(dx)\frac{d}{dx} +\diracbdy.
\end{equation}
In the matrix notation we have identified $W^+$ and $W^-$ via
$\sfc(dx)$ and put $A^+:=\bigl(\sfc(dx)^{-1}\diracbdy\bigl)\rest{W^+}$. 
$A^+$ is a first order self-adjoint elliptic differential operator. 

Let $P\in\pdo^0(\partial M;W^+)$
be a pseudodifferential projection with $P-1_{[0,\infty)}(A^+)$
of order $-1$. Then we denote by $\dirac_P$ the operator
$\dirac$ acting on the {\domain}
$\bigsetdef{u\in L^2_1(M;W^+)}{P(u\rest{\partial M})=0}$.

$\dirac_P$ is a Fredholm operator. The Agranovich--Dynin formula 
\cite[Prop. 21.4]{BooWoj:EBP}
\begin{equation}\label{eq:AgrDynin}
 \ind \dirac^+_P - \ind \dirac^+_Q=\ind (P,Q)=:
 \ind\bigl(P:\im Q\longrightarrow \im P\bigr)
\end{equation}
expresses the difference of two such indices
in terms of the relative index of the two projections $P,Q$.

Choosing for $P$ the positive spectral projection 
$P_+(A^+)=1_{[0,\infty)}(A^+)$ of $A^+$ we obtain the {\APS}  index
\begin{equation}\label{eq:APSindex}
\indAPS \dirac^+= \ind \dirac^+_{P_+(A^+)}.
\end{equation}

We shortly comment on the relative index introduced above 
(\emph{cf.}~\cite{AvrSeiSim:IPP}, \cite[Sec.~15]{BooWoj:EBP}, 
\cite[Sec.~3]{BruLes:BVPI}). Two orthogonal projections $P,Q$ in a Hilbert 
space are said to form a Fredholm pair if 
$PQ:\im Q\longrightarrow \im P$ is a Fredholm operator. The index of this 
Fredholm operator is then called the \emph{relative index},\sind{relative!index} $\ind(P,Q)$, 
of $P$ and $Q$. It is easy to see that $P,Q$ form a Fredholm pair if the 
difference $P-Q$ is a compact operator. Furthermore, the relative index 
is additive
\begin{equation}\label{eq:rel-index-additive}
\ind (P,R)=\ind (P,Q)+\ind (Q,R)
\end{equation}
if $P-Q$ or $P-R$ is compact
\cite[Prop. 15.15]{BooWoj:EBP}, \cite[Thm.~3.4]{AvrSeiSim:IPP}.
In general, just assuming that all three pairs
$(P,Q), (Q,R)$ and $(Q,R)$ are Fredholm is not sufficient for \eqref{eq:rel-index-additive} 
to hold.

\sind{spectral flow}
Sometimes the spectral flow of a path of self-adjoint operators can
be expressed as a relative index. The following proposition
is a special case of \cite[Thm.~3.6]{Les:USF}:

\begin{proposition}\label{p:SF-compact-perturbation} Let $T_s=T_0+\widetilde T_s, 0\le s\le 1,$ 
be a path of self-adjoint Fredholm operators in the Hilbert space $H$. Assume
that $\widetilde T_s$ is a continuous family of bounded $T_0$-compact
operators. Then the spectral flow of $(T_s)_{0\le s\le 1}$ is given by
\[
\SF (T_s)_{0\le s\le 1}=-\ind (P_+(T_1),P_+(T_0)),
\]
where $P_+(T_s):= 1_{[0,\infty)}(T_s)$.
\end{proposition}
\begin{remark} If $T_0$ is bounded then the condition of $T_0$-compactness
just means that the $T_s$ are compact operators. If $T_0$ is unbounded with
compact resolvent then any bounded operator is automatically $T_0$-compact.
The second case is the one of relevance for us.

In \cite[Thm.~3.6]{Les:USF} Proposition \plref{p:SF-compact-perturbation}
is proved for Riesz continuous paths of unbounded Fredholm operators. Since
$s\mapsto \widetilde T_s$ is continuous the map $s\mapsto T_0+\widetilde T_s$
is automatically Riesz continuous \cite[Prop. 3.2]{Les:USF}.

Note that our sign convention for the relative index differs from that of \emph{loc. cit.}
Therefore our formulation of Proposition \ref{p:SF-compact-perturbation} differs from
\cite[Thm.~3.6]{Les:USF} by a sign, too.
\end{remark}

\begin{proposition}\label{p:SFRelAPS}
Let $(\dirac_s)_{0\le s\le 1}$ be a smooth family of self-adjoint
$\Z_2$-graded \emph{(\cf~\eqref{eq:ML20090218-1})} Dirac operators on a compact
riemannian manifold with boundary. We assume that $\dirac_s$ is in product 
form near the boundary and that $\dirac_s=\dirac_0+\Phi_s$ with a bundle 
endomorphism $\Phi_s\in\Gamma^\infty(M;\End V)$. Then
\begin{equation}
\indAPS \dirac^+_1-\indAPS \dirac^+_0 =-\SF (A_s^+)_{0\le s\le 1}.
\end{equation}
Here, as explained above, $A_s^+=(\sfc(dx)^{-1}\diracbdyvar{s})\rest{W^+}$.
\end{proposition}
\begin{proof} The family $s\mapsto \dirac_{s,\APS}$ is not necessarily
continuous. The reason is that if eigenvalues of $A^+_s$
cross $0$ the family $P_+(A^+_s)$ of {\APS} projections jump.

However, since $A^+_s-A^+_0$ is $0^{\textup{th}}$ order, the corresponding {\APS} 
projections $P_+(A^+_s)$ all have the same leading symbol and hence 
$P_+(A^+_s)-P_+(A^+_{s'})$ is compact for all $s,s'\in [0,1]$.

Hence we can consider the family $\dirac_{s,P_+(A^+_0)},\; 0\le s\le 1$.
Now the boundary condition is fixed and thus $s\mapsto \dirac_{s,P_+(A^+_0)}$
is a graph continuous family of Fredholm operators
\cite{Nic:MIS,BooFur:MIF,BooLesZhu:CPD}. Therefore, its index is independent of $s$.

Applying the Agranovich--Dynin formula \eqref{eq:AgrDynin} and Proposition
\plref{p:SF-compact-perturbation} we find
\begin{align}
\indAPS(\dirac^+_1)&= \ind \dirac^+_{1,P_+(A^+_0)}+\ind \bigl(P_+(A^+_1),P_+(A^+_0)\bigr)\nonumber \\
		      &=  \ind \dirac^+_{0,P_+(A^+_0)}+\ind \bigl(P_+(A^+_1),P_+(A^+_0)\bigr)\\
		      &=\indAPS \dirac^+_0 -\SF(A_s^+)_{0\le s\le 1}.\qedhere
		      \end{align}
		      \end{proof}

Recall that a smooth idempotent $p:M\longrightarrow \Mat_N(\C)$ corresponds to a smooth
vector bundle $E\simeq \im p$ and using the Grassmann
connection the twisted Dirac operator $\dirac^E$ equals $p(\dirac\otimes \id_N)p$.
To simplify notation we will write $p\dirac p$ for $p(\dirac\otimes\id_N)p$ whenever confusions are unlikely. 

We would like to extend Proposition \plref{p:SFRelAPS} 
to families of twisted Dirac operators of the form $\dirac_s=p_s \dirac p_s$,
where $p_s:M\longrightarrow \Mat_N(\C)$ is a family of orthogonal projections. 

The difficulty is that not
only the leading symbol of $\dirac_s$ but even the Hilbert space of sections, on which the operator
acts, varies.

		      \begin{proposition}\label{p:SFRelAPSa}
Let $\dirac$ be a self-adjoint $\Z_2$-graded \emph{(\cf~\eqref{eq:ML20090218-1})}
Dirac operator on a compact riemannian manifold with boundary $M$. Let 
$p_s:M\longrightarrow \Mat_N(\C)$ be a smooth family of orthogonal projections. Assume
furthermore, that in a collar neighborhood $U=[0,\eps)\times \partial M$ of $\partial M$
we have ${p_s}\rest{U}=p_s^\partial$, i.e. ${p_s}\rest{U}$ is independent of the normal variable. Then
\begin{equation}
   \indAPS p_1\dirac^+ p_1-
   \indAPS p_0\dirac^+ p_0
   =-\SF(p_s A^+ p_s)_{0\le s\le 1}.
\end{equation}
\end{proposition}
\begin{proof}
By a standard trick often used in operator $K$-theory \cite[Prop. 4.3.3]{Bla:KTO} we may
choose a smooth path of unitaries $u:M\longrightarrow \Mat_N(\C)$ such that
$p_s=u_sp_0u_s^*, u_0=\id_N$. Furthermore, we may assume that 
$u\rest{[0,\eps)\times \partial M}=u^\partial$
is also independent of the normal variable. Then
$p_s\dirac^+ p_s= u_s\bigl( p_0 u_s^*\dirac^+ u_s p_0\bigr) u_s^*$ 
and
$\bigl(p_s \dirac^+ p_s\bigr)_{\APS}=
u_s\bigl( p_0 u_s^*\dirac^+ u_s p_0\bigr)_{\APS} u_s^*$.

Since $u_s^*\dirac^+ u_s=\dirac^++u_s^* \sfc(d u_s)$ 
Proposition \plref{p:SFRelAPS}  applies
to the family $p_0 u_s^*\dirac u_s p_0$. Since the spectral flow is invariant
under unitary conjugation we reach the conclusion.
\end{proof}

\begin{definition}\label{def:SF} In the sequel we will write somewhat more suggestively
and for brevity $\SF(p_{\cdot},\diracbdy)$ instead of $\SF(p_s A^+ p_s)_{0\le s\le 1}$.
\end{definition}

\begin{theorem}\label{t:RelativePairingKTheory}

Let $M$ be a compact manifold with boundary and $W$ a degree $q$
Clifford module on $M$. Let $g$ be a smooth riemannian metric
on $M$, $h$ a hermitian metric and $\nabla$ a unitary Clifford connection on $W$. Assume that
all structures are product near the boundary. Let $\dirac=\dirac(\nabla,g)$
be the Dirac operator.

Let
$[p,q,\gamma]\in K^0(M,\partial M)$ be a relative $K$-cycle . 
That is $p,q:M\longrightarrow \Mat_N(\C)$
are orthogonal projections and $\gamma:[0,1]\times\partial M\longrightarrow \Mat_N(\C)$
is a homotopy of orthogonal projections with $\gamma(0)=p^\partial, \gamma(1)=q^\partial$. Then
\begin{equation}\label{eq:RelativePairingKTheory}
\begin{split}
    \langle [\dirac],\, & [p,q,\gamma]\rangle\\
    =&-\indAPS p\dirac^+ p +\indAPS q\dirac^+ q
                       +\SF(\gamma,\diracbdy),\\
     =& \int_{M} \go_{\dirac(\nabla,g)}\wedge \bigl(\ch_\bullet(q)-\ch_\bullet(p)\bigr)
        - \int_{\pl M} \go_{\diracbdy(\nabla,g)}\wedge \Tslch_\bullet (h),
\end{split}
\end{equation}
in particular the right hand side of \eqref{eq:RelativePairingKTheory} 
depends only on the relative $K$-theory class $[p,q,\gamma]\in K^0(M,\partial M)$
and the degree $q$ Clifford module $W$. It is independent of $\nabla$ and $g$.

In case all structures are \textup{b}-structures, and $\dirac = \dirac (\bnabla, \bmet)$
is the \textup{b}-Dirac operator, then we still have
\begin{equation}
  \label{eq:brelpairing}
  \langle [\dirac],\,  [p,q,\gamma]\rangle =
  \int_{\tb M} \bomega_{\dirac}\wedge \bigl(\ch_\bullet(q)-\ch_\bullet(p)\bigr)
        - \int_{\pl M} \go_{\diracbdy}\wedge \Tslch_\bullet (h) .
\end{equation}
\end{theorem}
For the fact 
that relative $K$-cycles can be represented by triples $[p,q,\gamma]$ as above
we refer to  \cite[Thm.~5.4.2]{Bla:KTO}, \cite[Sec.~4.3]{HigRoe:AKH}, see
also \cite[Sec.~1.6]{LesMosPfl:RPC}.
\begin{proof} 
Denote the right hand side of \eqref{eq:RelativePairingKTheory} by $I(p,q,\gamma)$.
We first show that $I(p,q,\gamma)$ depends indeed only on the relative $K$-theory class of 
$(p,q,\gamma)$. By the stability of the Fredholm index we may assume that in a collar 
neighborhood of $\partial M$ the projections $p,q$ do not depend on the normal variable.

After stabilization we need to show the homotopy invariance of $I(p,q,\gamma)$.
Now consider a homotopy $(p_t,q_t,\gamma_t)$ of relative $K$-cycles. 
Then by Proposition \plref{p:SFRelAPSa} we have (\emph{cf.}~Figure \ref{fig:SF}, page \pageref{fig:SF})
\begin{align*}
            &I(p_1,q_1,\gamma_1)-I(p_0,q_0,\gamma_0)\\
       =\,  & \hphantom{-} \indAPS(p_1 \dirac^+ p_1)-\indAPS( q_1\dirac^+ q_1)- 
                 \SF\bigl(\gamma_1(\cdot),\diracbdy\bigr)\\
            & - \indAPS(p_0 \dirac^+ p_0)+\indAPS( q_0\dirac^+ q_0)
                + \SF\bigl(\gamma_0(\cdot),\diracbdy\bigr)\\
       =\,  &   
                - \SF\bigl(\gamma_{\cdot}(0),\diracbdy\bigr)
               + \SF\bigl(\gamma_{\cdot}(1),\diracbdy\bigr)
                - \SF\bigl(\gamma_1(\cdot),\diracbdy\bigr) + \SF\bigl(\gamma_0(\cdot),\diracbdy\bigr)\\
      =\,   &0,
\end{align*}
by the homotopy invariance of the Spectral Flow.

So the l.h.s \emph{and} the r.h.s. of \eqref{eq:RelativePairingKTheory} depend only
on the relative $K$-theory class of $[p,q,\gamma]$. By \excision\  in $K$-theory
(it can of course be shown in an elementary way by exploiting Swan's Theorem) every relative
$K$-theory class can even be represented by a triple $(p,q,p\rest{\pl M})$ such that
$p\rest{[0,\eps)\times \pl M}=q\rest{[0,\eps)\times \pl M}$ and hence 
$\gamma(s)=p\rest{\pl M}$ is constant.

\FigSF
Then the twisted version of the {\APS} Index Theorem gives
\begin{equation} 
\label{relAPS}
\begin{split}
   \indAPS q \dirac^+ q - \indAPS   p \dirac^+ p= 
   \int_M \omega_\dirac \wedge \big( \ch_\bullet (q) -  \ch_\bullet (p) \big), 
\end{split}
\end{equation}
where $\go_\dirac$ denotes the local index density of $\dirac$. Note that
since the tangential operators of $ p \dirac^+ p$
and of  $q \dirac^+ q$ coincide the $\eta$-terms cancel.

As outlined in Section \ref{Sec:RelConCheDirac} (\emph{cf.}~also the proof of Theorem
\plref{t: CC-character}) the {\CoChch} of $[\dirac]$
in $HP^\bullet \bigl(\cJ^\infty (\pl M , M )\bigr) \simeq H_\bullet^{\rm dR}
(M  \setminus \partial M ; \C)$ is represented by $\int_M \go_\dirac$. 
By construction, the form $\ch_\bullet(q)-\ch_\bullet(p)$ is compactly supported 
in $M\setminus \pl M$.
Thus the right hand side of \eqref{relAPS} equals 
the pairing $\langle [D], [p,q,p\rest{\pl M}]\rangle$ and the first equality in 
\eqref{eq:RelativePairingKTheory} is proved. 
\sind{pairing}

To prove the second equality in \eqref{eq:RelativePairingKTheory} 
we note that it represents the Poincar{\'e} duality pairing between
the \deRham\ cohomology class of $\go_{\dirac(\nabla,g)}$ 
(note $\iota^*\go_\dirac=\go_{\diracbdy}$)
and the relative \deRham\ cohomology
class of the pair of forms $(\ch_\bullet(q)-\ch_\bullet(p),\Tslch_\bullet(h))$. 
Hence it depends only on the class $[p,q,h]\in K^0 (M,\partial M)$ and on $[\dirac]$.
In the situation above where $p$ and $q$ coincide in a collar of the boundary it
equals $\langle [\dirac],[p,q,\gamma]\rangle$ and hence by homotopy invariance
the claim is proved in general up to Eq.~\eqref{eq:brelpairing}.

For the proof of Eq.~\eqref{eq:brelpairing} note first that for a closed even 
\textup{b}-differential form  $\omega$ the map (\emph{cf.}~Definition and Proposition \plref{defprop15})
\begin{displaymath}
  \begin{split}
    \Omega^k (M) \oplus \Omega^{k-1} (\partial M)  \rightarrow \C, \quad 
    (\eta,\tau)\mapsto\int_{\tb M}\omega \wedge\eta - 
    \int_{\partial M}\iota^*\omega \wedge \tau
  \end{split}
\end{displaymath}
descends naturally to a linear form on $H^k_\textrm{dR} (M,\pl M ; \C)$. Hence
the right hand side of  Eq.~\eqref{eq:brelpairing} is well-defined and depends only on the 
class of $[p,q,h] \in K^0 (M,\pl M)$. As before we may therefore specialize to
$(p,q,p_{|\pl M})$ such that 
$p_{|[0,\varepsilon ) \times \pl M}=q_{|[0,\varepsilon ) \times \pl M} $.
Then $\ch_\bullet(q)-\ch_\bullet(p)$ 
has compact support in $M\setminus \pl M$ and the remaining claim follows from 
Theorem \ref{t: CC-character} (4).  
\end{proof}

We now proceed to
express the pairing between relative K-theory classes and the fundamental
relative $K$-homology class in cohomological terms. We assume here that 
we are in the \textup{b}-setting. 
\sind{pairing}

Recall that a relative K-theory class in
$K^0(M,\pl M)$ is represented by a pair of bundles $(E,F)$ over $M$
whose restrictions $E_\pl,F_\pl$ to $\pl M$ are related by a homotopy $h$.
We will explicitly write the formul\ae\, in the even dimensional case and 
only point out where the odd dimensional case is different.

The Chern character of $[E,F,h] \in K^0 (M, \partial M)$ is then represented by
the relative cyclic homology class 
\begin{equation}
\begin{split}
  \ch_\bullet  \big( [E,F,h]  \big) \,= \,
  \Big(
    \ch_\bullet(p_F) - \ch_\bullet (p_E) \, , \, - \Tslch_\bullet (h)   \Big),
\end{split}
\end{equation}
\emph{cf.}~Eq.~\eqref{eq:ChernCharEven}.
 
By Theorem \plref{t: CC-character} we have for any $t > 0$  
\begin{equation}  \label{eq: McKean--Singer formula}
\begin{split}
\big\langle& [\dirac], \, [E,F,h] \big\rangle 
\,=\,\big\langle \big( \bch_t^n (\dirac), \ch_t^{n+1} (\diracbdy) \big),\, 
 \ch_\bullet([E,F,h])\big\rangle \\
 =\, & \big\langle \sum_{j \geq 0}  \bCh^{n-2j} (t\dirac) + B\, \tb\Tslch_t^{n+1} (\dirac), 
  \,  \ch_\bullet (p_F) -\ch_\bullet (p_E)  \big\rangle  \\
  & -  \big\langle  \sum_{j \geq 0}  \Ch^{n-2j+1} (t\dirac_\partial) +
  B \Tslch_t^{n+2} (\dirac_\partial) , \, \Tslch_\bullet (h)   \big\rangle.
\end{split}
\end{equation}
Letting $t \searrow 0$ yields, again by Theorem \plref{t: CC-character}, the local form of the pairing:
\begin{equation} \label{eq: local pairing even}
\begin{split}
\big\langle& [\dirac], \, [E,F,h] \big\rangle \\
& = \int_{\tb M} \go_\dirac \wedge \big( \ch_\bullet (p_F) - \ch_\bullet (p_E) \big) 
- \int_{\partial M} \go_{\diracbdy}(\partial M) \wedge \Tslch_\bullet (h) .
\end{split}
\end{equation}
If $\diracbdy$ is invertible then, at the opposite end, letting $t \nearrow \infty$ gives in view of Theorem   
\plref{t:ML20081215}
\begin{equation} \label{eq: stretched pairing even}
\begin{split}
\big\langle [\dirac], \, [E,F,h] &\big\rangle \, =
\big\langle \sum_{0\leq k \leq \ell}  \kappa^{2k} (\dirac)
 + B\tb \Tslch_\infty^{n+1} (\dirac) , 
\,  \ch_\bullet(p_F)  - \ch_\bullet (p_E)  \big\rangle \\
&\qquad- \big\langle B\Tslch_\infty^{n+2} (\dirac_\partial) ,   \Tslch_\bullet (h) \big\rangle,
\end{split}
\end{equation}
where $2\ell=n$.

By equating the above two limit expressions \eqref{eq: local pairing even}
and \eqref{eq: stretched pairing even}
one obtains the following identity:

\begin{corollary} \label{t: higher APS}
Let $n =2\ell \geq m $ and assume that $\diracbdy$ is invertible. 
Then 
\begin{equation*}
\begin{split}
  \big\langle \kappa^0(\dirac)  , \, & p_F - p_E \big\rangle + \\
  & + \, \sum_{1\leq k \leq \ell}  (-1)^k \frac{(2k)!}{k!} \, 
 \big\langle \kappa^{2k}(\dirac)  , \,  (p_F-\frac 12)\otimes p_F^{\otimes 2k} - 
 (p_E-\frac 12)\otimes p_E^{\otimes 2k}  \big\rangle \, \\
 =\,& \int_{\tb M} \go_\dirac  \wedge \big( \ch_\bullet (p_F) - \ch_\bullet (p_E) \big)  
 - \int_{\partial M} \go_{\diracbdy}\wedge  \Tslch_\bullet (h) \\
 & - (-1)^{n/2} \frac{n!}{(n/2)!}\, 
  \big\langle B \tb\Tslch_\infty^{n+1} (\dirac) , 
  (p_F-\frac 12)\otimes p_F^{\otimes n} - (p_E-\frac 12)\otimes p_E^{\otimes n}  \big\rangle\\
 & +\big\langle B \Tslch_\infty^{n+2} (\diracbdy),  \Tslch_{n+1}  (h) \big\rangle.
\end{split}
\end{equation*} 
\end{corollary}

The left hand side plays the role of a `higher' relative index, while
the right hand side contains local geometric terms and 
`higher' eta cochains. 

The pairing formula acquires a simpler form if one chooses
special representatives for the class $  [E,F,h] $. For example,
one can always assume that
 $E_\pl = F_\pl$, in which case one obtains
\begin{equation}  \label{eq: stretched pairing even sp}
\begin{split}
\big\langle [\dirac], \, [E,F,h_0] &\big\rangle \, =
\big\langle \sum_{0\leq k \leq \ell} \kappa^{2k} (\dirac)
 + B\tb \Tslch_\infty^{n+1} (\dirac) , 
\,  \ch_\bullet(p_F)  \big\rangle \\
&\qquad- \big\langle \sum_{0\leq k \leq \ell} \kappa^{2k} (\dirac)
 + B\tb \Tslch_\infty^{n+1} (\dirac) , 
\,  \ch_\bullet(p_E)  \big\rangle. 
\end{split}
\end{equation}
Specializing even more, one can assume $F = \C^N$. Then 
the pairing formula becomes
\begin{equation} \label{eq: stretched pairing even xsp}
\begin{split}
\big\langle [\dirac], \, [E,\C^N,h_0] \big\rangle \, = \,
& - \big\langle \sum_{0\leq k \leq \ell} \kappa^{2k} (\dirac)
 + \, B\tb \Tslch_\infty^{n+1} (\dirac) , 
\,  \ch_\bullet(p_E)  \big\rangle \\
&+\,  N ( \dim \Ker \dirac^+ -   \dim \Ker \dirac^-) . 
\end{split}
\end{equation}
On the other hand, applying Theorem \plref{t:RelativePairingKTheory} 
\begin{equation*}
  \big\langle [\dirac], \, [E,\C^N,h_0] \big\rangle \, = 
  \, -\indAPS (p_ED^+ p_E) + N \indAPS D^+ ,
\end{equation*}
one obtains an index formula for the $b$-Dirac operator which is the direct 
analogue of Eq.~(3.4) in \cite{ConMos:TCC}:

\begin{corollary} \label{t: b-index formula}
Let $E$ be a vector bundle on $M$
whose restriction to $\pl M$ is trivial and assume $\diracbdy$ to be invertible.
Then for any $n =2\ell \geq m $
\begin{equation} \label{eq: b-index formula}
\begin{split}
\indAPS (p_E D^+ p_E) = \big\langle  \sum_{0\leq k \leq \ell} \kappa^{2k}
(\dirac)
 + \, B\tb \Tslch_\infty^{n+1} (\dirac) ,
\,  \ch_\bullet(p_E)  \big\rangle  .
\end{split}
\end{equation}
\end{corollary}
The expression $\indAPS (p_E D^+ p_E)$ is to be understood as follows: if 
$p_E^\partial \dirac_\partial p_E^\partial$ is invertible, then it is the Fredholm index of 
$p_E\dirac^+p_E $. If  $p_E\dirac^+p_E $ is not Fredholm, then chose a metric $\tilde g$ 
smooth up to the boundary and construct on the Clifford module of $\dirac$ the Dirac 
operator $\tilde \dirac$ to the riemannian metric $\tilde g$. Then, by 
Theorem.~\ref{t: CC-character} the \CoChch s of $\tilde \dirac$
and $\dirac$ coincide and thus
\begin{displaymath}
  \big\langle [\dirac], [E,\C^N,h_0] \big\rangle = 
  - \indAPS p_E \tilde \dirac^+ p_E + N \indAPS \tilde\dirac^+ .  
\end{displaymath}


As a by-product of the above considerations, we can now establish
the following generalization of the
Atiyah-Patodi-Singer odd-index theorem for trivialized flat 
bundles (comp.~\cite[Prop.~6.2, Eq.~(6.3)]{APS:SARIII}. 
An analogue for even dimensional manifolds has been 
subsequently established by Z. Xie \cite{Xie:RIP}.

\begin{corollary} \label{t: gen APS flat}
Let $N$ be a closed odd dimensional spin manifold, and let
$E',F'$ be two vector bundles which are equivalent in
$K$-theory via a homotopy $h$. With $\dirac_{g'}$ denoting the Dirac
operator associated to a riemannian metric $g'$ on $N$, one has
\begin{equation} \label{eq: gen APS flat}
\begin{split}
\xi (\dirac_{g'}^{F'}) - \xi (\dirac_{g'}^{E'}) 
\, = \,  \int_{N} \hat{A} (\nabla_{g'}^2) \wedge
   \Tslch_\bullet (h) \, + \,  \SF (h , \dirac_{g'}) \, ,
\end{split}
\end{equation} 
or equivalently, 
\begin{equation} \label{eq: smooth eta}
\begin{split}
\int_0^1   \frac{1}{2} \frac{d}{dt} \big( \eta (p_{h(t)} \, \dirac_{g'} \, p_{h(t)}) \big) dt \, = \,
 \int_{N} \hat{A} (\nabla_{g'}^2) \wedge
   \Tslch_\bullet (h) \, ,  
\end{split}
\end{equation}
where $p_{h(t)})$ is the path of projections joining $E'$ and $F'$. 
\end{corollary}

\begin{proof}
This follows from equating the local pairing \eqref{eq:RelativePairingKTheory} 
and the relative {\APS}  index formula \eqref{Eq:Pairing2}
after the following modifications. We recall that 
Eq.~\eqref{eq:RelativePairingKTheory} holds in complete generality without invertibility 
hypothesis on $\dirac_\pl$. First by passing to a multiple one can assume
that $N=\pl M$. Then by adding a complement $G'$ we can replace $F'$ by a trivial
bundle. Then both $E'\oplus G'$ and $F'\oplus G'$ extend to $M$.
It remains to notice that both sides of the formula
 \eqref{eq: gen APS flat} are additive.

The alternative formulation \eqref{eq: smooth eta} follows immediately
from the known relation (see e.g. \cite[Lemma 3.4]{KirLes:EIM})
\begin{equation*}
\xi (\dirac_{g'}^{F'}) - \xi (\dirac_{g'}^{E'})  \, = \, \SF (h , \diracbdy) +
 \int_0^1 \frac{1}{2} \frac{d}{dt} \big( \eta (p_{h(t)} \, \dirac_{g'} \, p_{h(t)}) \big) dt .\qedhere
\end{equation*}
\end{proof}

In the odd dimensional case the pairing formul\ae\, are similar, except 
that the contribution from the kernel of $\dirac$ does not occur.
Let $(U,V,h)$ be a representative of an odd relative K-theory
class where $U,V:M\to U(N)$ are unitaries and $h$ is a homotopy
between $U_{\pl M}$ and $V_{\pl M}$.
Then 
\begin{equation}  \label{eq: stretched pairing odd}
\begin{split}
\big\langle [\dirac], \, [U,V,h] &\big\rangle \, =\,
\big\langle B \bTslch_\infty^{n+1} (\dirac) , 
\,  \ch_n (U)  - \ch_n (V)  \big\rangle \\
&\qquad- \big\langle B\Tslch_\infty^{n+2} (\diracbdy) ,   \Tslch_{n+1} (h) \big\rangle.
\end{split}
\end{equation}
Choosing a representative of the class with $U_{\pl M} = V_{\pl M}$, the above
formula simplifies to
\begin{equation}  \label{eq: stretched pairing odd sp}
\begin{split}
\big\langle [\dirac], \, [U,V,h_0] &\big\rangle \, = \,
\big\langle B \bTslch_\infty^{n+1} (\dirac) , 
\,  \ch_n (U)  -  \ch_n (V)  \big\rangle ,
\end{split}
\end{equation}
and if moreover one takes $V=\id$, it  reduces to
\begin{equation}  \label{eq: stretched pairing odd xsp}
\begin{split}
\big\langle [\dirac], \, [U,\id,h_0] &\big\rangle \, = \,
\big\langle B \bTslch_\infty^{n+1} (\dirac) ,
\,  \ch_n (U)   \big\rangle .
\end{split}
\end{equation}

Finally, the equality between the local form of the pairing \eqref{eq: local pairing even}
and the expression \eqref{eq: stretched pairing odd} gives the following odd 
analogue of Corollary \ref{t: higher APS}.

\begin{corollary} \label{t: higher APS odd}
Let $n  \geq m $, both odd and assume that $\diracbdy$ is invertible. Then
\begin{equation*}
\begin{split}
\big\langle B \bTslch_\infty^{n+1} (\dirac)&,\,  \ch_n(U)  \big\rangle -
\big\langle B\bTslch_\infty^{n+1} (\dirac),\,  \ch_n(V)  \big\rangle \, = \\ 
&= \int_{\tb M} \hat{A} (\bnabla^2_g)  \wedge \big( \ch_\bullet (U) - \ch_\bullet (V) \big)
- \int_{\partial M} \hat{A} (\nabla^2_{g'}) \wedge
   \Tslch_\bullet (h) \\
& \qquad 
- \big\langle B\Tslch_\infty^{n+2} (\diracbdy) ,   \Tslch_{n+1} (h) \big\rangle.
\end{split}
\end{equation*}
\end{corollary}

\section{Relation with the generalized \textup{APS}  pairing}
\label{s:conclude}


Wu~\cite{Wu:CCC} showed that the full cochain $\eta^\bullet (\diracbdy)$ has a finite radius of
convergence, proportional to the lowest eigenvalue of $\vert \diracbdy \vert$
(assumed to be invertible). 
Both Wu and Getzler~\cite{Get:CHA} proved, by different
methods, the following generalized Atiyah-Patodi-Singer index formula :
\begin{equation} \label{HAPS} 
 \indAPS  \dirac^E \, = \,
   \int_{\tb M} \hat{A} (\bnabla^2_g)  \wedge \ch_\bullet(p_E)  
\, +\,  \big\langle  \eta^\bullet (\diracbdy)\circ i^\ast, \, \ch_\bullet(p_E) \big\rangle ,
\end{equation}
for any vector bundle $E = \im p_E$ over $M$ whose restriction to the boundary
satisfies the {\it almost $\partial$-flatness} condition
$ \Vert [\diracbdy ,  r_\partial (p_E) ] \Vert < \lambda_1 (\vert \diracbdy \vert)$. 
Their result does provide a decoupled index pairing,
but only for those classes in $K^m (M, \partial M)$
which can be represented by pairs of almost $\partial$-flat bundles. 
Furthermore, if $(E, F , h)$ is  such a triple, on applying \eqref{HAPS} and
Theorem \plref{t:RelativePairingKTheory} one obtains  
\begin{equation} \label{eq: GWpair}
 \begin{split}
\langle [\dirac], [E, F, h] \rangle \,= &
\int_{\tb M} \hat{A}(\bnabla^2_g) \wedge \big(\ch_\bullet(p_F) - \ch_\bullet(p_E) \big)\\
&  + \big\langle \eta^\bullet (\diracbdy) , i^\ast\big( \ch_\bullet(p_F) - \ch_\bullet(p_E) \big)\big\rangle 
\,+ \,  \SF (h , \diracbdy) ,
\end{split}
\end{equation}
where $\SF(h,\diracbdy)$ is an abbreviation for $\SF\bigl( h(s) A^+ h(s)\bigr)_{0\le s\le 1},$
\emph{cf.}~also Eq.~\eqref{eq:ML20090218-1}. By \cite[Proof of Thm.~3.1]{Wu:CCC}, 
\begin{equation*} 
   (b+B) \eta^\bullet (\diracbdy) \, =\,-  \int_{\pl M} \hat{A}(\nabla^2_{g'}) \wedge -\, .
\end{equation*}
At the formal level 
\begin{equation} \label{eq: corrected}
 \begin{split}
&\big\langle \eta^\bullet (\diracbdy)\circ i^\ast, 
\ch_\bullet(p_F) - \ch_\bullet(p_E) \big\rangle 
 =  \big\langle \eta^\bullet (\diracbdy), (b+B)  \Tslch_\bullet (h) \big\rangle \\
& =  \big\langle  (b+B) \eta^\bullet (\diracbdy), \Tslch_\bullet (h) \big\rangle 
\, = \,- \int_{\pl M} \hat{A}(\nabla^2_{g'}) \wedge \Tslch_\bullet (h) .
\end{split}
\end{equation}
However, to ensure that the pairing $ \big\langle \eta^\bullet (\diracbdy), (b+B)  \Tslch_\bullet (h) \big\rangle$
makes sense one has to assume that $p_{h(t)}$ satisfy the same almost $\partial$-flatness
condition. Then 
$$
\Ker (p_{h(t)}  \diracbdy p_{h(t)}) = 0 ,
$$
 hence there is no spectral flow along
the path. 

Thus, the total eta cochain disappears and
\eqref{eq: corrected} together with \eqref{eq: GWpair} just lead to the known 
local pairing formula, \emph{cf.}~Eq.~\eqref{eq:RelativePairingKTheory},
\begin{equation*} 
\begin{split}
\langle [\dirac], &[E, F, h] \rangle \\
&= \int_{\tb M} \hat{A}(\bnabla^2_g) \wedge \big(\ch_\bullet(p_F) - \ch_\bullet(p_E) \big)
 \, - \,\int_{\pl M} \hat{A}(\nabla^2_{g'}) \wedge \Tslch_\bullet (h)  \, .
\end{split}
\sind{pairing|)}
\end{equation*}
 
The above considerations also show that the total eta cochain
necessarily has a finite radius of
convergence. If it was entire, then $\SF(h, D) = 0$ for any $h(t)$,
which is easy to disprove by a counterexample.






\backmatter
\bibliography{localbib}

\providecommand{\bysame}{\leavevmode\hbox to3em{\hrulefill}\thinspace}
\providecommand{\MR}{\relax\ifhmode\unskip\space\fi MR }
\providecommand{\MRhref}[2]{%
  \href{http://www.ams.org/mathscinet-getitem?mr=#1}{#2}
}
\providecommand{\href}[2]{#2}
\begin{thebibliography}{\textsc{BBWo93}}

\bibitem[\textsc{APS75}]{APS:SARI}
\textsc{M.~F. Atiyah}, \textsc{V.~K. Patodi}, and \textsc{I.~M. Singer},
  \emph{Spectral asymmetry and {R}iemannian geometry. {I}}, Math. Proc.
  Cambridge Philos. Soc. \textbf{77} (1975), 43--69. \MR{0397797 (53 \#1655a)}

\bibitem[\textsc{APS76}]{APS:SARIII}
\bysame, \emph{Spectral asymmetry and {R}iemannian geometry. {III}}, Math.
  Proc. Cambridge Philos. Soc. \textbf{79} (1976), no.~1, 71--99. \MR{0397799
  (53 \#1655c)}

\bibitem[\textsc{ASS94}]{AvrSeiSim:IPP}
\textsc{J.~Avron}, \textsc{R.~Seiler}, and \textsc{B.~Simon}, \emph{The index
  of a pair of projections}, J. Funct. Anal. \textbf{120} (1994), no.~1,
  220--237. \MR{1262254 (95b:47012)}

\bibitem[\textsc{BBFu98}]{BooFur:MIF}
\textsc{B.~Booss-Bavnbek} and \textsc{K.~Furutani}, \emph{The {M}aslov index: a
  functional analytical definition and the spectral flow formula}, Tokyo J.
  Math. \textbf{21} (1998), no.~1, 1--34. \MR{1630119 (99e:58172)}

\bibitem[\textsc{BBLZ09}]{BooLesZhu:CPD}
\textsc{B.~Boo{\ss}-Bavnbek}, \textsc{M.~Lesch}, and \textsc{C.~Zhu}, \emph{The
  {C}alder\'on projection: new definition and applications}, J. Geom. Phys.
  \textbf{59} (2009), no.~7, 784--826. \texttt{arXiv:0803.4160v1 [math.DG] ,
  DOI: 10.1016/j.geomphys.2009.03.012,
  http://dx.doi.org/10.1016/j.geomphys.2009.03.012}, \MR{2536846}

\bibitem[\textsc{BBWo93}]{BooWoj:EBP}
\textsc{B.~Boo{\ss}-Bavnbek} and \textsc{K.~P. Wojciechowski}, \emph{Elliptic
  boundary problems for {D}irac operators}, Mathematics: Theory \&
  Applications, Birkh\"auser Boston Inc., Boston, MA, 1993. \MR{1233386
  (94h:58168)}

\bibitem[\textsc{BDT89}]{BauDouTay:CRC}
\textsc{P.~Baum}, \textsc{R.~G. Douglas}, and \textsc{M.~E. Taylor},
  \emph{Cycles and relative cycles in analytic {$K$}-homology}, J. Differential
  Geom. \textbf{30} (1989), no.~3, 761--804. \MR{1021372 (91b:58244)}

\bibitem[\textsc{BGV92}]{BerGetVer:HKD}
\textsc{N.~Berline}, \textsc{E.~Getzler}, and \textsc{M.~Vergne}, \emph{Heat
  kernels and {D}irac operators}, Grundlehren der Mathematischen Wissenschaften
  [Fundamental Principles of Mathematical Sciences], vol. 298, Springer-Verlag,
  Berlin, 1992. \MR{1215720 (94e:58130)}

\bibitem[\textsc{Bla86}]{Bla:KTO}
\textsc{B.~Blackadar}, \emph{{$K$}-theory for operator algebras}, Mathematical
  Sciences Research Institute Publications, vol.~5, Springer-Verlag, New York,
  1986. \MR{859867 (88g:46082)}

\bibitem[\textsc{BlFo90}]{BloFox:APO}
\textsc{J.~Block} and \textsc{J.~Fox}, \emph{Asymptotic pseudodifferential
  operators and index theory}, Geometric and topological invariants of elliptic
  operators ({B}runswick, {ME}, 1988), Contemp. Math., vol. 105, Amer. Math.
  Soc., Providence, RI, 1990, pp.~1--32. \MR{1047274 (91e:58186)}

\bibitem[\textsc{BrLe01}]{BruLes:BVPI}
\textsc{J.~Br{\"u}ning} and \textsc{M.~Lesch}, \emph{On boundary value problems
  for {D}irac type operators. {I}. {R}egularity and self-adjointness}, J.
  Funct. Anal. \textbf{185} (2001), no.~1, 1--62. \texttt{arXiv:math/9905181
  [math.FA]}, \MR{1853751 (2002g:58034)}

\bibitem[\textsc{BrPf08}]{BraPfl:HAWFSS}
\textsc{J.-P. Brasselet} and \textsc{M.~J. Pflaum}, \emph{On the homology of
  algebras of {W}hitney functions over subanalytic sets}, Ann. of Math. (2)
  \textbf{167} (2008), no.~1, 1--52. \MR{2373151}

\bibitem[\textsc{CoMo90}]{ConMos:CCN}
\textsc{A.~Connes} and \textsc{H.~Moscovici}, \emph{Cyclic cohomology, the
  {N}ovikov conjecture and hyperbolic groups}, Topology \textbf{29} (1990),
  no.~3, 345--388. \MR{1066176 (92a:58137)}

\bibitem[\textsc{CoMo93}]{ConMos:TCC}
\bysame, \emph{Transgression and the {C}hern character of finite-dimensional
  {$K$}-cycles}, Comm. Math. Phys. \textbf{155} (1993), no.~1, 103--122.
  \MR{1228528 (95a:46091)}

\bibitem[\textsc{CoMo95}]{ConMos:LIF}
\textsc{A.~Connes} and \textsc{H.~Moscovici}, \emph{The local index formula in
  noncommutative geometry}, Geom. Funct. Anal. \textbf{5} (1995), no.~2,
  174--243. \MR{1334867 (96e:58149)}

\bibitem[\textsc{Con85}]{Con:NDG}
\textsc{A.~Connes}, \emph{Noncommutative differential geometry}, Inst. Hautes
  \'Etudes Sci. Publ. Math. \textbf{62} (1985), 257--360. \MR{823176
  (87i:58162)}

\bibitem[\textsc{Con88}]{Con:ECC}
\bysame, \emph{Entire cyclic cohomology of {B}anach algebras and characters of
  {$\theta$}-summable {F}redholm modules}, $K$-Theory \textbf{1} (1988), no.~6,
  519--548. \MR{MR953915 (90c:46094)}

\bibitem[\textsc{Con94}]{Con:NG}
\bysame, \emph{Noncommutative geometry}, Academic Press Inc., San Diego, CA,
  1994. \MR{1303779 (95j:46063)}

\bibitem[\textsc{CoSa96}]{ConSav:MFBFA}
\textsc{G.~M. Constantine} and \textsc{T.~H. Savits}, \emph{A multivariate
  {F}a\`a di {B}runo formula with applications}, Trans. Amer. Math. Soc.
  \textbf{348} (1996), no.~2, 503--520. \MR{1325915 (96g:05008)}

\bibitem[\textsc{CuQu93}]{CunQui:EPCC}
\textsc{J.~Cuntz} and \textsc{D.~Quillen}, \emph{On excision in periodic cyclic
  cohomology}, C. R. Acad. Sci. Paris S\'er. I Math. \textbf{317} (1993),
  no.~10, 917--922. \MR{1249360 (94i:19002)}

\bibitem[\textsc{CuQu94}]{CunQui:EPCCII}
\bysame, \emph{On excision in periodic cyclic cohomology. {II}. {T}he general
  case}, C. R. Acad. Sci. Paris S\'er. I Math. \textbf{318} (1994), no.~1,
  11--12. \MR{1260526 (94m:19001)}

\bibitem[\textsc{FuKe88}]{FulKen:RPG}
\textsc{S.~A. Fulling} and \textsc{G.~Kennedy}, \emph{The resolvent parametrix
  of the general elliptic linear differential operator: a closed form for the
  intrinsic symbol}, Trans. Amer. Math. Soc. \textbf{310} (1988), no.~2,
  583--617. \MR{973171 (90b:58260)}

\bibitem[\textsc{GBVF01}]{GraVarFig:ENG}
\textsc{J.~M. Gracia-Bond{\'{\i}}a}, \textsc{J.~C. V{\'a}rilly}, and
  \textsc{H.~Figueroa}, \emph{Elements of noncommutative geometry},
  Birkh\"auser Advanced Texts: Basler Lehrb\"ucher. [Birkh\"auser Advanced
  Texts: Basel Textbooks], Birkh\"auser Boston Inc., Boston, MA, 2001.
  \MR{1789831 (2001h:58038)}

\bibitem[\textsc{GeSz89}]{GetSze:CCT}
\textsc{E.~Getzler} and \textsc{A.~Szenes}, \emph{On the {C}hern character of a
  theta-summable {F}redholm module}, J. Funct. Anal. \textbf{84} (1989), no.~2,
  343--357. \MR{1001465 (91g:19007)}

\bibitem[\textsc{Get83}]{Get:POS}
\textsc{E.~Getzler}, \emph{Pseudodifferential operators on supermanifolds and
  the {A}tiyah-{S}inger index theorem}, Comm. Math. Phys. \textbf{92} (1983),
  no.~2, 163--178. \MR{728863 (86a:58104)}

\bibitem[\textsc{Get93a}]{Get:CHA}
\bysame, \emph{Cyclic homology and the {A}tiyah-{P}atodi-{S}inger index
  theorem}, Index theory and operator algebras ({B}oulder, {CO}, 1991),
  Contemp. Math., vol. 148, Amer. Math. Soc., Providence, RI, 1993, pp.~19--45.
  \MR{1228498 (94j:58162)}

\bibitem[\textsc{Get93b}]{Get:OCC}
\bysame, \emph{The odd {C}hern character in cyclic homology and spectral flow},
  Topology \textbf{32} (1993), no.~3, 489--507. \MR{1231957 (95c:46118)}

\bibitem[\textsc{Gil95}]{Gil:ITH}
\textsc{P.~B. Gilkey}, \emph{Invariance theory, the heat equation, and the
  {A}tiyah-{S}inger index theorem}, second ed., Studies in Advanced
  Mathematics, CRC Press, Boca Raton, FL, 1995. \MR{1396308 (98b:58156)}

\bibitem[\textsc{GrSe95}]{GruSee:WPP}
\textsc{G.~Grubb} and \textsc{R.~T. Seeley}, \emph{Weakly parametric
  pseudodifferential operators and {A}tiyah-{P}atodi-{S}inger boundary
  problems}, Invent. Math. \textbf{121} (1995), no.~3, 481--529. \MR{1353307
  (96k:58216)}

\bibitem[\textsc{HaSu78}]{HalSun:BIO}
\textsc{P.~R. Halmos} and \textsc{V.~S. Sunder}, \emph{Bounded integral
  operators on {$L\sp{2}$} spaces}, Ergebnisse der Mathematik und ihrer
  Grenzgebiete [Results in Mathematics and Related Areas], vol.~96,
  Springer-Verlag, Berlin, 1978. \MR{517709 (80g:47036)}

\bibitem[\textsc{HiRo00}]{HigRoe:AKH}
\textsc{N.~Higson} and \textsc{J.~Roe}, \emph{Analytic {$K$}-homology}, Oxford
  Mathematical Monographs, Oxford University Press, Oxford, 2000, Oxford
  Science Publications. \MR{1817560 (2002c:58036)}

\bibitem[\textsc{JLO88}]{JLO:QKT}
\textsc{A.~Jaffe}, \textsc{A.~Lesniewski}, and \textsc{K.~Osterwalder},
  \emph{Quantum {$K$}-theory. {I}. {T}he {C}hern character}, Comm. Math. Phys.
  \textbf{118} (1988), no.~1, 1--14. \MR{954672 (90a:58170)}

\bibitem[\textsc{KiLe04}]{KirLes:EIM}
\textsc{P.~Kirk} and \textsc{M.~Lesch}, \emph{The {$\eta$}-invariant, {M}aslov
  index, and spectral flow for {D}irac-type operators on manifolds with
  boundary}, Forum Math. \textbf{16} (2004), no.~4, 553--629.
  \texttt{arXiv:math/0012123 [math.DG]}, \MR{2044028 (2005b:58029)}

\bibitem[\textsc{Les99}]{Les:NRP}
\textsc{M.~Lesch}, \emph{On the noncommutative residue for pseudodifferential
  operators with log-polyhomogeneous symbols}, Ann. Global Anal. Geom.
  \textbf{17} (1999), no.~2, 151--187. \texttt{arXiv:dg-ga/9708010},
  \MR{1675408 (2000b:58050)}

\bibitem[\textsc{Les05}]{Les:USF}
\bysame, \emph{The uniqueness of the spectral flow on spaces of unbounded
  self-adjoint {F}redholm operators}, Spectral geometry of manifolds with
  boundary and decomposition of manifolds, Contemp. Math., vol. 366, Amer.
  Math. Soc., Providence, RI, 2005, pp.~193--224. \texttt{arXiv:math/0401411
  [math.FA]}, \MR{2114489 (2005m:58049)}

\bibitem[\textsc{LMP08}]{LesMosPfl:RTK}
\textsc{M.~Lesch}, \textsc{H.~Moscovici}, and \textsc{M.~J. Pflaum},
  \emph{Regularized traces and {$K$}-theory invariants of parametric
  pseudodifferential operators}, Traces in number theory, geometry and quantum
  fields, Aspects Math., E38, Friedr. Vieweg, Wiesbaden, 2008, pp.~161--177.
  \texttt{www.matthiaslesch.de publications}, \MR{2427595}

\bibitem[\textsc{LMP09}]{LesMosPfl:RPC}
\bysame, \emph{Relative pairing in cyclic cohomology and divisor flows}, J.
  K-Theory \textbf{3} (2009), no.~2, 359--407. \texttt{arXiv:math/0603500
  [math.KT]}, \MR{2496452}

\bibitem[\textsc{Loy05}]{Loy:DOB}
\textsc{P.~Loya}, \emph{Dirac operators, boundary value problems, and the
  {$b$}-calculus}, Spectral geometry of manifolds with boundary and
  decomposition of manifolds, Contemp. Math., vol. 366, Amer. Math. Soc.,
  Providence, RI, 2005, pp.~241--280. \MR{2114491 (2005k:58038)}

\bibitem[\textsc{Mal67}]{Mal:IDF}
\textsc{B.~Malgrange}, \emph{Ideals of differentiable functions}, Tata
  Institute of Fundamental Research Studies in Mathematics, No. 3, Tata
  Institute of Fundamental Research, Bombay, 1967. \MR{0212575 (35 \#3446)}

\bibitem[\textsc{Mel93}]{Mel:APSIT}
\textsc{R.~B. Melrose}, \emph{The {A}tiyah-{P}atodi-{S}inger index theorem},
  Research Notes in Mathematics, vol.~4, A K Peters Ltd., Wellesley, MA, 1993.
  \MR{1348401 (96g:58180)}

\bibitem[\textsc{MoPi11}]{MorPia:RPA}
\textsc{H.~Moriyoshi} and \textsc{P.~Piazza}, \emph{Relative pairings and the
  {A}tiyah--{P}atodi--{S}inger index formula for the {G}odbillon--{V}ey
  cocycle}, Noncommutative Geometry and Global Analysis, Contemp. Math., vol.
  546, Amer. Math. Soc., Providence, RI, 2011, pp.~225--247.
  \texttt{arXiv:math/0907.0173}

\bibitem[\textsc{M{\"u}l94}]{Mul:EIM}
\textsc{W.~M{\"u}ller}, \emph{Eta invariants and manifolds with boundary}, J.
  Differential Geom. \textbf{40} (1994), no.~2, 311--377. \MR{1293657
  (96c:58165)}

\bibitem[\textsc{Nic95}]{Nic:MIS}
\textsc{L.~I. Nicolaescu}, \emph{The {M}aslov index, the spectral flow, and
  decompositions of manifolds}, Duke Math. J. \textbf{80} (1995), no.~2,
  485--533. \MR{1369400 (96k:58208)}

\bibitem[\textsc{Pfl98}]{Pfl:NSRM}
\textsc{M.~J. Pflaum}, \emph{The normal symbol on {R}iemannian manifolds}, New
  York J. Math. \textbf{4} (1998), 97--125 (electronic). \MR{1640055
  (99e:58181)}

\bibitem[\textsc{Shu01}]{Shu:POS}
\textsc{M.~A. Shubin}, \emph{Pseudodifferential operators and spectral theory},
  second ed., Springer-Verlag, Berlin, 2001, Translated from the 1978 Russian
  original by Stig I. Andersson. \MR{1852334 (2002d:47073)}

\bibitem[\textsc{Tay96}]{Tay:PDEII}
\textsc{M.~E. Taylor}, \emph{Partial differential equations. {II}}, Applied
  Mathematical Sciences, vol. 116, Springer-Verlag, New York, 1996, Qualitative
  studies of linear equations. \MR{1395149 (98b:35003)}

\bibitem[\textsc{Tou72}]{Tou:IFD}
\textsc{J.-C. Tougeron}, \emph{Id\'eaux de fonctions diff\'erentiables},
  Springer-Verlag, Berlin, 1972, Ergebnisse der Mathematik und ihrer
  Grenzgebiete, Band 71. \MR{0440598 (55 \#13472)}

\bibitem[\textsc{Wid79}]{Wid:STC}
\textsc{H.~Widom}, \emph{Szeg{\H o}'s theorem and a complete symbolic calculus
  for pseudodifferential operators}, Seminar on {S}ingularities of {S}olutions
  of {L}inear {P}artial {D}ifferential {E}quations ({I}nst. {A}dv. {S}tudy,
  {P}rinceton, {N}.{J}., 1977/78), Ann. of Math. Stud., vol.~91, Princeton
  Univ. Press, Princeton, N.J., 1979, pp.~261--283. \MR{547022 (81b:58043)}

\bibitem[\textsc{Wid80}]{Wid:CSCPO}
\bysame, \emph{A complete symbolic calculus for pseudodifferential operators},
  Bull. Sci. Math. (2) \textbf{104} (1980), no.~1, 19--63. \MR{560744
  (81m:58078)}

\bibitem[\textsc{Wu93}]{Wu:CCC}
\textsc{F.~Wu}, \emph{The {C}hern-{C}onnes character for the {D}irac operator
  on manifolds with boundary}, $K$-Theory \textbf{7} (1993), no.~2, 145--174.
  \MR{1235286 (95f:58076)}

\bibitem[\textsc{Xie11}]{Xie:RIP}
\textsc{Z.~Xie}, \emph{Relative index pairing and odd index theorem for even
  dimensional manifolds}, J. Funct. Anal. \textbf{260} (2011), no.~7,
  2064--2085. \MR{2756150}

\end{thebibliography}
\bibliographystyle{amsalpha-lmp}

\sind{finite part|see{partie finie}}
\sind{integral!partie finie|see{partie finie}}

\Printindex{ind}{Subject Index}

\indexcomment{
Except a few standard notations, all symbols are explained at their
first occurence. We recall a few very standard notations and then we 
provide an index to the used symbols.

\par\bigskip\noindent
\begin{tabular}{rl}
$\N, \Z, \R, \C$  & Natural (including $0$), integer, real and complex numbers  \\
$\R_+$      & Nonnegative real numbers $x\ge 0$                                 \\
$\Z_+$      & Synonym for $\N$                                                  \\ 
$\cC(\ldots), \cC^\infty(\ldots)$ & Continuous resp. smooth functions           \\
$\cC_0(\ldots), \cC^\infty_0(\ldots)$ & Ditto, vanishing at infinity (on locally compact space)           \\
$\cC_c(\ldots), \cC^\infty_c(\ldots)$ & Ditto, compactly supported           \\
$\Gamma^\infty(M;E)$   &  Smooth sections of the vector bundle $E$ over $M$,
                          where $\Gamma^\infty_c, \Gamma^\infty_0$\\
                       &  have the analogous meaning as for $\cC^\infty$.\\
$\GL_N(\cA)$           &  Invertible $N\times N$ matrices with entries in $\cA$ \\
$\cH$                  &  Generic name for a Hilbert space \\
$\sL(\cH)$             &  Algebra of bounded linear operators on the Hilbert
space $\cH$\\ 
$L^2(M;E)$             &  Square--integrable sections of the hermitian vector bundle $E$  \\
$\Mat_N(\cA)$          &  $N\times N$--matrix algebra over the algebra $\cA$         \\
$\spec(T)$             &  Spectrum of the linear operator $T$  \\
$\specess(T)$          &  Essential spectrum of the linear operator $T$ \\
$\supp(f)$             &  Support of the distribution(al section) $f$ 
\end{tabular} 
}

\Printindex{not}{Notation Index}

\end{document}